\documentclass[11pt,reqno]{amsart}
\usepackage{amssymb}
\usepackage{amsfonts}
\usepackage{mathrsfs}
\usepackage{amsmath}
\usepackage{graphicx}
\usepackage{hyperref}
\usepackage{float}
\usepackage{epstopdf}
\usepackage{color}
\usepackage{bm}
\usepackage{comment}
\usepackage{soul}
\usepackage{centernot}

\allowdisplaybreaks
\newtheorem{theorem}{Theorem}[section]
\newtheorem{proposition}[theorem]{Proposition}
\newtheorem{lemma}[theorem]{Lemma}
\newtheorem{corollary}[theorem]{Corollary}

\newtheorem{definition}[theorem]{Definition}

\def\Ep{E^{(p)}}
\def\mcT{\mathcal{T}}
\def\mcJ{\mathcal{J}}

\def\mcE{\mathcal{E}}

\def\Osc{\text{Osc}}

\numberwithin{equation}{section}

\begin{document}
\title[$p$-energies on p.c.f. self-similar sets]{$p$-energies on p.c.f. self-similar sets}

\author{Shiping Cao}
\address{Department of Mathematics, Cornell University, Ithaca 14853, USA}
\email{sc2873@cornell.edu}

\author{Qingsong Gu}
\address{Department of Mathematics, Nanjing University, Nanjing, 210093, P. R. China}
\email{qingsonggu@nju.edu.cn}

\author{Hua Qiu}
\address{Department of Mathematics, Nanjing University, Nanjing, 210093, P. R. China}
\thanks{The research of Qiu was supported by the National Natural Science Foundation of China, grant 12071213, and the Natural Science Foundation of Jiangsu Province in China, grant BK20211142.}
\email{huaqiu@nju.edu.cn}


\date{}


\begin{abstract}

We study $p$-energies on post critically finite (p.c.f.) self-similar sets for $1<p<\infty$, as limits of discrete $p$-energies on approximation graphs, extending the construction of Dirichlet forms, the $p=2$ setting. By suitably enlarging the choices of discrete $p$-energies, and employing the energy averaging method developed by Kusuoka-Zhou, we prove the existence of symmetric $p$-energies on affine nested fractals, and extend Sabot's celebrated criterion for existence and non-existence of Dirichlet forms on p.c.f. self-similar sets to the $1<p<\infty$ setting.
\end{abstract}
\maketitle
\tableofcontents

\section{Introduction}
In this paper, we consider the problem of constructing $p$-energy forms ($p\in(1,\infty)$) on post critically finite (p.c.f.) self-similar sets. A typical example is on the unit interval $[0,1]$, for each $f\in W^{1,p}$, the $p$-energy of $f$ is defined as
\[\mcE_p(f)=\int_0^1|\nabla f(x)|^pdx,\]
which can be approximated by discrete energies,
\[\mcE_p(f)=\lim_{k\to\infty}2^{(p-1)k}\sum_{l=1}^{2^k}\left|f\left(\frac
l{2^k}\right)-f\left(\frac{l-1}{2^k}\right)\right|^p.\]
For $p=2$, on p.c.f. self-similar sets, the same construction was introduced by Kigami \cite{ki1,ki2} to construct Dirichlet forms. In this paper, for $1<p<\infty$, on p.c.f. self-similar sets, we will define the $p$-energy forms in a similar manner, by taking the limit of the averaged  discrete energies on first $n$ level graphs approximating to the fractal, inspired by Kusuoka-Zhou's method \cite{KZ}.

In history, the energies ($p=2$ case) on p.c.f. self-similar sets were motivated by the study of Brownian motions on fractals, with pioneering works of Kusuoka \cite{kus}, Goldstein \cite{G} and Barlow-Perkins \cite{BP} on the Sierpinski gasket, and of Lindstr{\o}m \cite{Lindstrom} on nested fractals. By introducing the tool of Dirichlet forms \cite{ki2}, all the constructions can be done in a purely analytic way (see books \cite{B,ki3,s}). There have been deep studies on the existence of such forms \cite{HMT,M1,M2,M3,Pe,Pe2,Sabot}. In particular, a famous test was introduced in the celebrated work of Sabot \cite{Sabot} (also contributes to the uniqueness under some reasonable assumptions).

On the other hand, for $p\neq 2$, on a general p.c.f. self-similar set $K$ (with associated iterated function system $\{F_i\}_{i=1}^N$, $N\geq 2$), we can not expect a direct extension of Kigami's approach, since there is no eigenforms of the renormalization map acting only on $p$-energies of the (standard) form
\[\sum_{x,y\in V_0}c_{x,y}|f(x)-f(y)|^p.\]
Here $V_0$ is the boundary of $K$ which consists of finitely many points, the renormalization map refers to the process of defining a new energy form on $V_0$ out of an old one, by first defining a form on $V_1$ (the first level iteration of $V_0$ under $\{F_i\}_{i=1}^N$) in a self-similar way, and then tracing it back on $V_0$. This was first observed in the explorative work \cite{HPS} of Herman, Peirone and Strichartz on the Sierpinski gasket, and the problem was resolved by considering the renormalization map on a larger class which includes certain non-standard $p$-energies. In this work, we will admit the same idea of \cite{HPS} and show the existence of an eigenform on affine nested fractals, and generalize Sabot's test to $p\neq 2$, which greatly extends the results of \cite{HPS}. 

To construct a natural $p$-energy form on a fractal $K$, we admit the construction of Kusuoka-Zhou \cite{KZ}. As long as we show the existence of an eigenform of the renormalization map with scaling factors $\{r_i\}_{i=1}^N$, we see that
\[\frac{1}{n}\sum_{m=0}^{n-1}\sum_{w\in W_m}r_w^{-1}\big|f(F_wx)-f(F_wy)\big|^p\]
admits a converging subsequence,
\[\mcE(f)=\lim\limits_{l\to\infty}\frac{1}{n_l}\sum_{m=0}^{n_l-1}\sum_{w\in W_m}r_w^{-1}\big|f(F_wx)-f(F_wy)\big|^p,\]
where the limit is defined to be the self-similar $p$-energy of a function $f$ on $K$ (see Section \ref{sec3} for the standard notations of $F_w$, $r_w$ and $W_m$). In particular, when $p=2$, the same construction gives the Dirichlet form on $K$. Following Kumagai's method \cite{ku}, we will see that even in the non-regular case ($r_i>1$ for some $i$), the $p$-energy can also be well understood as defined on $L^p(K,\mu)$ providing the measure $\mu$ is suitably chosen.

In fact, the technique of Kusuoka and Zhou \cite{KZ} was firstly invented to construct self-similar Dirichlet forms on the Sierpinski carpets, typical symmetric fractals with infinitely ramification property. Recently, by introducing necessary arguments replacing the ``Knight move'' construction \cite{BB,BB3,BBKT}, two of the authors successfully show a purely analytic construction of Dirichlet forms on a wider class of planar fractals named unconstrained Sierpinski carpets \cite{CQ}. Subsequently, Kigami \cite{ki4} and Shimizu \cite{shi} introduced the arguments and the notion of $p$-modules to the setting of $p$-energies for $1<p<\infty$ on Sierpinski carpets and some other square-based fractals. However, a theory for $p$-energy forms on general p.c.f. self-similar sets remains blank and is urgent to be set up. By using the method of $\Gamma$-convergence of Besov-type seminorms, Gao, Yu and Zhang \cite{GYZ} constructed $p$-energy forms on a class of p.c.f. self-similar sets under additional assumptions on the critical Besov space. The purpose of this paper is to attempt to establish general pratical criterions for the existence and non-existence of $p$-energy forms on general p.c.f. self-similar sets.

We organize the structure of the paper as follows. It is roughly divided into four parts.

In Section \ref{sec2}-\ref{sec4}, we consider the problem of finding an eigenform of the renormalization map. In Section \ref{sec2}, following the pioneering work \cite{HPS}, we introduce several classes of non-standard $p$-energy forms on finite sets, and prove that one can compare two forms via their associated $p$-effective resistances. In Section \ref{sec3}, we introduce the basic concepts of p.c.f. self-similar sets and the renormalization maps, and develop some basic properties. In Section \ref{sec4}, by modifying the idea of Kusuoka-Zhou \cite{KZ}, we obtain an equivalent condition for the existence of an eigenform, denoted as assumption (\textbf{A}).

In Section \ref{sec5}, we turn to the construction of $p$-energies on p.c.f. self-similar sets under assumption (\textbf{A}). In particular, in Section \ref{subsec52}, \ref{subsec53}, along with Appendix \ref{AppendixA}, we will focus on the non-regular case.

The short section, Section \ref{sec6}, will deal with affine nested fractals. By verifying assumption (\textbf{A}), we show that for any given symmetric renormalization factors, there always exists a symmetric $p$-energy form on the fractals.

Finally, in Section \ref{sec7} and \ref{sec8}, for general p.c.f. self-similar sets, we prove criterion for assumption (\textbf{A}). We will follow the spirit of Sabot in \cite{Sabot} and utilize the technique of preserved equivalent relations in the proof. However, many of the arguments need essential modifications due to the wild properties of discrete $p$-energy forms. Moreover, we are able to drop the (\textbf{H}) condition of Sabot (see \cite[page 649]{Sabot}). Readers may notice that  Metz has improved Sabot's results in an earlier work  \cite{M3}, and recently Peirone presented a new proof \cite{Pe2} based on a fixed point theorem of anti-attracting maps.\vspace{0.2cm}

Throughout this paper, on any set $X$, we write $l(X)$ for the space of real-valued functions on $X$; for $f\in l(X)$, we write $\Osc(f)=\sup\{|f(x)-f(y)|: x,y\in X\}$, the oscillation of $f$ on $X$; if $(X,d)$ is a metric space, we write $C(X)$ for the space of continuous functions on $X$. We always make the convention that $\frac{1}{\infty}=0$ and $\frac{1}{0}=\infty$. For two variables $a$ and $b$, we sometimes use $a\gtrsim b$  to mean $a\geq Cb$ for some constant $C>0$, and  $a\lesssim b$ in a similar way. We write $a\asymp b$ if both $a\gtrsim b$ and $a\lesssim b$ hold. We also denote $a\vee b:=\max\{a,b\}$ and $a\wedge b:=\min\{a,b\}$.

\section{Discrete $p$-energies}\label{sec2}
In this section, we define three classes of discrete $p$-energies on a finite set $A$:
\[\mathcal{M}_p(A)\supset \mathcal{Q}_{p}(A)\supset \mathcal{S}_{p}(A).\]
Here $\mathcal{S}_p(A)$ stands for standard energies of the form $\Ep(f)=\sum_{x,y\in A}c_{x,y}|f(x)-f(y)|^p$, that are natural analogs to discrete Dirichlet forms ($p=2$ case). Unfortunately, when $p\neq 2$, the trace of standard forms is in general non-standard, so we define $\mathcal{Q}_p(A)$ for the closure of traces of standard forms, where ``$\mathcal{Q}$'' stands for quasi-standard. $\mathcal{M}_p(A)$ is the most general class introduced in \cite{HPS}, which can be avoided in the statement of theorems, but provides us freedom in the proofs (essential for Sections \ref{sec7} and \ref{sec8}). We also write $\widetilde{\mathcal{M}}_p(A)$ for the class including degenerate forms, extending $\mathcal{M}_p(A)$.

We will list the basic properties of these $p$-energies that we need to use in Section \ref{sec3}-\ref{sec6}, and we will return to talk more in Section \ref{sec7}.\vspace{0.2cm}

First, we follow \cite{HPS} to introduce a large class of discrete $p$-energies on a finite set $A$. The main result in this section is Proposition \ref{prop26}, which is a $p$-version of  \cite[Lemma 2.7]{C}. In earlier celebrated work by Sabot \cite{Sabot}, there was also a weaker version (\cite[Lemma 1.19]{Sabot}).

\begin{definition}\label{def21}
Let $A$ be a finite set with $\# A\geq 2$ and $1<p<\infty$. A functional $\Ep$ on $l(A)$ is called a $p$-energy on $A$ if it satisfies
	
\noindent(i). {\bf Non-negativity:}  $\Ep(f)\geq0$ for all $f\in l(A)$;
	
\noindent (ii). {\bf Convexity:} $E^{(p)}(tf+(1-t)g)\leq tE^{(p)}(f)+(1-t)E^{(p)}(g)$ for all $f,g\in l(A)$ and $t\in(0,1)$;
	
\noindent (iii). {\bf Homogeneity of degree $p$:} $E^{(p)}(tf)= |t|^pE^{(p)}(f)$ for all $f\in l(A)$ and $t\in \mathbb{R}$;

\noindent (iv). {\bf invariant under addition of constants:} $\Ep(f+t)=\Ep(f)$ for all $f\in l(A)$ and $t\in\mathbb R$;
	
\noindent (v). {\bf Markov property:} $\Ep(\bar{f})\leq \Ep(f)$ for any $f\in l(A)$, where $\bar {f}=(f\vee 0)\wedge1$;

\noindent (vi). {\bf Non-degeneracy:} $\Ep(f)=0$ if and only if $f$ is constant on $A$.

We will write $\mathcal{M}_p(A)$ for the collection of all $p$-energies $\Ep$ satisfying (i)--(vi). We also write $\widetilde{\mathcal{M}}_p(A)$ for the collection of $\Ep$ satisfying (i)--(v) and that $\Ep(1)=0$.
\end{definition}

The following proposition lists some simple properties of $p$-energies which can be checked directly and we leave the proof to readers.

\begin{proposition}\label{prop22} Let $A$ be a finite set and $B\subset A$.

(a). If $\Ep_1, \Ep_2\in \widetilde{\mathcal{M}}_p(A)$, then $\Ep_1+\Ep_2\in \widetilde{\mathcal{M}}_p(A)$; if further $\Ep_1$ or $\Ep_2$ is in $\mathcal{M}_p(A)$, then $\Ep_1+\Ep_2\in\mathcal{M}_p(A)$.

(b). If $\Ep\in \widetilde{\mathcal{M}}_p(B)$, then $E'^{(p)}\in\widetilde{\mathcal{M}}_p(A)$, where $E'^{(p)}: l(A)\to [0,\infty)$ is defined by
\[E'^{(p)}(f)=\Ep(f|_{B}), \quad \text{ for } f\in l(A).\]

(c). If $\Ep\in \widetilde{\mathcal{M}}_p(A)$, then $E'^{(p)}\in\widetilde{\mathcal{M}}_p(B)$, where $E'^{(p)}: l(B)\to [0,\infty)$ is defined by
\[E'^{(p)}(f)=\inf \big\{\Ep(f'): f'\in l(A), f'|_B=f\big\}, \quad \text{ for } f\in l(B).\]
We will always write this $E'^{(p)}$ as $[\Ep]_B$ and call it the trace of $\Ep$ to $B$. In addition, $[\Ep]_B\in\mathcal M_p (B)$ providing that $\Ep\in\mathcal{M}_p(A)$.

(d). For $\Ep\in \widetilde{\mathcal{M}}_p(A)$ and $f\in l(A)$, if we denote $f_+=f\vee 0$ and $f_-=f_+-f$, then $\Ep(f_+)\leq\Ep(f)$ and $\Ep(f_-)\leq \Ep(f)$.
\end{proposition}

\begin{definition}\label{def23}
Let $\Ep\in \widetilde{\mathcal{M}}_p(A)$. For any $x\neq y\in A$, we define the $p$-effective resistance ($p$-resistance for short) between $x,y$ to be
\[R^{(p)}(x,y)=\frac{1}{\inf\big\{\Ep(f):f(x)=0,f(y)=1\big\}}\in (0,\infty], \]
and write $R^{(p)}(x,x)=0$ for any $x\in A$ by convention.
\end{definition}

We have the following simple properties of $p$-resistances.
\begin{proposition}\label{prop24}
(a). For $\Ep\in \widetilde{\mathcal{M}}_p(A)$, $R^{(p)}$ is symmetric, i.e. $R^{(p)}(x,y)=R^{(p)}(y,x)$ for any $x,y\in A$.

(b). For $\Ep\in \widetilde{\mathcal{M}}_p(A)$, if for some $x\neq y\in A$, $R^{(p)}(x,y)<\infty$, then $R^{(p)}(x,y)$ is realized by some function $f\in l(A)$, i.e. there exists a function $f$ satisfying $f(x)=0$, $f(y)=1$ and $\Ep(f)=R^{(p)}(x,y)^{-1}$.

(c). For $\Ep\in \widetilde{\mathcal{M}}_p(A)$ and $x,y,z\in A$ such that $R^{(p)}(x,y)$, $R^{(p)}(x,z)$, $R^{(p)}(y,z)<\infty$, it holds that
\begin{equation}\label{eqn21}
	\big(R^{(p)}(x,y)\big)^{1/p}\leq \big(R^{(p)}(x,z)\big)^{1/p}+\big(R^{(p)}(z,y)\big)^{1/p}.
	\end{equation}
In particular, if $\Ep\in {\mathcal{M}}_p(A)$, $\{R^{(p)}(\cdot,\cdot)\}^{1/p}$ is a metric on $A$.
\end{proposition}
\begin{proof}
(a). By the homogeneity of $\Ep$, for any $f\in l(A)$, we have $\Ep(f)=\Ep(-f)$, and also $\Ep(-f)=\Ep(1-f)$ by the invariance under addition of constants, hence $\Ep(f)=\Ep(1-f)$, and this implies that $R^{(p)}(x,y)=R^{(p)}(y,x)$ for any $x,y\in A$.

(b). Let $f_n\in l(A)$ be a sequence of functions satisfying $f_n(x)=0, f_n(y)=1$ and $0\leq f_n\leq1$ such that
\begin{equation*}
\lim_{n\rightarrow\infty}\Ep(f_n)=R^{(p)}(x,y)^{-1}.
\end{equation*}
Since $f_n$ are uniformly bounded, we may choose a subsequence, which still denote by $f_n$, such that for any $z\in A$, $\lim_{n\rightarrow\infty}f_n(z)$ exists, denoted by $f(z)$. Then the function $f$ is as required by the continuity of $\Ep$ (see Lemma \ref{lemmaa1} for a proof of a stronger result).

(c). Indeed, let $f\in l(A)$ such that $f(x)=0$, $f(y)=1$ and $\Ep(f)=\frac{1}{R^{(p)}(x,y)}$. Then by the homogeneity of $\Ep$, we have $|f(x)-f(z)|\leq\big (R^{(p)}(x,z)\Ep(f)\big)^{\frac 1 p}=\left(\frac{R^{(p)}(x,z)}{R^{(p)}(x,y)}\right)^{\frac{1}{p}}$ and $|f(z)-f(y)|\leq \left(R^{(p)}(z,y)\Ep(f)\right)^{\frac 1 p}=\left(\frac{R^{(p)}(z,y)}{R^{(p)}(x,y)}\right)^{\frac{1}{p}}$. Then \eqref{eqn21} follows from the triangle inequality $|f(x)-f(y)|\leq |f(x)-f(z)|+|f(z)-f(y)|$.
\end{proof}

In $p=2$ the quadratic case, a discrete Dirichlet form on $A$ is uniquely determined by the associated resistance metric \cite{ki3}. But for $p\neq 2$, $\widetilde{\mathcal{M}}_p(A)$ is a large class, a $p$-energy is not determined by the $p$-resistance. However, we still have useful estimates from the $p$-resistance, which we will see soon in Proposition \ref{prop26}.

\begin{lemma}\label{lemma25}
(a). For any $\Ep\in \widetilde{\mathcal{M}}_p(A)$ and $f\in l(A)$, we have $\Ep(|f|)\leq 2^p\Ep(f)$.
	
(b). For any $\Ep\in \widetilde{\mathcal{M}}_p(A)$ and $f,g\in l(A)$, we have
\[\begin{cases}
\Ep(f\wedge g)\leq 2^{{2p}-1}\big(\Ep(f)+\Ep(g)\big),\\
\Ep(f\vee g)\leq 2^{{2p}-1}\big(\Ep(f)+\Ep(g)\big).
\end{cases} \]

(c). There exists a constant $C>0$ depending on $\#A$ and $p$, such that for any $\Ep\in \widetilde{\mathcal{M}}_p(A)$ and $B\subsetneq A$, we have
\[C^{-1}\sum_{x\in B,y\in A\setminus B}R^{(p)}(x,y)^{-1}\leq\Ep(1_B)\leq C\sum_{x\in B,y\in A\setminus B}R^{(p)}(x,y)^{-1}. \]
\end{lemma}

\begin{proof}
(a). It follows from Proposition \ref{prop22} (d) that $\Ep(f_+)\leq \Ep(f)$ and $\Ep(f_-)\leq \Ep(f)$. So by the homogeneity and convexity of $\Ep$, we have $\Ep(|f|)=\Ep(f_++f_-)\leq 2^{p-1}\big(\Ep(f_+)+\Ep(f_-)\big)\leq2^p\Ep(f)$. 	
	
(b). We only deal with the ``$f\wedge g$" case, since the ``$f\vee g$" case is similar. By using homogeneity, convexity of $E^{(p)}$, and by (a), it follows
\[\begin{aligned}
\Ep(f\wedge g)&=\Ep\left(\frac{f+g-|f-g|}{2}\right)\leq 2^{-1}\big(\Ep(f+g)+\Ep(|f-g|)\big)\\
&\leq 2^{p-1}\big(\Ep(f+g)+\Ep(f-g)\big)\\
&\leq 2^{2p-1}\big(\Ep(f)+\Ep(g)\big).
\end{aligned}\]

(c). For the lower bound, note that for any $x\in B$ and $y\in A\setminus B$, we have \[R^{(p)}(x,y)^{-1}\leq \Ep(1_B).\]
By summing up the above inequality over all $x\in B$ and $y\in A\setminus B$ and observe that $\#A$, $\#B$ are finite, we get the desired estimate.

For the upper bound, for each $x\in B$ and $y\in A\setminus B$, let $f_{x,y}\in l(A)$ satisfy $f_{x,y}(x)=1,f_{x,y}(y)=0$, $0\leq f_{x,y}\leq 1$, and $\Ep(f_{x,y})=(R^{(p)}(x,y))^{-1}$. Observe that $1_B=\max\limits_{x\in B}\min\limits_{y\in A\setminus B}f_{x,y}$. By applying (b) a few times to $\Ep(1_B)=\Ep(\max\limits_{x\in B}\min\limits_{y\in A\setminus B}f_{x,y})$, we obtain the inequality as required.
\end{proof}

\begin{proposition}\label{prop26}
There exists a constant $C>0$ depending only on $\#A$ and $p$, such that for any $\Ep_1,\Ep_2\in \widetilde{\mathcal{M}}_p(A)$ and non-constant function $f\in l(A)$, we have
\[C^{-1}\inf_{x\neq y\in A}\frac{R^{(p)}_2(x,y)}{R^{(p)}_1(x,y)}\leq \frac{\Ep_1(f)}{\Ep_2(f)}\leq C\sup_{x\neq y\in A}\frac{R^{(p)}_2(x,y)}{R^{(p)}_1(x,y)},\]
where $R^{(p)}_i$ is the $p$-resistance associated with $\Ep_i$ for $i=1,2$, and we do not count $\frac{\infty}{\infty}$ in the supremum or infimum.
\end{proposition}
\begin{proof}
	Assume $f(A)=\{l_j\}_{j=1}^m$, where we order $l_1<l_2<\cdots<l_m$ with $2\leq m\leq \#A$. Then, we have
	\[f=l_1+\sum_{j=2}^m \big(l_j-l_{j-1}\big)\cdot 1_{\{f\geq l_j\}}.\]

	On one hand, by applying the right inequality in Lemma \ref{lemma25} (c), for $i=1,2$, we have
	\[
	\begin{aligned}
	\Ep_i(f)\leq C_1\sum_{j=2}^m \Ep_i\big((l_j-l_{j-1})\cdot 1_{\{f\geq l_j\}}\big)\leq C_2\sum_{j=2}^m(l_j-l_{j-1})^p\sum_{x\in \{f\geq l_j\},y\in \{f<l_j\}}R_i^{(p)}(x,y)^{-1},
    \end{aligned}
    \]
    where $C_1,C_2>0$ are constants only depending on $\#A$ and $p$.

    On the other hand, by the Markov property and the left inequality in Lemma \ref{lemma25} (c), for $i=1,2$, we obtain
    \[
    \begin{aligned}
    	\Ep_i(f)&\geq \max_{2\leq j\leq m} \Ep_i\big((l_j-l_{j-1})\cdot 1_{\{f\geq l_j\}}\big)\geq \frac{1}{\#A-1}\sum_{j=2}^m \Ep_i\big((l_j-l_{j-1})\cdot 1_{\{f\geq l_j\}}\big)\\
    	&\geq \frac{C_3}{\#A-1}\sum_{j=2}^m(l_j-l_{j-1})^p\sum_{x\in \{f\geq l_j\},y\in \{f<l_j\}}R_i^{(p)}(x,y)^{-1}\\
        &= C_4\sum_{j=2}^m(l_j-l_{j-1})^p\sum_{x\in \{f\geq l_j\},y\in \{f<l_j\}}R_i^{(p)}(x,y)^{-1},
    \end{aligned}
    \]
    where $C_3,C_4>0$ are constants only depending on $\#A$ and $p$.

    Now the proposition follows immediately from the above two estimates.
\end{proof}

The following is a direct corollary of Proposition \ref{prop26}.

\begin{corollary}\label{coro27}
There is a constant $C>0$ depending on $\#A$ and $p$, such that for any $\Ep\in {\widetilde{\mathcal{M}}_p(A)}$, we can find $\bar{E}^{(p)}\in {\widetilde{\mathcal{M}}_p(A)}$ of the form
\begin{equation}\label{standardform}
\bar{E}^{(p)}(f)=\sum_{x,y\in A,\ x\neq y}c_{x,y}\big|f(x)-f(y)\big|^p,
\end{equation}
	with $c_{x,y}\geq0$ such that for any $f\in l(A)$,
\begin{equation}\label{eqEbar}
C^{-1}\bar{E}^{(p)}(f)\leq \Ep(f)\leq C\bar{E}^{(p)}(f).
\end{equation}
\end{corollary}
\begin{proof}
	We simply choose $c_{x,y}=\frac1{R^{(p)}(x,y)}$ for any $x\neq y\in A$, and denote $\bar{R}^{(p)}$ the $p$-resistance associated with $\bar{E}^{(p)}$. Clearly, $\bar{E}^{(p)}\in\widetilde{\mathcal M}_p(A)$.

We will show that $C_1R^{(p)}(x,y)\leq\bar{R}^{(p)}(x,y)\leq R^{(p)}(x,y)$ for some constant $C_1>0$ depending only on $\# A$ and $p$, then the corollary follows from Proposition \ref{prop26}.

The upper bound $\bar{R}^{(p)}(x,y)\leq R^{(p)}(x,y)$ is immediate from the definition. It remains to show \begin{equation}\label{eqn24}
	C_1R^{(p)}(x,y)\leq\bar{R}^{(p)}(x,y).
	\end{equation}
	
For this purpose, we fix $x,y$ and let $\rho=\big(R^{(p)}(x,y)\big)^{\frac{1}{p}}$.
Let $A_x\subsetneq A$ be the set of all points $z$ having the property:

 there exists a sequence of distinct points $z=z_0,z_1,\cdots,z_m=x$ such that
\[
R^{(p)}(z_i,z_{i+1})<\left(\frac{\rho}{\#A}\right)^p.
\]
Clearly, $x\in A_x$, and $y\notin A_x$ by Proposition \ref{prop24} (c).
We then define a function $f$ taking values $0$ on $A_x$ and $1$ on $A_x^c$. Then by $f(x)=0, f(y)=1$, we have
\begin{equation}\label{eqRR1}
\bar{R}^{(p)}(x,y)^{-1}\leq\bar{E}^{(p)}(f).
\end{equation}
Also, for any $z\in A_x$ and $w\in A_x^c$, we must have $R^{(p)}(z,w)\geq \left(\frac{\rho}{\#A}\right)^p$, and hence
\begin{equation}\label{eqRR2}
\bar{E}^{(p)}(f)=\sum_{z\in A_x,w\in A_x^c}\frac{1}{R^{(p)}(z,w)}\leq \#A_x\cdot\#A_x^c\cdot\left(\frac{\#A}{\rho}\right)^p\leq (\#A)^{p+2}R^{(p)}(x,y)^{-1}.
\end{equation}
Thus (\ref{eqn24}) holds by combining \eqref{eqRR1}, \eqref{eqRR2} and taking $C_1= (\#A)^{-p-2}$.
\end{proof}\vspace{0.2cm}

In most part of the paper, we will restrict ourselves to more concrete classes $\mathcal{S}_p(A),\mathcal{Q}_p(A)$ of $p$-energies. All the stories in Sections \ref{sec3}-\ref{sec6} can be stated with only these two classes.

\begin{definition}\label{def28}
We introduce a norm $\|\cdot\|_{\widetilde{\mathcal M}_p(A)}$ on $\widetilde{\mathcal M}_p(A)$ as follows:
\[\|\Ep\|_{\widetilde{\mathcal M}_p(A)}=\sup_{f\in l(A)\setminus Constants}\frac{\Ep(f)}{\max_{x,y\in A}|f(x)-f(y)|^p},\quad \forall \Ep\in \widetilde{\mathcal M}_p(A).\]

(a). Define
\[\begin{aligned}
	\mathcal{S}_p(A)=\big\{\Ep\in \mathcal{M}_p(A):&\text{ There exist constants  }c_{x,y}\geq 0 \text{ depending only on }x,y\\&\text{ so that }\Ep(f)=\sum_{x,y\in A}c_{x,y}|f(x)-f(y)|^p\big\}.
\end{aligned}\]

(b). Define
\[\begin{aligned}
	\mathcal{Q}'_p(A)=\big\{\Ep\in \mathcal{M}_p(A):\ &\text{There exists }B\supset A\text{ and }\Ep_B\in \mathcal{S}_p(B)\text{ so that }\Ep=[\Ep_B]_A\big\}.
\end{aligned}\]
Take $\mathcal{Q}_p(A)$ to be the closure of $\mathcal{Q}'_p(A)$ in $\big(\mathcal{M}_p(A),\|\cdot\|_{\widetilde{\mathcal M}_p(A)}\big)$.
\end{definition}

\noindent\textbf{Remark 1.} The class $\widetilde{\mathcal M}_p(A)$ is a cone. The spanned space of $\widetilde{\mathcal M}_p(A)$ under $\|\cdot\|_{\widetilde{\mathcal M}_p(A)}$ is a complete normed space. However, $\mathcal{S}_p(A)\subset\mathcal{Q}_p(A)\subset\mathcal{M}_p(A)$ are all not closed in $\widetilde{\mathcal M}_p(A)$ under this norm.\vspace{0.2cm}

\noindent\textbf{Remark 2.} Any $\Ep\in \mathcal{S}_p(A)$ is uniformly convex (\cite{H}), which means for any $\delta>0$ and any $f,g\in l(A)$ such that $\Ep(f)=\Ep(g)=1$ and $\Ep(f-g)\geq \delta$, we have
\[\Ep\left(\frac{f+g}{2}\right)\leq 1-\varepsilon\]
for some $\varepsilon>0$ depending on $p$ and $\delta$. So it is direct to check that any $\Ep\in \mathcal{Q}_p(A)$ is also uniformly convex, hence strictly convex.\vspace{0.2cm}

\noindent\textbf{Remark 3.} It follows from Corollary \ref{coro27} that, for any $\Ep\in \mathcal{M}_p(A)$, we can find  $\bar{E}^{(p)}\in {\mathcal{S}_p(A)}$ such that the two energies are comparable. \vspace{0.2cm}

We will list some important properties of $\mathcal{Q}_p(A)$ in Appendix \ref{AppendixA}, which will be used in Section \ref{subsec52}.

\section{p.c.f. self-similar sets and the renormalization maps}\label{sec3}

In this section, we introduce the renormalization map in $p$-energy setting on the post critically finite (p.c.f. for short) self-similar sets introduced by Kigami \cite{ki2, ki3}. Please refer to \cite{HMT, HN, Lindstrom, M1,M2, N,Pe} and the references therein for fruitful previous works for the $p=2$ case, especially, the celebrated work of Sabot \cite{Sabot} on the existence and uniqueness of an eigenvector of this map for nested fractals.

For simplicity, we will focus on the self-similar sets in $\mathbb R^d$. Let $\{F_i\}_{i=1}^N$ be a sequence of $\alpha_i$-similitudes on $\mathbb R^d$, i.e. for $i=1,\cdots,N$, $N\geq 2$, each $F_i:\ \mathbb R^d\rightarrow \mathbb R^d$ is of the form
\begin{equation*}
F_i(x)=\alpha_i U_ix+a_i,
\end{equation*}
where $\alpha_i\in(0,1)$, $U_i$ is an orthogonal matrix, $a_i\in \mathbb{R}^d$.
We call $\{F_i\}_{i=1}^N$ an \textit{iterated function system} (i.f.s. for short). Let $K$ be the associated {\it self-similar set}, i.e. $K$ is the unique non-empty compact set in $\mathbb{R}^d$ satisfying
\begin{equation*}
K=\bigcup_{i=1}^N F_iK.
\end{equation*}

 We then define the associated symbolic space. Let $W= \{1,\cdots, N\}$ be the alphabet. Let $W_0=\{\emptyset\}$,  $W_n$ be the set of \textit{words} $w = w_1\cdots w_n$ of length $n$ with $w_i\in W$, $n\geq 1$ (we also denote $|w|=n$ for the \textit{length} of $w$), and write $W_*=\bigcup_{n\geq 0}W_n$. For $w\in W_n$, we write $F_w = F_{w_1} \circ \cdots \circ F_{w_n}$, and call $F_w K$ an \textit{$n$-cell} of $K$.  Let $\Sigma$ be the set of \textit{infinite words} $\omega = \omega_1 \omega_2 \cdots$ with $\omega_i\in W$. There is a continuous surjection   $\pi : \Sigma \to K$   defined by
$$ \{\pi(\omega)\}=\bigcap_{n\geq1}F_{[\omega]_n}K,$$ where for $\omega=\omega_1\omega_2\cdots$ in $\Sigma$ we write $[\omega]_n=\omega_1\cdots\omega_n\in W_n$ for each $n\geq 1$.

Following \cite{ki3},  we define the
\emph{critical set} $ \mathcal{C} $ and the \emph{post-critical set } $\mathcal{P}$ for $K$  by
\begin{equation*}
\mathcal{C}=\pi^{-1}\Big({\bigcup}_{i\neq j }\big( F_i K \cap F_jK\big )\Big),\quad \mathcal{P}={\bigcup}_{m\geq 1}\tau^m(\mathcal{C}),
\end{equation*}
where $\tau :\Sigma \to \Sigma$ is the left shift map defined as $\tau(\omega_1\omega_2\cdots)=\omega_2\omega_3\cdots$. If  $\mathcal{P}$ is a finite set, we call $\{F_i\}_{i=1}^N$ a {\it post-critically finite} (p.c.f.) i.f.s., and $K$  a \textit{p.c.f. self-similar set}. In what follows, we always assume that $K$ is a connected p.c.f. self-similar set.
 The  {\it boundary}  of $K$ is defined to be  $V_0= \pi(\mathcal{P})$.  We also define
$
V_n=\bigcup_{i=1}^NF_iV_{n-1}
$
inductively,
and write $V_*=\bigcup_{n\geq 0}V_n$.
It is clear that $V_*$ is a proper subset of $K$ and the closure of $V_*$ under the Euclidean metric is $K$.  \vspace{0.2cm}

When the fractal is symmetric, we will be particularly interested in symmetric $p$-energy forms. More precisely, we will consider a symmetric group $\mathscr{G}$ as follows:\vspace{0.2cm}
	
	\noindent\textbf{($\mathscr{G}$-symmetry).} Let $\mathscr{G}$ be a finite group of homeomorphisms $K\to K$. We say $(K,\{F_i\}_{i=1}^N)$ is \textit{$\mathscr{G}$-symmetric} (or simply $K$ is $\mathscr{G}$-symmetric) if for any $\sigma\in \mathscr G$ and $n\geq 0$, there is a permutation $\sigma^{(n)}:W_n\to W_n$ such that $\sigma\circ F_wK=F_{\sigma^{(n)}(w)}K$ and $\sigma\circ F_wV_0=F_{\sigma^{(n)}(w)}V_0$ for any $w\in W_n$.\vspace{0.2cm}

To construct a $p$-energy form on $K$, it suffices to study the forms on $V_n,n\geq 0$. In the rest of this section, we introduce a renormalization map on the forms on $V_n,n\geq 0$, and study some basic properties. We will talk more in Section \ref{sec7}.\vspace{0.2cm}

\noindent\textbf{Remark}. From now on, we study the limit of discrete $p$-energies on the expanding sequence $\{V_n\}_{n\geq0}$. Note that all the constants that involved will depend only on $\#V_0$ and $p$. In the following of this paper, we will fix $p\in(1,\infty)$. For simplicity, we will write $\mathcal{M}=\mathcal{M}_p(V_0)$, $\widetilde{\mathcal{M}}=\widetilde{\mathcal{M}}_p(V_0)$, and omit the index $p$ and superscript ${(p)}$ when no confusion is caused.\vspace{0.2cm}

We introduce a renormalization map $\mcT:\mathcal{M}(V_n)\to\mathcal{M}(V_n)$.

\begin{definition}\label{def31}
	Fix $\bm{r}=(r_1,r_2,\cdots,r_N)$ to be a positive vector. For $w=w_1\cdots w_n\in W_*$, we write $r_w=r_{w_1}\cdots r_{w_n}$. Let $n\geq 0$ and $E\in \widetilde{\mathcal{M}}(V_n)$.
	
	(a).  We define $\Lambda E\in \widetilde{\mathcal{M}}(V_{n+1})$ as
	\[\Lambda E(f)=\sum_{i=1}^Nr_i^{-1}E(f\circ F_i),\quad\text{ for }f\in l(V_{n+1}).\]
	
	(b).  We define $\mcT:\widetilde{\mathcal{M}}({V_{n}})\to \widetilde{\mathcal{M}}(V_n)$ as
	\[\mcT E=[\Lambda E]_{V_n}.\]
\end{definition}
\noindent{\bf Remark 1.} There is some abuses of notations, as the definition of $\Lambda,\mathcal{T}$ actually depends on $n$. However, we admit the notations since they provide great convenience. In particular, we can use the notation $\Lambda^m:\widetilde{\mathcal{M}}(V_{n})\to \widetilde{\mathcal{M}}(V_{n+m})$ for all $n\geq 0$. See the following Lemma \ref{lemma32} for example. \vspace{0.2cm}

\noindent{\bf Remark 2.} It follows from Proposition \ref{prop22} that the operators $\Lambda$ and $\mcT$ are well-defined.\vspace{0.2cm}

\noindent \textbf{More about ($\mathscr{G}$-symmetry).} Assume $K$ is $\mathscr{G}$-symmetric.
	We say $E\in \widetilde{\mathcal{M}}(V_n),n\geq 0$ is \textit{$\mathscr G$-symmetric} if $E(f\circ \sigma)=E(f)$ for any $f\in l(V_n)$ and $\sigma\in \mathscr{G}$. It is straightforward to see that if $E\in \widetilde{\mathcal{M}}(V_{n+1}),n\geq 0$ is $\mathscr{G}$-symmetric then $[E]_{V_n}$ is also $\mathscr{G}$-symmetric.
	
In addition, we assume $\bm{r}$ is $\mathscr{G}$-symmetric, i.e. $r_i=r_{\sigma^{(1)}(i)},\forall 1\leq i\leq N$. Then it is straightforward to check that if $E\in \widetilde{\mathcal{M}}(V_n),n\geq 0$ is $\mathscr{G}$-symmetric then $\Lambda E$ is also $\mathscr{G}$-symmetric.

\begin{lemma}\label{lemma32}
(a). Let $n\geq 0$ and $E\in \widetilde{\mathcal{M}}(V_{n+2})$, then $[[E]_{V_{n+1}}]_{V_{n}}=[E]_{V_n}$.

(b). Let $n\geq 0$ and $E\in \widetilde{\mathcal{M}}(V_n)$, then $[\Lambda^mE]_{V_n}=\mathcal{T}^mE$ for any $m\geq 1$.
\end{lemma}
\begin{proof}
(a). It suffices to show $[[E]_{V_{n+1}}]_{V_{n}}(f)=[E]_{V_n}(f)$ for each $f\in l(V_n)$.

First, one can find $f_1\in l(V_{n+1})$ so that $f_1|_{V_n}=f$ and  $[E]_{V_{n+1}}(f_1)=[[E]_{V_{n+1}}]_{V_{n}}(f)$; one can then find $f_2$ so that $f_2|_{V_{n+1}}=f_1$ and $E(f_2)=[E]_{V_{n+1}}(f_1)=[[E]_{V_{n+1}}]_{V_{n}}(f)$. So we have $[[E]_{V_{n+1}}]_{V_{n}}(f)\geq [E]_{V_n}(f)$ by definition.

Next, for any $g\in l(V_{n+2})$ satisfying $g|_{V_n}=f$, we have  $E(g)\geq [E]_{V_{n+1}}(g|_{V_{n+1}})\geq [E]_{V_{n+1}}(f_1)=E(f_2)$, hence $[[E]_{V_{n+1}}]_{V_{n}}(f)\leq [E]_{V_n}(f)$.

(b). It is easy to see that $[\Lambda^m E]_{V_{n+m-1}}=\Lambda^{m-1}\mathcal{T}E$. In fact, let $f\in l(V_{n+m-1})$, one can define $f_1\in l(V_{n+m})$ such that $\mathcal{T}E(f\circ F_w)=\Lambda E(f_1\circ F_w)$ for each $w\in W_{m-1}$, then for any $g\in l(V_{n+m})$ such that $g|_{V_{n+m-1}}=f$ we have
\[\begin{aligned}
\Lambda^{m} E(f_1)
&=\sum_{w\in W_{m-1}}r_w^{-1}\Lambda E(f_1\circ F_w)\\
&=\sum_{w\in W_{m-1}}r_w^{-1}\mathcal{T} E(f\circ F_w)\leq \sum_{w\in W_{m-1}}r_w^{-1}\Lambda E(g\circ F_w)=\Lambda^{m} E(g),
\end{aligned}\]
hence $[\Lambda^mE]_{V_{n+m-1}}(f)=\Lambda^{m} E(f_1)=\sum_{w\in W_{m-1}}r_w^{-1}\mathcal{T} E(f\circ F_w)=\Lambda^{m-1}\mathcal {T}E(f)$.

By using (a), one can now repeat the argument to see (b).
\end{proof}

Finally, we end this section with some basic properties of $\mathcal{T}$, which are essentially due to Metz \cite{M1} for the $p=2$ case.

\begin{definition}\label{def33}
	(a). Let $E_1,E_2\in \widetilde{\mathcal{M}}$. We define
	\[\begin{aligned}
		\sup(E_1|E_2)=\sup\left\{\frac{E_1(f)}{E_2(f)}:f\in l(V_0),E_1(f)+E_2(f)>0\right\},\\
		\inf(E_1|E_2)=\inf\left\{\frac{E_1(f)}{E_2(f)}:f\in l(V_0),E_1(f)+E_2(f)>0\right\}.
	\end{aligned}\]
	
	(b). Let $E\in \mathcal{M}$, we define
	\[\theta(E)=\frac{\sup(\mathcal{T}E|E)}{\inf(\mathcal{T}E|E)}.\]
\end{definition}

Following the routine argument on the $p=2$ setting, we can easily prove the following result.

\begin{lemma}\label{lemma34}
	(a). Let $E_1,E_2\in \mathcal{M}$, we have
	\[\begin{cases}
		\sup(\mathcal{T}E_1|\mathcal{T}E_2)\leq \sup(\Lambda E_1|\Lambda E_2)\leq\sup(E_1|E_2),\\
		\inf(\mathcal{T}E_1|\mathcal{T}E_2)\geq \inf(\Lambda E_1|\Lambda E_2)\geq\inf(E_1|E_2).
	\end{cases}
	\]
	
	(b). Let $E\in \mathcal{M}$, we have
	\[\theta(\mathcal{T}E)\leq \theta(E).\]
	
	(c). Let and $E_1,E_2\in \mathcal{M}$, we have
	\[\inf(\mathcal{T}E_1|E_1)\leq \sup(\mathcal{T}E_2|E_2).\]
\end{lemma}
\begin{proof}
	For each $f\in l(V_0)$, we denote by $f_i\in l(V_1)$ the minimal energy extension of $f$ with respect to $\Lambda E_i$ for $i=1,2$.
	
	(a). First, for any $g\in l(V_{1})$, we have
	\[\frac{\Lambda E_1(g)}{\Lambda E_2(g)}=\frac{\sum_{i=1}^N r_i^{-1}E_1(g\circ F_i)}{\sum_{i=1}^N r_i^{-1} E_2(g\circ F_i)}\leq \sup(E_1|E_2),\]
	whence $\sup(\Lambda E_1|\Lambda E_2)\leq \sup( E_1| E_2)$.
	
	Next, for any $f\in l(V_0)$, we have
	\[\mathcal{T}E_1(f)=\Lambda E_1(f_1)\leq\Lambda E_1(f_2)\leq \sup(\Lambda E_1|\Lambda E_2)\Lambda E_2(f_2)=\sup(\Lambda E_1|\Lambda E_2)\mathcal{T}E_2(f),\]
	so $\sup(\mathcal{T}E_1|\mathcal{T}E_2)\leq \sup(\Lambda E_1|\Lambda E_2)$.
	
	Combining the above two arguments, we get the inequality for the ``$\sup$'' part.
	Noticing that $\inf(E_1|E_2)=\sup(E_2|E_1)^{-1}$, we also get the inequality for the ``$\inf$'' part.
	
	(b) is an immediate consequence of (a):
	\[
	\theta(\mathcal{T}E)=\frac{\sup(\mathcal{T}^2E|\mathcal{T}E)}{\inf(\mathcal{T}^2E|\mathcal{T}E)}\leq \frac{\sup(\mathcal{T}E|E)}{\inf(\mathcal{T}E|E)}=\theta(E).
	\]
	
	(c). Choose $f\in l(V_0)$ so that $\frac{E_1(f)}{E_2(f)}=\sup(E_1|E_2)$. Then
	\[\mathcal{T}E_1(f)=\Lambda E_1(f_1)\leq \Lambda E_1(f_2)\leq \sup(E_1|E_2)\Lambda E_2(f_2)=\frac{E_1(f)}{E_2(f)}\mathcal{T}E_2(f),\]
	whence
	\[\inf(\mathcal{T}E_1|E_1)\leq\frac{\mathcal{T}E_1(f)}{E_1(f)}\leq \frac{\mathcal{T}E_2(f)}{E_2(f)}\leq \sup(\mathcal{T}E_2|E_2).\]
\end{proof}

\section{The existence of an eigenform}\label{sec4}
In this section, we study the existence of an eigenvector (eigenform) of $\mathcal{T}$, i.e. to find $E\in \mathcal{M}$ and $\lambda>0$ such that
\[\mcT E=\lambda E.\]

\begin{definition}\label{def41}
	For $E\in\mathcal M$, we define
	\begin{equation*}
		\delta(E)=\frac{\min_{x\neq y} R(x,y)}{\max_{x\neq y} R(x,y)}.
	\end{equation*}
\end{definition}
In the rest of this paper, we refer to the following condition as (\textbf{A}):\vspace{0.2cm}

\noindent(\textbf{A}). There exists $E\in \mathcal{M}$ such that $\inf_{n\geq 0}\delta(\mathcal{T}^nE)>0$. \vspace{0.2cm}

This condition always holds for nested fractals (Section \ref{sec6}). For general p.c.f. fractals, we have a strengthened version of Sabot's criteria (Section \ref{sec8}).

Similar as the notations $\mathcal M$, $\widetilde{\mathcal M}$, from now on, we abbreviate $\mathcal{S}_p(V_0)$, $\mathcal{Q}'_p(V_0)$ and $\mathcal{Q}_p(V_0)$ to $\mathcal S$, $\mathcal Q'$ and $\mathcal Q$, respectively.

\begin{theorem}\label{thm42}
Let $1<p<\infty$ and $K$ be a p.c.f. self-similar set. Then there exists $E\in \mathcal{Q}$ and $\lambda>0$ such that $\mcT E=\lambda E$ if and only if (\textbf{A}) holds. In particular, $\lambda$ is unique.

In addition, if $K$ and $\bm{r}$ are $\mathscr{G}$-symmetric, then we can  also require that $E$ is $\mathscr{G}$-symmetric.
\end{theorem}

We will prove the existence of an eigenform by the spirit of Kusuoka-Zhou's idea \cite{KZ}, using an averaging technique. The original proof was designed to construct a self-similar Dirichlet form directly, and has never been modified for the renormalization map before.

We also remark that in Proposition 4.4 of \cite{M1}, the $\omega$-limit technique was applied to show the existence of an eigenform (see Theorem 4.1 of \cite{N}), which can also be modified here to show the existence of an eigenform $E\in \mathcal{M}$ (it is not clear whether $\mathcal{Q}$ is a cone). Our new proof is self-contained and fundamental. Readers do not need further knowledge about fixed point theorems or deep results about Hilbert's projective metrics \cite{N}.\vspace{0.2cm}

We first state an easy fact about the norm $\|\cdot\|_{\widetilde{\mathcal{M}}}$.

\begin{lemma}\label{lemma43}
Let $E_n\in \widetilde{\mathcal{M}},n\geq 1$.
	
(a). If $\sup_{n\geq 1}\|E_n\|_{\widetilde{\mathcal{M}}}<\infty$, then there is a subsequence $n_l,l\geq 1$ and $E\in \widetilde{\mathcal{M}}$ so that $\|E_{n_l}-E\|_{\widetilde{\mathcal{M}}}\to 0$ as $l\to\infty$.

(b). If $E\in \mathcal{M}$ and $\|E_n-E\|_{\widetilde{\mathcal{M}}}\to 0$ as $n\to\infty$, then $\|\mathcal{T}E_n-\mathcal{T}E\|_{\widetilde{\mathcal{M}}}\to 0$.
\end{lemma}
\begin{proof}
Let $M=\big\{f\in l(V_0):\Osc(f)=1,\sum_{x\in V_0}f(x)=0\big\}$, then $M$ is a compact subset of $l(V_0)$ (with respect to the $\|\cdot\|_{\infty}$ norm). 	
	
(a). By the assumption, $\{E_n\}_{n}$ are uniformly bounded on $M$; and also for $f\neq g\in M$, we have
\[E_n(f)\leq tE_n\big(t^{-1}(f-g)\big)+(1-t)E_n\big((1-t)^{-1}g\big)=t^{1-p}E_n(f-g)+(1-t)^{1-p}E_n(g),\]
and similarly $E_n(g)\leq t^{1-p}E_n(f-g)+(1-t)^{1-p}E_n(f)$. Hence,
\[\big|E_n(f)-E_n(g)\big|\leq t^{1-p}E_n(f-g)+\big((1-t)^{1-p}-1\big)\cdot\max\big\{E_n(f),E_n(g)\big\}.\]
Let $t=\big(E_n(f-g)\big)^{1/p}$, noticing that
$E_n(f-g)\leq\|f-g\|^p_{\infty}\|E_n\|_{\widetilde{\mathcal M}}$ and $E_n(f)\leq \|E_n\|_{\widetilde{\mathcal M}}, E_n(g)\leq \|E_n\|_{\widetilde{\mathcal M}}$, we can see
\[\big|E_n(f)-E_n(g)\big|\leq \|E_n\|_{\widetilde{\mathcal M}}^{1/p}\|f-g\|_{\infty}+\big((1-\|E_n\|_{\widetilde{\mathcal M}}^{1/p}\|f-g\|_{\infty})^{1-p}-1\big)\cdot\|E_n\|_{\widetilde{\mathcal M}}.\]
Thus $\{E_n\}_{n}$ are equicontinuous on $M$.

Hence by Arzel\`a-Ascoli theorem, there is subsequence $\{n_l\}_{l\geq 1}$ so that $E_{n_l}$ converges uniformly on $M$, and hence $E_{n_l}$ converges to some $E\in \widetilde{\mathcal{M}}$ pointwisely on $l(V_0)$, noticing that any function in $l(V_0)$ can be decomposed into the sum of a constant function and a multiple of some function in $M$.

(b).  Since $M$ is compact and $E\in \mathcal{M}$, we have $S:=\min_{f\in M}E(f)>0$. Then

\[-\frac{\|E_n-E\|_{\widetilde{\mathcal{M}}}}{S}\leq \frac{E_n(f)-E(f)}{E(f)}\leq \frac{\|E_n-E\|_{\widetilde{\mathcal{M}}}}{S},\quad \forall f\in M,\]
which implies
\[\frac{S-\|E_n-E\|_{\widetilde{\mathcal{M}}}}{S}\leq \frac{E_n(f)}{E(f)}\leq \frac{S+\|E_n-E\|_{\widetilde{\mathcal{M}}}}{S},\quad \forall f\in M.\]
Hence, by Lemma \ref{lemma34} (a),
\[\frac{S-\|E_n-E\|_{\widetilde{\mathcal{M}}}}{S}\leq \frac{\mcT E_n(f)}{\mcT E(f)}\leq \frac{S+\|E_n-E\|_{\widetilde{\mathcal{M}}}}{S},\quad \forall f\in M,\]
which is equivalent to
\[\frac{\big|\mcT E_n(f)-\mcT E(f)\big|}{\mcT E(f)}\leq \frac{\|E_n-E\|_{\widetilde{\mathcal{M}}}}{S},\quad \forall f\in M.\]
Thus, $\|\mcT E_n-\mcT E\|_{\widetilde{\mathcal{M}}}\leq \frac{\|\mcT E\|_{\widetilde{\mathcal{M}}}}{S}\|E_n-E\|_{\widetilde{\mathcal{M}}}$.
\end{proof}

The following lemma is an alternative to \cite[Theorem 7.16]{KZ} by Kusuoka and Zhou.
\begin{lemma}\label{lemma44}
Assume (\textbf{A}), and let $E\in \mathcal{M}$. Then

(a). $\inf_{n\geq 0}\delta(\mcT^nE)>0$.

(b). There exist $\lambda>0$ and $C>0$ such that
	\[C^{-1}\lambda^nE\leq \mcT^n E\leq C\lambda^nE,\quad\forall n\geq 0.\]
\end{lemma}
\begin{proof}
(a) is an immediate consequence of Proposition \ref{prop26} and Lemma \ref{lemma34} (a).	
	
(b). Denote by $R_n$ and $R_0$ the $p$-resistances w.r.t. $\mcT^n E$ and $E$ respectively. Let $M_n=\sup_{x\neq y\in V_0}\frac{R(x,y)}{R_n(x,y)}$, and from (a), we denote $\delta=\inf_{n\geq 0}\delta(\mcT^n E)>0$. Then
\begin{equation}\label{equation4.1}
	M_n=\sup_{x\neq y\in V_0}\frac{R(x,y)}{R_n(x,y)}\geq \inf_{x\neq y\in V_0}\frac{R(x,y)}{R_n(x,y)}\geq \delta^2\sup_{x\neq y\in V_0}\frac{R(x,y)}{R_n(x,y)}=\delta^2M_n.
\end{equation}
Then by applying Proposition \ref{prop26} to $\mcT^n E$ and $E$, we see that there exists $C_1>1$ such that
\[C_1^{-1}M_nE\leq \mcT^n E\leq C_1M_nE,\qquad\forall n\geq 0.\]
By acting $\mcT^m$ to the above inequality and using it again, we obtain
\[\begin{aligned}
		C_1^{-2}M_nM_mE\leq C_1^{-1}M_n\mcT^mE\leq \mcT^{n+m}E\leq C_1M_n\mcT^mE\leq C^2_1M_nM_mE,\quad\forall m,n\geq 0,
\end{aligned}\]
whence by Proposition \ref{prop26} again, there exists $C_2>1$ such that
\[C_2M_nM_m\leq M_{n+m}\leq C_2M_nM_m.\]
It follows by Fekete's lemma that there exists $C_3>0$ and $\lambda>0$ such that
\[C_3^{-1}\lambda^n\leq M_n\leq C_3\lambda^n \qquad\forall n\geq1,\]
which together with \eqref{equation4.1} implies that
\begin{equation}\label{equation4.2}
C_4^{-1}\lambda^nR_n(x,y)\leq R(x,y)\leq C_4\lambda^n R_n(x,y),\qquad\forall x\neq y\in V_0\end{equation}
for some $C_4>1$. (b) then follows by \eqref{equation4.2} and  Proposition \ref{prop26}.
\end{proof}
\vspace{0.2cm}

\begin{proof}[Proof of Theorem \ref{thm42}] By Lemma \ref{lemma44} (a), we can choose $E_0\in\mathcal{S}$. In addition, if both $K$ and $\bm{r}$ are $\mathscr{G}$-symmetric, we can require that $E_0$ is also $\mathscr{G}$-symmetric, so all the constructions in the proof will be $\mathscr{G}$-symmetric.
	
Let $\lambda$ be the same constant in Lemma \ref{lemma44} (b) for $E_0$. We follow the idea of Kusuoka and Zhou \cite{KZ} to consider the averaged energy \[\mcE_n(f)=\frac{1}{n+1}\sum_{m=0}^{n} \lambda^{-m}\Lambda^mE_0(f|_{V_m}),\]
where $f\in l(V_{n})$. In \cite{KZ}, a normalized limit energy of $\mcE_n$ on $L^2(K,\mu)$ is considered, where $\mu$ is the normalized Hausdorff measure restricted on $K$. Thanks to the p.c.f. structure, here we can avoid considering the limit in the sense of energies on $L^p(K,\mu)$. We only need to consider the limit of
\[E_n=[\mcE_n]_{V_0}\in \mathcal{Q}'.\]

\noindent\textit{Claim. there is $C>0$ so that $C^{-1}E_0\leq E_n\leq CE_0$.}\vspace{0.2cm}

First, by Lemma \ref{lemma32} (b), one can see that
\[E_n\geq \frac1{n+1}\sum_{m=0}^{n}\lambda^{-m}[\Lambda^mE_0]_{V_0}=\frac1{n+1}\sum_{m=0}^{n}\lambda^{-m}\mcT^mE_0.\]
So by Lemma \ref{lemma44} (b), $E_n\geq C_1^{-1} E_0$, where $C_1>0$ is the constant in Lemma \ref{lemma44} (b).

Next, for $f\in l(V_0)$, let $f_n\in l(V_n)$ be an extension of $f$ which is a $p$-energy minimizer with respect to $\Lambda^n E_0$. Then, by Lemma \ref{lemma32} (b) and Lemma \ref{lemma44} (b), we have
\[
\begin{aligned}
	E_n(f)\leq\mcE_n(f_n)&=\frac{1}{n+1}\sum_{m=0}^{n} \lambda^{-m}\Lambda^mE_0(f_n|_{V_m})\\
	&\leq C_1\frac{1}{n+1}\lambda^{-n}\sum_{m=0}^{n}\Lambda^m\mcT^{n-m}E_0(f_n|_{V_m})=C_1\lambda^{-n}\mcT^nE_0(f)\leq C_1^2 E_0(f).
\end{aligned}
\]
This finishes the proof of the Claim.\vspace{0.2cm}

It is then from the claim that $\sup_{n\geq 1}\|E_n\|_{\widetilde{\mathcal M}}<\infty$, so by Lemma \ref{lemma43} (a), we can find a subsequence $\{E_{n_l}\}_{l\geq 1}$ that converges to some $E'\in\mathcal Q$ under norm $\|\cdot\|_{\widetilde{\mathcal M}}$. Note that for each $n\geq 1$, we have   \[\lambda^{-1}\Lambda\mcE_n-\mcE_n=\frac{1}{n+1}(\lambda^{-{n-1}}\Lambda^{n+1}E_0-E_0).\]
So that
\[\lambda^{-1}\mcT E_n\geq E_n-\frac{1}{n+1}E_0.\]
By taking the limit in the subsequence $\{E_{n_l}\}_{l\geq 1}$ of each side, and using Lemma \ref{lemma43} (b), we see that $E'\leq \lambda^{-1}\mcT E'\leq \lambda^{-2}\mcT^2 E'\leq \cdots$. We can use Lemma \ref{lemma44} (b) again to see the orbit $\lambda^{-n}\mcT^n E'$ is bounded from above, so we have $E=\lim\limits_{n\to\infty}\lambda^{-n}\mcT^n E'$ exists in $\mathcal Q$ (noticing that $\mathcal{Q}$ is closed in $\mathcal{M}$, although neither of which is complete), and clearly by Lemma \ref{lemma43} (b) again, \[\mcT E=\lim\limits_{n\to\infty}\lambda^{-n}\mcT^{n+1}E'=\lambda\lim\limits_{n\to\infty}\lambda^{-n}\mcT^nE'=\lambda E.\]
This shows that $\mcT$ has a $\mathscr{G}$-symmetric eigenform.

Finally $\lambda$ is unique since if there is another eigenform $\bar{E}$ such that $\mcT\bar{E}=\bar\lambda\bar{E}$ for some $\bar\lambda>0$, then by Lemma \ref{lemma34} (a), $\mcT^n\bar{E}\asymp \mcT^n E$ for any $n\geq 0$, which will force $\lambda=\bar\lambda$.
\end{proof}

\section{p-energies on p.c.f. self-similar sets}\label{sec5}
In this section, we construct  $p$-energies on p.c.f. self-similar sets, which covers the $2$-energy forms (Dirichlet forms) as a special case. We apply the idea of \cite{KZ} in this section.

Our story in this section is still based on Assumption (\textbf{A}). However, for convenience, we further require that $\lambda=1$ in Theorem \ref{thm42} (by multiplying $\bm{r}$ with a constant). Let us introduce a condition $(\textbf{A}')$ for convenience, which is essentially the same as (\textbf{A}) by Theorem \ref{thm42}. \vspace{0.2cm}

\noindent$(\textbf{A}')$. (\textbf{A}) holds and $\bm{r}$ is properly chosen so that there is $E\in \mathcal{Q}$ such that $\mcT E=E$. \vspace{0.2cm}

Another condition based on $(\textbf{A}')$ is called regular condition, under which, we will see the whole story happens on $C(K)$.\vspace{0.2cm}

\noindent(\textbf{R}). Assuming $(\textbf{A}')$, we say $(E, \bm{r})$ is \textit{regular} if $r_i<1$ for all $1\leq i\leq N$. \vspace{0.2cm}

\noindent\textbf{Remark.} Assuming $(\textbf{A}')$, by the later Lemma \ref{lemma54}, we can always find some $i$ such that $r_i<1$.

\begin{theorem}\label{thm51}
	Assume $(\textbf{A}')$. Let $E_0\in \mathcal{S}$ be defined as $E_0(f)=\frac{1}{2}\sum_{x\neq y}\big|f(x)-f(y)\big|^p, \forall f\in l(V_0)$.  Define
	\[\mathcal{F}=\{f\in l(V_*):\sup_{n\geq 1}\Lambda^nE_0(f|_{V_n})<\infty\}.\] Then there exists a subsequence $n_l,l\geq 1$ such that $\frac{1}{n_l+1}\sum_{m=0}^{n_l}\Lambda^mE_0(f|_{V_m})$ converges for every $f\in \mathcal{F}$. Define
	\[\mcE(f)=\lim\limits_{l\to\infty}\frac{1}{n_l+1}\sum_{m=0}^{n_l}\Lambda^mE_0(f|_{V_m}),\quad\forall f\in \mathcal{F}.\]
	Furthermore, we have the following:
	
	(a). Let $\sim$ be the equivalence relation on $\mathcal{F}$ defined by $f\sim g$ if and only if $f-g$ is a constant. Then $(\mathcal{F}/\sim,\mcE^{1/p})$ is a Banach space.
	
	(b). If (\textbf{R}) holds, then there exists $C>0$ such that for any $f\in\mathcal{F}$, for any $w\in W_*$,
\begin{equation*}
\Osc(f\circ F_w)\leq Cr_w^{1/p}\mcE(f\circ F_w)^{1/p},
\end{equation*}
Consequently, $\mathcal{F}$ embeds into $C(K)$ naturally (by continuous extension). In addition, $\mathcal{F}$ is dense in $C(K)$.
\end{theorem}

However, if (\textbf{R}) fails, we need to consider the following condition (\textbf{B}) which involves a given probability Radon  measure $\mu$ on $K$. \vspace{0.2cm}

\noindent(\textbf{B}). $\sum_{m=0}^\infty \kappa^{1/p}_m<\infty$, where  $\kappa_m=\kappa_m(\mu,\bm{r}):=\max_{w\in W_m}r_w\mu(F_wK)$.\vspace{0.2cm}

\noindent\textbf{Remark.}  (\textbf{R}) always implies (\textbf{B}). Even if (\textbf{R}) fails, there are a lot of $\mu$ satisfying (\textbf{B}), for example, a self-similar probability measure $\mu$ on $K$ satisfying $\mu(F_iK)r_i<1$ for each $1\leq i\leq N$, noticing that we can always find some $1\leq j\leq N$ such that $r_j<1$ by Lemma \ref{lemma54}.

\begin{theorem}\label{thm52}
	Assume $(\textbf{A}')$ and (\textbf{B}). Let $\mathcal{F}_c=\{f\in \mathcal{F}: f\text{ is uniformly continuous}\}$, and consider the natural embedding (by continuous extension) $\Psi:\mathcal{F}_c\to C(K)\subset L^p(K,\mu)$. Then $\Psi$ can be extended continuously to $\mathcal{F}\to L^p(K,\mu)$ in the sense that if $\{f_n\}\subset \mathcal F_c$, $f\in\mathcal F$, $f_n\to f$ pointwisely on $V_*$ and $\mcE(f_n-f)\to 0$, then $\Psi(f_n)\to \Psi(f)$ under $L^p(K,\mu)$ norm. In addition, $\Psi$ is self-similar:
	\[\Psi(f\circ F_w)=\Psi(f)\circ F_w,\quad \forall w\in W_*,f\in \mathcal{F}.\]
	
	By identifying $\mathcal F$ with $\Psi (\mathcal F)$,  $(\mathcal{E},\mathcal{F})$ becomes a closed $p$-energy form on $L^p(K,\mu)$ in the sense that $\mathcal{E}$ has the Markov property and is uniformly convex, and $(\mathcal{F},\mathcal{E}_1^{1/p})$ is closed and separable, where $\mathcal{E}_1$ is  defined by $\mathcal{E}_1(f)=\mathcal{E}(f)+\|f\|^p_{L^p(K,\mu)}$. In addition, $(\mathcal{E},\mathcal{F})$ is regular and strongly local in the following senses:\vspace{0.15cm} \\
	\emph{(regular).} Let $\mathscr{C}=\mathcal{F}\cap C(K)$. $\mathscr{C}$ is dense in $C(K)$, and $\mathscr{C}$ is dense in $\mathcal{F}$ with respect to the norm $\mathcal{E}_1^{1/p}$. \\
	\emph{(strongly local).} For $f,g\in \mathscr{C}$, if $f$ is constant on a neighbourhood of $\text{supp}(g)$, then $\mcE(f+g)=\mcE(f)+\mcE(g)$.
\end{theorem}

\noindent\textbf{Remark.} Theorem \ref{thm52} follows immediate from Theorem \ref{thm51} if (\textbf{R}) holds. Readers can skip Subsections \ref{subsec52},  \ref{subsec53} if they are only interested in the regular case.\vspace{0.2cm}

We will prove Theorem \ref{thm51} in Subsection \ref{subsec51}, and prove Theorem \ref{thm52} in Subsection \ref{subsec53}.

\subsection{The limit form}\label{subsec51}
The proof of Theorem \ref{thm51} is similar to the $p=2$ case. We will take the advantage of the existence of a fixed point of $\mcT$ guaranteed by Theorem \ref{thm42} under assumption ({\textbf{A}). For convenience, we simply write $E(f)$ instead of $E(f|_{V_n})$ for $E\in \widetilde{\mathcal{M}}_p(V_n)$ and $f\in l(V_*)$.

\begin{proposition}\label{prop53}
	Assume all the same conditions as in Theorem \ref{thm51}.  Let $E\in \mathcal{Q}$ be a fixed point of $\mcT$, i.e. $\mcT E=E$. Then,
	
	(a). For any $f\in l(V_*)$, we have $E(f)\leq \Lambda E(f)\leq \Lambda^2 E(f)\leq \cdots$. So the following functional $\mcE_*$ with extended real values is well defined,
	\[\mcE_*(f)=\lim\limits_{n\to\infty}\Lambda^nE(f),\quad\forall f\in l(V_*).\]
	
	(b). $\mathcal{F}=\{f\in l(V_*):\mathcal{E}_*(f)<\infty\}$.
	
	(c). $(\mathcal{F}/\sim,\mcE_*^{1/p})$ is a separable Banach space. In addition, if $\mcE_*(f_n-f)\to0$ as $n\to\infty$, then we can find constants $\{c_n\}$ so that $f_n-c_n$ converges to $f$ pointwisely on $V_*$.
	
	(d). If in addition (\textbf{R}) holds, then there exists $C>0$ such that for any $f\in\mathcal{F}$, for any $w\in W_*$,
\begin{equation*}
\Osc(f\circ F_w)\leq Cr_w^{1/p}\mcE_*(f\circ F_w)^{1/p}.
\end{equation*}
Consequently, $\mathcal{F}$ embeds densely into $C(K)$.

\end{proposition}

Before proving Proposition \ref{prop53}, we introduce the concept of piecewise $p$-harmonic functions. For $n\geq 0$, $f\in l(V_n)$, we define $H_*f\in l(V_*)$ by
\[H_*f|_{V_n}=f,\quad \mcE_*(H_*f)=\min\{\mcE_*(g):g\in l(V_*), g|_{V_n}=f\},\]
and call $H_*f$ a \textit{$p$-harmonic extension} of $f$ to $V_*$ with respect to $\mathcal E_*$. Note that $H_*f$ is unique by Lemma \ref{lemmaa2}. Strictly speaking, the operator $H_*$ depends on $n$, but we insist not using a superscript $n$ for brevity and it will cause no confusion.
Define $$\mathscr{H}=\big\{h=H_*f:f\in l(V_n),n\geq 0\big\},$$ call it the collection of \textit{piecewise $p$-harmonic functions} with respect to $\mathcal E_*$.

\begin{proof}[Proof of Proposition \ref{prop53}]
	(a) is trivial, (b) follows from the fact that $E\asymp E_0$.
	
	(c). First, $\mcE_*(f)>0$ for any $f\in \mathcal{F}\setminus Constants$, and it is straightforward to check $\mcE_*^{1/p}$ is a norm on $\mathcal{F}/\sim$. In fact, it is not hard to see that for each $n\geq 0$, $(\Lambda^n E)^{1/p}$ defines a seminorm since $\Lambda^n E\in \mathcal{Q}(V_n)$. So the limit $\mcE_*^{1/p}$ is also a seminorm.
	
	Next, let $f_n\in l(V_*), n\geq1$ and $\mcE_*(f_n-f_m)\to 0$ as $n,m\to \infty$, we show that there is a sequence of constants $c_n>0$ and $f\in \mathcal{F}$ so that
	\[
	\begin{cases}
		\mcE_*(f_n-f)\to 0,\\
		f_n-c_n\to f\text{ pointwisely}.
	\end{cases}
	\]
	In fact, we fix $x\in V_0$, and choose $c_n=f_n(x)$. Then, for each $l\geq 0$, we have  $\Lambda^lE(f_n-f_m)\to 0$ as $n,m\to \infty$, so $\Osc(f_n-f_m)\to 0$ as $n,m\to \infty$, hence $f_n-c_n$ converges pointwisely on $V_l$. Let $f\in l(V_*)$ be the pointwise limit of $f_n-c_n$.
	
	Then, we show $f\in \mathcal{F}$ and $\mcE_*(f_n-f)\to 0$. In fact, for each $l\geq 0$, $\Lambda^lE$ is continuous on $l(V_*)$ with respect to the topology of pointwise convergence, and thus the monotone limit $\mcE_*$ is lower-semicontinuous (with extended real values on $l(V_*)$). Then the claim follows since $\mcE_*(f-f_n)\leq \liminf_{m\to\infty}\mcE_*(f_m-f_n)$ by the lower-semicontinuity of $\mcE_*$.
	
	Finally, since $\mcE_*$ is the limit of $\Lambda^nE\in \mathcal{Q}_p(V_n)$ (in some reasonable sense), by Remark 2 after Definition \ref{def28}, one can easily see that $\mcE_*$ is uniformly convex. So for any $f\in\mathcal F$, noticing that $\mcE_*(H_*f|_{V_n})\to \mcE_*(f)$ as $n\to\infty$ and $\mcE_*\big(\frac{1}{2}(H_*f|_{V_n})+\frac12f\big)\geq \mcE_*(H_*f|_{V_n})$, we can see that $\mcE_*(H_*f|_{V_n}-f)\to 0$ as $n\to\infty$,  so $\mathscr{H}/\sim$ is dense in $(\mathcal{F}/\sim,\mcE_*^{1/p})$. We can choose a countable dense subset of $\mathscr{H}/\sim$, so $(\mathcal{F}/\sim,\mcE_*^{1/p})$ is separable.
	
	(d) follows from a routine argument of resistance estimate. For $x,y\in V_n$, define $R_{n,*}(x,y)$ to be the $p$-resistance associated with $\Lambda^nE$. Then $R_{n,*}(x,y)=R_{n+1,*}(x,y)=\cdots$, so for any $x,y\in V_*$, we can define $R_*(x,y)=R_{n,*}(x,y)$ providing $x,y\in V_n$ for some $n$.
	
	By the renormalization process, we know $R_*(F_wx,F_wy)\leq r_wR_*(x,y)$, $\forall w\in W_*,x,y\in V_*$.  By Proposition \ref{prop24} (c), $R_*^{1/p}$ is a metric on $V_*$, so by using a chaining argument and noticing that $r_i<1,\forall 1\leq i\leq N$, we have
	\[R_*(x,y)\leq C^{1/p},\quad\forall x,y\in V_*,\]
	for some $C>0$  independent of $x,y$. Hence $\Osc(f)\leq C\mcE_*^{1/p}(f)$ for any $f\in \mathcal{F}$.
	In particular, $\Osc(f\circ F_w)\leq C\mcE_*^{1/p}(f\circ F_w)\leq Cr^{1/p}_w\mcE_*^{1/p}(f)$, for any $w\in W_*$. \end{proof}

\begin{proof}[Proof of Theorem \ref{thm51}]
	By Proposition \ref{prop53}, we choose a countable dense subset $\tilde{\mathcal F}$ of $(\mathcal{F}/\sim,\mathcal{E}_*^{1/p})$ and by choosing a representative from each equivalent class, we get $\tilde{\mathcal{F}}\subset\mathcal{F}$. Then, noticing that $\frac{1}{n_l+1}\sum_{m=0}^{n_l}\Lambda^mE_0(f)\leq C_1\mcE_*(f)$ for each $f\in \mathcal{F}$ for some constant $C_1>0$ independent of $f$ and $l$, by a diagonal argument, we can pick a subsequence such that
	\[\mcE(f)=\lim\limits_{l\to\infty}\frac{1}{n_l+1}\sum_{m=0}^{n_l}\Lambda^mE_0(f)\]
exists for each $f\in \tilde{\mathcal F}$. Finally, since $\Lambda^nE_0\asymp \Lambda^n E$, one has $\mcE\asymp \mcE_*$, and hence we can extend the definition of $\mcE(f)$ by continuity to any $f\in \mathcal{F}$. Also, $\mcE^{1/p}$ is a seminorm since it is the limit of the seminorms $(\frac{1}{n_l+1}\sum_{m=0}^{n_l}\Lambda^mE_0)^{1/p}$.
	
	Since $\mcE\asymp \mcE_*$, (a) follows immediately from  Proposition \ref{prop53} (c); (b) follows immediately  from Proposition \ref{prop53} (d).
\end{proof}

\noindent\textbf{Remark. } Now that once Theorem \ref{thm51} is proved, we can simply let \[E=[\mcE]_{V_0}\]
where $[\mcE]_{V_0}(f)=\min\{\mcE(g):g\in\mathcal{F},g|_{V_0}=f\}, \forall f\in l(V_0)$. Then it is not hard to check $E\in \mathcal{Q}$ and $\mcT E=E$. So the temporary form $\mcE_*$ can be just replaced with $\mcE$, and in addition, $\mathscr{H}$ is the space of piecewise $p$-harmonic functions associated with $(\mcE,\mathcal{F})$. However, we will still keep the notation $ H_*$.\vspace{0.2cm}

Finally, similar to the $p=2$ case (see \cite[Proposition 3.1.8]{ki3}),  we should get an estimate $r_w<1$, if $\dot{w}=www\cdots\in \mathcal{P}$. This will help us to see that (\textbf{R}) holds on nested fractals for certain ``good'' $\bm{r}$. The following argument is due to \cite[Theorem 5.9]{HPS}.

\begin{lemma}\label{lemma54}
	Assume $(\textbf{A}')$, if $\dot{w}=www\cdots\in \mathcal{P}$, then $r_w<1$.
\end{lemma}
\begin{proof}
	Let $x=\pi(\dot{w})\in V_0$, and choose $m$ large enough so that $F^m_wK\cap V_0=\{x\}$. We let $h=H_*h'$ with $h'\in l(V_0)$ and $h'(y)=\delta_{x,y},\forall y\in V_0$. In other words, $h\in \mathcal{F}$ is a $p$-harmonic function with the boundary value $1_x$ on $V_0$.
	
	Next, we define some functions with the same boundary values as $h$ on $V_0$. First, we let
	$f=H_*f'$, where $f'\in l(V_{m|w|})$ is defined as $f'(y)=\delta_{x,y},\forall y\in V_{m|w|}$. So by self-similarity, we immediately see
	\begin{equation}\label{eqn51}
		\mcE(f)=r_w^{-m}\mcE(h).
	\end{equation}
	Define $g'\in l(V_{m|w|})$ as
	\[g'(y)=\begin{cases}
		1,\text{ if }y\in F^m_wV_0,\\
		0,\text{ if }y\in V_{m|w|}\setminus (F^m_wV_0).
	\end{cases}\]
Let $g=H_*g'$.
	Then, since $f|_{K\setminus F_w^mK}\equiv 0$ and $g|_{F_w^mK}\equiv 1$, by using the strongly local property and the homogeneity of $\mcE$, it is not hard to see
	\[
	\begin{aligned}
		u(t):=\mcE((1-t)f+tg)=(1-t)^p\mcE(f)+t^p\mcE(g), \quad {\forall}\ 0<t<1.
	\end{aligned}
	\]
	Since $\mcE(f)>0$, we can see that $\frac{d}{dt}u(0)<0$, and hence
	\[\mcE(h)\leq \min_{t\in [0,1]}\mcE((1-t)f+tg)<\mcE(f)=r_w^{-m}\mcE(h),\]
	where the first inequality is because $h|_{V_0}=\big((1-t)f+tg\big)|_{V_0}$ and $h$ is $p$-harmonic. Clearly, $\mcE(h)=E(h')>0$, so $r_w<1$.
\end{proof}

\subsection{Embedding $\mathscr{H}$ into $C(K)$}\label{subsec52}
In this subsection, we show that $H_*f$ satisfies some estimate of local oscillation, hence one can embed $\mathscr{H}$ into $C(K)$ by continuous extension. In particular, $\mathscr{H}$ is a dense subspace of $C(K)$, and we can see $\mathscr{H}/\sim$ is dense in $\mathcal{F}/\sim$ with respect to $\mcE^{1/p}$.

The proof is based on a similar idea of Lemma \ref{lemma54}, but much more complicated. The essential difficulty is solved by Lemma \ref{lemmaa3} in Appendix \ref{AppendixA}.

\begin{proposition}\label{prop55}
	Assume $(\textbf{A}')$. There exists $m\geq 1$ and $\eta<1$ so that $\Osc(H_*f\circ F_w)<\eta\Osc(f)$ for any non-constant function $f\in l(V_0)$ and $w\in W_m$.
\end{proposition}
\begin{proof}
	We choose some fixed $m\geq 1$ so that $\#(F_wV_0)\cap V_0\leq 1$ for each $w\in W_m$.
	
	First, let us fix a non-constant function $f\in l(V_0)$, and complete the proof by taking advantage of compactness. For convenience, we write $h_f=(H_*f)|_{V_m}$, we then claim that
	\begin{equation}\label{eqn52}
\max_{w\in W_m}\Osc(h_f\circ F_w)<\Osc(f).
	\end{equation}
We will show this by contradiction. Indeed, if it is not true, we must have $\max_{w\in W_m}\Osc(h_f\circ F_w)=\Osc(f)$. So we can find $\tau\in W_m$ such that $\Osc(h_f\circ F_\tau)=\Osc(f)$, and hence there exist $A,B\subset V_0$ such that
	\[
	h_f\circ F_\tau|_A=\min_{x\in V_0}f(x),\quad h_f\circ F_\tau|_B=\max_{x\in V_0}f(x).
	\]
By assumption, we have $\#((F_\tau V_0)\cap V_0)\leq 1$, hence one of $F_\tau(A)$ and $F_\tau(B)$ must not intersect $V_0$. Without loss of generality, we may assume that $F_\tau(B)\cap V_0=\emptyset$ (otherwise consider $-f$). Then, we define $h_{f,t}=h_f+t\cdot 1_{F_\tau B}$ for $t\in\mathbb R$, where $1_{F_\tau B}\in l(V_m)$ is the indicator function of $F_\tau B$. We denote by $\frac{d}{dt-}$ the left derivative, i.e. $\frac{d}{dt-}u|_{t=0}=\lim\limits_{t\nearrow 0}\frac{u(t)-u(0)}{t}$. Then
	\begin{equation}\label{equation5.3}
	\frac{d}{dt-}\Lambda^mE(h_{f,t})|_{t=0}=\sum_{w\in W_m}r_w^{-1}\frac{d}{dt-}E(h_{f,t}\circ F_w)|_{t=0}.
	\end{equation}
Since $h_f\circ F_\tau|_B=\max_{x\in V_0} h_f(x)=\max_{x\in V_m} h_f(x)$, by applying Lemma \ref{lemmaa3}, one can see that each term in the right hand side of \eqref{equation5.3} is non-negative and in particular

\[\frac{d}{dt-}E(h_{f,t}\circ F_\tau)|_{t=0}>0,\]
thus we obtain $\frac{d}{dt-}\Lambda^mE(h_{f,t})|_{t=0}>0$. Hence, there is $t<0$ such that $\Lambda^mE(h_{f,t})<\Lambda^mE(h_{f,0})=\Lambda^mE(h_f)=E(f)$. This is a contradiction since $h_f$ is a $p$-harmonic extension of $f$ and $h_{f,t}|_{V_0}=f$.
	
	To complete the proof, we notice that $\max_{w\in W_m}\Osc(H_*f\circ F_w)$ is continuous on $l(V_0)$ by Lemma \ref{lemmaa2}. Hence, by (\ref{eqn52}), we have
	\[\max_{f\in M}\max_{w\in W_m}\Osc(H_*f\circ F_w)<1,\]
	where $M=\{f\in l(V_0):\Osc(f)=1,\sum_{x\in V_0}f(x)=0\}$ is a compact subset of $l(V_0)$.
\end{proof}

\begin{corollary}\label{coro56}
	Assume $(\textbf{A}')$. We have the natural embedding $\mathscr{H}\subset C(K)$.
\end{corollary}

\subsection{Embedding $\mathcal{F}$ into $L^p(K,\mu)$}\label{subsec53}
The idea is essentially due to Kumagai \cite{ku}. Readers can also find the proof in the book \cite[Section 3.4]{ki3}.

By Corollary \ref{coro56}, for $f\in l(V_*)$, $m\geq 0$, we can naturally embed $H_*(f|_{V_m})$ into $L^p(K,\mu)$. Indeed, $\Psi(H_*(f|_{V_m}))\in C(K)\subset L^p(K,\mu)$, where $\Psi$ is the same operator in  the statement of Theorem \ref{thm52}. For short, we write
	\[P_mf=\Psi(H_*(f|_{V_m})), \quad f\in l(V_*), m\geq 0.\]

\begin{lemma}\label{lemma57}
	Assume $(\textbf{A}')$ and (\textbf{B}). Then
	$P_mf$ converges in $L^p(K,\mu)$ as $m\to\infty$ for any $f\in \mathcal F$.
\end{lemma}
\begin{proof}
	Fix $m\geq 0$, and let $f_{m+1}=(P_mf)|_{V_{m+1}}$ and $g_{m+1}=f|_{V_{m+1}}$. Then \[\sum_{w\in W_m}r_w^{-1}\Lambda E(f_{m+1}\circ F_w-g_{m+1}\circ F_w)=\Lambda^{m+1} E(f_{m+1}-g_{m+1})\leq 2^p\mcE(f).\]
	Hence, recall that $\kappa_m=\max_{w\in W_m}r_w\mu(F_wK)$, there is a constant $C>0$ independent of $f$ such that
	\begin{equation}\label{eqn53}
		\begin{aligned}
			\|P_{m+1}f-P_mf\|^p_{L^p(K,\mu)}
			&\leq \sum_{w\in W_m}\mu(F_wK)\|P_{m+1}f-P_mf\|^p_{L^\infty(F_wK,\mu)}\\
			&=\sum_{w\in W_m}\mu(F_wK)\|f_{m+1}\circ F_w -g_{m+1}\circ F_w\|^p_{l^\infty(V_1)}\\
			&\leq C\sum_{w\in W_m}\mu(F_wK)\Lambda E(f_{m+1}\circ F_w -g_{m+1}\circ F_w)\\
			&\leq 2^pC\kappa_m\mcE(f),
		\end{aligned}
	\end{equation}
	where we use Lemma \ref{lemmaa2} in the equality. Hence by (\textbf{B}),
	\begin{equation}\label{eqn54}
		\sum_{m=0}^\infty\|P_{m+1}f-P_mf\|_{L^p(K,\mu)}\leq 2C^{1/p}\left(\sum_{m=0}^\infty \kappa_m^{1/p}\right)\mcE^{1/p}(f)<\infty,
	\end{equation}
	whence $P_{m}f$ converges in $L^p(K,\mu)$ as $m\rightarrow\infty$.
\end{proof}

By Lemma \ref{lemma57}, we can extend $\Psi$ to $\mathcal{F}\to L^p(K,\mu)$ by defining $\Psi (f)=\lim\limits_{m\to\infty}P_mf$.

Then for $f$ uniformly continuous, $P_mf$ converges to $f$ uniformly, hence $\Psi(f)$ is simply the continuous extension of $f$, and the notation is in consistency with that in the statement of Theorem \ref{thm52}.

\begin{lemma}\label{lemma58}
	Assume $(\textbf{A}')$ and (\textbf{B}).
	The extended embedding operator $\Psi:\ \mathcal F\rightarrow L^p(K,\mu)$ is continuous in the sense that if $f_n\to f$ pointwisely and $\mcE(f_n-f)\to 0$, then $\Psi(f_n)\to \Psi(f)$ in $L^p(K,\mu)$. Hence $\Psi$ is linear. Furthermore, $\Psi:\mathcal{F}\to L^p(K,\mu)$ is injective.
\end{lemma}
\begin{proof}
	We first prove the continuity of $\Psi$. Assume $f_n|_{V_*}\to f|_{V_*}$ and $\mcE(f_n-f)\to 0$ as $n\rightarrow\infty$. Then, by (\ref{eqn53}) and Lemma \ref{lemmaa2}, for any $m\geq 1$,
	\[\begin{aligned}
		\|\Psi (f_n)-\Psi (f)\|_{L^p(K,\mu)}&\leq \|P_m f_n-P_m f\|_{L^p(K,\mu)}+\|P_m f-\Psi (f)\|_{L^p(K,\mu)}+\|P_m f_n-\Psi (f_n)\|_{L^p(K,\mu)}\\
		&\leq \|f_n-f\|_{l^\infty(V_m)}+C_1(\sum_{m'=m}^\infty \kappa_{m'}^{1/p})\big(\mcE^{1/p}(f)+\mcE^{1/p}(f_n)\big),
	\end{aligned}\]
for some $C_1>0$.	Hence, $\Psi(f_n)\to \Psi(f)$ in $L^p(K,\mu)$.
	
	Then we prove the injectivity of $\Psi$. The proof is due to \cite{ku}, see also \cite[ Lemma 3.4.4]{ki3}. The key observation is the following claim. \vspace{0.2cm}
	
	\noindent\textit{Claim}. If $f\in \mathcal{F}$ and $\Psi(f)=0$, then $\|f|_{V_0}\|_{l^\infty(V_0)}\leq C_2\mcE^{1/p}(f)$ for some $C_2>0$ independent of $f$.\vspace{0.2cm}
	
	On one hand, by (\ref{eqn53}), we know that there is $C_3>0$ such that
	\[\|P_0f\|_{L^p(K,\mu)}=\|\Psi(f)-P_0f\|_{L^p(K,\mu)}\leq C_3\mcE^{1/p}(f).\]
	On the other hand, noticing that by Lemma \ref{lemmaa2}, $\Psi\circ H_*$ is continuous on $l(V_0)$. Hence, by letting $C_4=\max\{\|P_0 f'\|_{L^p(K,\mu)}^{-1}:f'\in l(V_0),\|f'\|_{l^\infty(V_0)}=1\}$, we have
	\[\|f|_{V_0}\|_{l^\infty(V_0)}\leq C_4\|P_0f\|_{L^p(K,\mu)}.\]
	The claim follows immediately from the above two estimates.\vspace{0.2cm}
	
	To finish the proof, it suffices to apply Lemma \ref{lemma54}. For each $x\in V_*$, there is $\tau,w\in W_*$ so that $x=\pi(\tau\dot{w})$. Hence, if $f\in \mathcal{F}$ and $\Psi(f)=0$, using the claim on $F_{\tau w^m}(V_0)$, for any $m\geq 1$, we have
	\[|f(x)|\leq C_2\mcE^{1/p}(f\circ F_{\tau}\circ F_{w^m})\leq C_2(r_\tau r^m_w)^{1/p}\mcE^{1/p}(f).\]
	 Since $r_w<1$ by Lemma \ref{lemma54}, we have $f(x)=0$, whence $f=0$ since the argument works for any $x\in V_*$.
\end{proof}

\begin{proof}[Proof of Theorem \ref{thm52}]
	By Lemma \ref{lemma58}, $\Psi(\mathcal{F})$ is a subspace of $L^p(K,\mu)$. Since $P_m(f\circ F_w)=(P_{m+|w|}f)\circ F_w$ for any $w\in W_*$ and $m\geq 0$, we have $\Psi (f\circ F_w)=(\Psi f)\circ F_w$.\vspace{0.2cm}

	\noindent(\textit{Closedness}). Let $f_n\in \mathcal{F},n\geq 1$ and assume $\mcE_1(\Psi (f_n)-\Psi (f_m))\to 0$ as $n,m\to \infty$, we need to show there is $ f\in \mathcal{F}$ so that  $\mcE_1(\Psi (f_n)-\Psi (f))\to 0$.
	
	Since $\Psi$ is linear by Lemma \ref{lemma58}, we see that $\|\Psi(f_n-f_m)\|_{L^p(K,\mu)}\to 0$ as $n,m\to\infty$. In addition, by (\ref{eqn53}), we also see that $\|(\Psi-P_0)(f_n-f_m)\|_{L^p(K,\mu)}\to 0$. Hence, $\|f_n-f_m\|_{l^\infty(V_0)}\leq C_1\|P_0(f_n-f_m)\|_{L^p(K,\mu)}\to 0$ as $n,m\to\infty$.
	
	Then, we apply Proposition \ref{prop53} (c) to see that there is an $f\in \mathcal{F}$ so that $f_n\to f$ pointwisely, and $\mathcal{E}(f_n-f)\to 0$. Here we can choose $c_n=0$ in Proposition \ref{prop53} since $f_n|_{V_0}$ converges pointwisely. Hence by Lemma \ref{lemma58}, $\|\Psi (f_n)-\Psi (f)\|_{L^p(K,\mu)}\to 0$. \vspace{0.2cm}
	
	\noindent(\textit{Regularity}). Since by Corollary \ref{coro56}, $\mathscr{H}\subset\mathscr{C}:=\mathcal{F}\cap C(K)$, so $\mathscr{C}$ is dense in $C(K)$, and in addition, for each $f\in \mathcal F$, we have $\mathcal{E}_1(P_m(f)-\Psi (f))\to 0$ as $m\to \infty$, which implies that $\mathscr{C}$ is dense in $\mathcal{F}$ with respect to $\mcE_1^{1/p}$. \vspace{0.2cm}

	\noindent(\textit{Markov property}). Write $\bar{f}=(f\wedge 1)\vee 0$ for short. It suffices to show that $\Psi(\bar{f})=\overline{\Psi(f)},\forall f\in \mathcal{F}$. To see this, for any $f\in \mathcal{F}$, we choose a sequence $f_n\in \mathcal{F}$ so that $\Psi(f_n)\in C(K)$ and $\mcE_1(\Psi (f_n)-\Psi (f))\to 0$. By continuity of $\Psi(f_n)$, we immediately have $\bar{f}_n=\overline{\Psi(f_n)}$. In addition, one can check
	\begin{eqnarray}
		\label{eqn55}&\|\overline{\Psi(f)}-\overline{\Psi(f_n)}\|_{L^p(K,\mu)}\leq \|\Psi(f)-\Psi(f_n)\|_{L^p(K,\mu)}\to 0,\\
		\label{eqn56}&\mathcal{E}(\bar{f_n}-\bar{f})\to 0,\text{ and }\bar{f}_n\to \bar{f}\text{ pointwisely.}
	\end{eqnarray}
	In particular, by (\ref{eqn56}) and Lemma \ref{lemma58}, we know that $\overline{\Psi(f_n)}=\Psi(\bar{f}_n)\to \Psi(\bar{f})$ in $L^p(K,\mu)$. Hence, by (\ref{eqn55}), we have $\overline{\Psi(f)}=\Psi(\bar{f})$. \vspace{0.2cm}
	
	\noindent(\textit{Strong locality}). It follows from a routine argument, noticing that if $f,g\in \mathscr{C}$ and $f$ is constant on an neighbourhood of support of $g$, then $\Lambda^n E(f+g)=\Lambda^nE(f)+\Lambda^nE(g)$ for any $n$ large enough.
\end{proof}

\section{Affine nested fractals}\label{sec6}
In this section, we prove the existence of $p$-energy for any $p\in(1,\infty)$ on a class of highly symmetric p.c.f. fractals, namely the affine nested fractals. This class of fractals is firstly introduced in \cite{FHK}, which is a generalization of nested fractals introduced by Lindstr{\o}m \cite{Lindstrom} in that the contraction ratios of the i.f.s. are allowed to be distinct.\vspace{0.2cm}

The nested fractals are featured with the reflection symmetries interchanging essential fixed points. For convenience, we introduce the following notations for reflections in this section.

\begin{definition}\label{def61}
Let $x,y\in \mathbb{R}^d$ and $x\neq y$. Define
	\[H_{x,y}=\{z\in \mathbb R^d:\ |x-z|=|y-z|\},\]
and denote the reflection with respect to $H_{x,y}$ by $\sigma_{x,y}$. In addition, for convenience, we write
\[H_x^y=\{z\in \mathbb R^d:\ |x-z|\leq |y-z|\}\]
for the closed half space containing $x$. Similarly, $H_y^x=\{z\in \mathbb R^d:\ |y-z|\leq |x-z|\}$.
\end{definition}

In the following, we take the short definition of affine nested fractals from the book \cite[Section 3.8]{ki3}, noticing that we assume $K\subset \mathbb{R}^d$ and $F_i$ are similitudes for simplicity.

\begin{definition}\label{def62}
Let $K$ be a p.c.f. self-similar set associated with the i.f.s. $\{F_i\}_{i=1}^N$. We say $(K,\{F_i\}_{i=1}^N)$ is an affine nested fractal if $(K,\{F_i\}_{i=1}^N)$ has $\mathscr{G}$-symmetry, where $\mathscr{G}$ is the symmetry group generated by $\{\sigma_{x,y}:\ x\neq y\in V_0\}$.
\end{definition}

Our main result in this section is the following.
\begin{theorem}\label{thmaffnfs}
Let $(K,\{F_i\}_{i=1}^N)$ be an affine nested fractal. If the renormalization factor $\bm{r}$ is $\mathscr{G}$-symmetric, then condition (\textbf{A}) holds.
\end{theorem}

We refer to \cite[Section 3.8]{ki3} for basic properties of affine nested fractals, and an alternative proof of the existence of Dirichlet forms on nested fractals (essentially due to \cite{Lindstrom} and \cite{FHK}). In particular, we need an easy geometric fact about affine nested fractals (\cite[Lemma 3.8.6]{ki3}, see also \cite[IV.9 Lemma]{Lindstrom} for nested fractals), which we state as Lemma \ref{lemma64} below.\vspace{0.2cm}

Consider the distance set of points in $V_0$: $\{|x-y|:\ x,y\in V_0,x\neq y\}$, and we arrange them in order as $\ell_0<\ell_1<\cdots<\ell_m$, i.e. $\ell_0=\min\{|x-y|:\ x,y\in V_0,x\neq y\}$ is the minimal distance between two points in $V_0$, and $\ell_m$ is the maximal distance.
For any two points $x,y\in V_0$, a sequence $x=x_0,x_1\cdots,x_k=y$ in $V_0$ such that $|x_i-x_{i+1}|=\ell_0$ for $i=0,1,\cdots,k-1$ is called a \textit{strict $0$-walk} between $x$ and $y$.
\begin{lemma}\label{lemma64}
	Let $(K, \{F_i\}_{i=1}^N,V_0)$ be an affine nested fractal. Then there exists a strict $0$-walk between $x$ and $y$ for any $x,y\in V_0$, $x\neq y$.
\end{lemma}

\begin{proof}[Proof of Theorem \ref{thmaffnfs}] Let $E\in \mathcal{S}$ be defined as
\[E(f)=\frac{1}{2}\sum_{x,y\in V_0: |x-y|=\ell_0}\big|f(x)-f(y)\big|^p,\quad \forall f\in l(V_0).\]
For $n\geq 0$, let $R_n$ be the $p$-resistance associated with $\Lambda^n E$. Let $p_1,p_2\in V_0$ be such that $|p_1-p_2|=\ell_0$. We claim that
\begin{equation}\label{eqRcom}
	\frac{1}{2}R_n(p_1,p_2)\leq R_n(x,y)\leq N^p R_n(p_1,p_2)\qquad \forall x\neq y\in V_0,\forall n\geq 0.
\end{equation}
As a consequence, $\inf_{n\geq 0}\delta(\mcT^nE)\geq \frac 12 N^{-p}$, and hence condition (\textbf{A}) holds. The upper bound in (\ref{eqRcom}) will be proved by Lemma \ref{lemma64}, and the lower bound in (\ref{eqRcom}) can be proved by using an analogous reflection principle.\vspace{0.2cm}

``\textit{Upper bound}''. Indeed, for any $x,y\in V_0$, $x\neq y$, let $x=x_0,x_1\cdots,x_k=y$ be a $0$-walk in $V_0$, whose existence is guaranteed by Lemma \ref{lemma64}. Note that $k\leq N$. Then by the triangle inequality of $R_n^{1/p}$ (Proposition \ref{prop24}(c)), we obtain
\begin{equation*}
	R_n(x,y)^{1/p}\leq \sum_{i=0}^{k-1}R_n(x_i,x_{i+1})^{1/p}=kR_n(p_1,p_2)^{1/p}\leq NR_n(p_1,p_2)^{1/p},
\end{equation*}
which gives the upper bound estimate in \eqref{eqRcom}.\vspace{0.15cm}

``\textit{Lower bound}''. If $|x-y|=\ell_0$, then $R_n(x,y)=R_n(p_1,p_2)$, so we have nothing to prove. Thus, we assume $|x-y|=\ell_k$ for some $1\leq k\leq m$. We let $z\in V_0$ so that $|x-z|=\ell_0$.

Let $n\geq 0$.  We define $f\in l(V_n)$ such that
\[\begin{cases}
	f(x)=1,\ f(z)=0,\\
	\Lambda^nE(f)=\min\{\Lambda^n E(g):\ g\in l(V_n),\ g(x)=1,g(z)=0\}.
\end{cases}
\]
Hence $\Lambda^nE(f)=\big(R_n(x,z)\big)^{-1}$. One can then define $h\in l(V_n)$ by
\[h(q)=
\begin{cases}
	f(q),&\text{ if }q\in H_z^y\cap V_n,\\
	f\circ\sigma_{y,z}(q),&\text{ if }q\in H_y^z\cap V_n.
\end{cases}
\]
In particular, since $|x-z|=\ell_0<\ell_k=|x-y|$, we have $x\in H_z^y$ hence $h(x)=1$. In addition, $h(y)=h(\sigma_{y,z}(z))=0$. So we also have $\big(R_n(x,y)\big)^{-1}\leq \Lambda^n E(h)$.

Finally, noting that we collect a term $|f(p)-f(q)|^p$ in $E(f)$ if and only if $|p-q|=\ell_0$, and hence the contributions of terms $|h(p)-h(q)|^p$ in $\Lambda^nE(h)$ for $p\in H_y^z\setminus H_{y,z}$ and $q\in H_z^y\setminus H_{y,z}$ is $0$, so $\Lambda^n E(h)\leq 2\Lambda^n E(f)$. Combining all the information, we get
\[\big(R_n(x,y)\big)^{-1}\leq \Lambda^n E(h)\leq 2\Lambda^n E(f)=2\big(R_n(x,z)\big)^{-1}=2\big(R_n(p_1,p_2)\big)^{-1},\]
which gives the lower bound estimate in \eqref{eqRcom}.
\end{proof}

\section{Behavior of nearly degenerate $p$-energies}\label{sec7}
In Theorem \ref{thmaffnfs}, we obtain a uniform positive lower bound for $\delta(\Lambda^n E)$ on affine nested fractals by using the strong symmetry. On general p.c.f. self-similar sets, such an estimate is either not always true or hard to be established. Instead, we follow the idea of Sabot \cite{Sabot} to show that under suitable conditions, $\mcT$ will bounce near the boundary $\{E\in\widetilde{\mathcal{M}}:\delta(E)=0\}$, in other words when $\delta(E)$ is very small.

Before talking about Sabot's criteria, we return to study the general forms $E\in \widetilde{\mathcal{M}}_p(A)$ on a finite set $A$ when $\delta(E)$ is very small. The key idea, due to \cite{Sabot}, is to describe the distribution of the resistances using the equivalence relations. We will drop the (H) condition that involved in \cite[page 649]{Sabot}, so we need some results considering all the equivalence relations at the same time (see Lemma \ref{lemma73}, Corollary \ref{coro74} and Proposition \ref{prop75}).

\subsection{Equivalence relations}
Throughout this paper, we will use $\mcJ$ to denote an equivalence relation on a finite set $A$, and write $A/\mcJ$ for the set of equivalence classes. We will use the notation $I$ for an equivalence class of $\mcJ$, i.e. $I\in A/\mcJ$, and notice that $I\subset A$. We say that $\mcJ$ is \textit{trivial} if it is one of the following two relations on $A$: $\mcJ=0$ if $\#(A/\mcJ)=\#A$; $\mcJ=1$ if $\#(A/\mcJ)=1$; otherwise $\mcJ$ is called \textit{non-trivial}. In addition, for two relations $\mcJ,\mcJ'$, we say $\mcJ\subset \mcJ'$ if
\begin{equation*}
	x\mcJ y\Longrightarrow x\mcJ'y\qquad\text{for any $x,y\in A$}.
\end{equation*}
We introduce $\delta_{\mcJ,\mcJ'}(\Ep)$ to provide more information than $\delta(\Ep)$.

\begin{definition}\label{def71}
	Let $\mcJ,\mcJ'$ be two non-trivial equivalence relations on $A$ such that $\mcJ\subset \mcJ'$, and let $\Ep\in \mathcal{M}_p(A)$ with associated $p$-resistance $R^{(p)}$. We define
	\[
	\delta_{\mcJ,\mcJ'}(\Ep)={\frac{\max_{x\mcJ y} R^{(p)}(x,y)}{\min_{x\mcJ\mkern-10.5mu\backslash' y} R^{(p)}(x,y)}}.
	\]
	In particular, when $\mcJ=\mcJ'$, we simply write $\delta_{\mcJ}(\Ep)=\delta_{\mcJ,\mcJ}(\Ep)$.
\end{definition}

Of course, we are interested in the case that $\delta_\mcJ(\Ep)<1$. The following lemma explains why we only care about the case $\mcJ\subset \mcJ'$ in Definition \ref{def71}.
\begin{lemma}\label{lemma72}
	Let $A$ be a finite set and $\Ep\in\mathcal M_p(A)$.  Let $\mcJ,\mcJ'$ be two non-trivial equivalence relations such that $\delta_\mcJ(E)<1$ and $\delta_{\mcJ'}(E)<1$, then either $\mcJ\subset \mcJ'$ or $\mcJ'\subset \mcJ$.
\end{lemma}
\begin{proof}
	Assume neither $\mcJ\subset\mcJ'$ nor $\mcJ'\subset\mcJ$ is true, then one can find $x,y$ and $x',y'$ such that
	\begin{eqnarray}
		\label{eqn71}	x\mcJ y,\quad &x'\mcJ\mkern-10.5mu\backslash y',\\
		\label{eqn72}	x\mcJ\mkern-10mu\backslash' y,\quad&x'\mcJ'y'.
	\end{eqnarray}
	Then, since $\delta_{\mcJ}<1$, $\frac{R^{(p)}(x,y)}{R^{(p)}(x',y')}<1$ by (\ref{eqn71}); since $\delta_{\mcJ'}<1$, $\frac{R^{(p)}(x',y')}{R^{(p)}(x,y)}<1$ by (\ref{eqn72}). Clearly, this is impossible.
\end{proof}

By a similar argument as in the proof of Corollary \ref{coro27}, we can prove the following lemma.

\begin{lemma}\label{lemma73}
	Let $A$ be a finite set, $\Ep\in\mathcal M_p(A)$ and $R^{(p)}$ be the effective resistance associated with $\Ep$. Let $0<\delta<(\#A)^{-\#A(\#A-1)p/2}$.
	
	(a). If $\delta(\Ep)<\delta$, then there is a non-trivial equivalence relation $\mcJ$ such that $\delta_\mcJ(\Ep)<\delta^{\frac{2}{\#A(\#A-1)}}\cdot (\#A)^p$.
	
	(b). More generally, let $\mcJ,\mcJ'$ be two equivalence relations (can be trivial) such that $\mcJ\subsetneq \mcJ'$, $\delta_{\mcJ}(\Ep)<1$ if $\mcJ$ is non-trivial, and $\delta_{\mcJ'}(\Ep)<1$ if $\mcJ'$ is non-trivial.  If
	\begin{equation}\label{eq73}
	{\frac{\min\limits_{x\mcJ\mkern-10.5mu\backslash y, x\mcJ' y} R^{(p)}(x,y)}{\max\limits_{x\mcJ\mkern-10.5mu\backslash y, x\mcJ' y} R^{(p)}(x,y)}}{<\delta},\end{equation}
	then one can find $\tilde{\mcJ}$ such that  $\mcJ\subsetneq\tilde{\mcJ}\subsetneq \mcJ'$ and $\delta_{\tilde{\mcJ}}(\Ep)<\delta^{\frac{2}{\#A(\#A-1)}}\cdot (\#A)^p$.
\end{lemma}
\begin{proof}
	We first prove (b). 
%
%
%
%
By (\ref{eq73}), one can find a real $r>0$ such that
	\begin{eqnarray}
		&\label{eqn73}\text{either }R^{(p)}(x,y)< r\text{, or }R^{(p)}(x,y)> \delta^{-\frac{2}{\#A(\#A-1)}}r,\quad\forall x,y\text{ such that }x\mcJ\mkern-10.5mu\backslash y, x\mcJ' y,\\
		\label{eqn74} &\{(x,y):x\mcJ\mkern-10.5mu\backslash y, x\mcJ' y,R^{(p)}(x,y)<r\}\neq \emptyset,\\
		\label{eqn75} &\{(x,y):x\mcJ\mkern-10.5mu\backslash y, x\mcJ' y,R^{(p)}(x,y)> \delta^{-\frac{2}{\#A(\#A-1)}}r\}\neq \emptyset,
	\end{eqnarray}
	since there are at most $\#A(\#A-1)/2$ different values of $R^{(p)}(x,y)$. In addition, by the assumption $\delta_{\mcJ}(\Ep)<1$ and (\ref{eqn74}), we know $R^{(p)}(x,y)<r$ for any $x\mcJ y$; by the assumption $\delta_{\mcJ'}(\Ep)<1$ and (\ref{eqn75}), we know $R^{(p)}(x,y)>\delta^{-\frac{2}{\#A(\#A-1)}}r$ for any $x\mcJ\mkern-10.5mu\backslash' y$. Hence, (\ref{eqn73}) can be strengthened to be
	\begin{equation}\label{eqn76}
		\text{either }R^{(p)}(x,y)< r\text{, or }R^{(p)}(x,y)> \delta^{-\frac{2}{\#A(\#A-1)}}r,\quad\forall x\neq y\in A.
	\end{equation}
	
	One can define an equivalence relation {$\tilde{\mathcal{J}}$} on $A$ so that $x\tilde\mcJ y$ if and only if there exists a sequence of distinct points $x=z_0,z_1,\cdots,z_m=y$ such that
	\[
	R^{(p)}(z_i,z_{i+1})< r,\quad\forall 0\leq i\leq m-1.
	\]
	Also by (\ref{eqn76}), $R^{(p)}(x,y)> \delta^{-\frac{2}{\#A(\#A-1)}}r$ if $x\tilde\mcJ\mkern-10.5mu\backslash y$. On the other hand, by a same estimate as in the proof of Corollary \ref{coro27}, we have
	\begin{equation}\label{equation7.7}
	R^{(p)}(x,y)<(\#A)^pr,\quad \forall x\neq y, x\tilde\mcJ y.
	\end{equation}
	Hence, $\delta_{\tilde\mcJ}(\Ep)<\delta^{\frac{2}{\#A(\#A-1)}}\cdot (\#A)^p$.
	
	It remains to show $\mcJ\subsetneq \tilde{\mcJ}\subsetneq \mcJ'$. On one hand, since $x\mcJ y$ implies $R^{(p)}(x,y)<r$, we have $\mcJ\subset\tilde{\mcJ}$. On the other hand, by the assumption that $\delta<(\#A)^{-\#A(\#A-1)p/2}$ and \eqref{equation7.7}, we have for $x\tilde\mcJ y$,
	\begin{equation}\label{eq79}
	R^{(p)}(x,y)<(\#A)^pr<\delta^{-\frac{2}{\#A(\#A-1)}}r,
	\end{equation}
	 which implies $x\mcJ' y$, and hence $\tilde{\mcJ}\subset \mcJ'$. Finally, we see that $\mcJ\neq \tilde{\mcJ}$ by (\ref{eqn74}) and $\tilde\mcJ\neq \mcJ'$ by (\ref{eqn75}) and (\ref{eq79}).
	
	(a) follows from (b) by letting $\mcJ=0$ and $\mcJ'=1$.
\end{proof}

We can restate Lemma \ref{lemma73} (b) in the following way.

\begin{corollary}\label{coro74}
	Let $A$ be a finite set, $\Ep\in\mathcal M_p(A)$ and $R^{(p)}$ be the effective resistance associated with $\Ep$. Let $\mcJ,\mcJ'$ be two equivalence relations (can be trivial) such that $\mcJ\subsetneq \mcJ'$, $\delta_{\mcJ}(\Ep)<1$ if $\mcJ$ is non-trivial, and $\delta_{\mcJ'}(\Ep)<1$ if $\mcJ'$ is non-trivial.
	
	If for some $\delta\in (0,1)$, there does not exist $\tilde{\mcJ}$ such that $\mcJ\subsetneq\tilde{\mcJ}\subsetneq \mcJ'$ and $\delta_{\tilde{\mcJ}}(\Ep)<\delta$, then
	\[{\frac{\min\limits_{x\mcJ\mkern-10.5mu\backslash y, x\mcJ' y} R^{(p)}(x,y)}{\max\limits_{x\mcJ\mkern-10.5mu\backslash y, x\mcJ' y} R^{(p)}(x,y)}}\geq \delta^{\frac{\#A(\#A-1)}{2}}\cdot (\#A)^{-\#A(\#A-1)p/2}.\]
\end{corollary}

We conclude this subsection with the following proposition.
\begin{proposition}\label{prop75}
	Let $A$ be a finite set and $\Ep\in\mathcal M_p(A)$. Let $0<\delta<1$ and assume there is a non-trivial equivalence relation $\mcJ$ so that $\delta_\mcJ(\Ep)<\delta$.
	
	Let $\{\mcJ_m\}_{m=1}^M$ be the set of all non-trivial equivalence relations satisfying $\delta_{\mcJ_m}(\Ep)<\delta$. Then
	
	(a). We can order $\mcJ_m$ in an increasing order, i.e. $\mcJ_1\subsetneq \mcJ_2\subsetneq \cdots \subsetneq \mcJ_M$.
	
	(b). Let $1\leq m_1<m_2\leq M$, then we have
	\[\delta(E^{(p)})\geq C(\delta)\cdot \delta_{\mcJ_{m_1},\mcJ_{m_2}}(E^{(p)})\cdot \prod_{m\notin [m_1,m_2]}\delta_{\mcJ_m}(\Ep),\]
	where $C(\delta)$ is a positive constant depending only on $\delta$, $\#A$ and $p$.
\end{proposition}
\begin{proof}
	(a). is a consequence of Lemma \ref{lemma72}.
	
	(b). Let $x_0\neq y_0\in A$ such that $R^{(p)}(x_0,y_0)=\min\{R^{(p)}(x,y):x\neq y\in A\}$; let $x_M\neq y_M\in A$ such that $R^{(p)}(x_M,y_M)=\max\{R^{(p)}(x,y):x\neq y\in A\}$. We see that $x_0\mcJ_1 y_0$, since otherwise by choosing $x\mcJ_1 y$, $x\neq y$, we have $R^{(p)}(x,y)<R^{(p)}(x_0,y_0)$; we also see that $x_M\mcJ\mkern-10.5mu\backslash_M\ y_M$, since otherwise by choosing $x\mcJ\mkern-10.5mu\backslash_M\ y$ we have $R^{(p)}(x,y)>R^{(p)}(x_M,y_M)$.
	
	Next, for $1\leq m<M$, we choose $x_m,y_m\in A$ such that
	\[x_m\mcJ\mkern-10.5mu\backslash_{m}\ y_m,\quad x_m\mcJ_{m+1}\ y_m.\]
	Then, by Corollary \ref{coro74} and the fact that $\{\mcJ_m\}_{m=1}^M$ is the collection of all non-trivial equivalence relation $\mcJ$ such that $\delta_{\mcJ}(\Ep)<\delta$, we see that for any $1\leq m\leq M$,
	\begin{equation*}
\frac{R^{(p)}(x_{m-1},y_{m-1})}{R^{(p)}(x_m,y_m)}\geq C_1(\delta)\cdot \delta_{\mcJ_m}(\Ep),\quad \frac{R^{(p)}(x_{m_1-1},y_{m_1-1})}{R^{(p)}(x_{m_2},y_{m_2})}\geq C_1(\delta)\cdot \delta_{\mcJ_{m_1},\mcJ_{m_2}}(\Ep),\end{equation*}
where $C_1(\delta)>0$ is some small number depending only on $\delta$ and $\#A$. Hence,
	\[
	\delta(\Ep)=\frac{R^{(p)}(x_0,y_0)}{R^{(p)}(x_M,y_M)}\geq \big(C_1(\delta)\big)^{M-m_2+m_1}\cdot \delta_{J_{m_1},J_{m_2}}(\Ep)\cdot \prod_{m\notin [m_1,m_2]}\delta_{\mcJ_m}(\Ep).
	\]
	The lemma follows, noticing that $M\leq \#A$.
\end{proof}

\subsection{Trace of $p$-energies}
Next, we consider the trace of $p$-energies (recall Proposition \ref{prop22} (c)). In this subsection, we will fix an equivalence relation $\mcJ$. The results are inspired by \cite[Chapter 4]{Sabot}. Our proof is simplified by introducing a modified trace mapping.

\begin{definition}\label{def76}
	For $\delta>0$, $E^{(p)}\in\mathcal  M_p(A)$, we say that $\Ep\in \mathcal{M}_{p}(A,\mcJ,\delta)$ if the following two conditions hold:
	
	\noindent (i). $\delta_\mcJ(\Ep)<\delta$.
	
	\noindent (ii). There exists $\Ep_I\in \mathcal{M}_p(I)$ for each $I\in A/\mcJ$, and a finite collection $\{\Ep_i\}_{i=1}^m$ in $\widetilde{\mathcal M}_p(A)$ such that
	\begin{equation}\label{eqn77}
		\Ep(f)=\sum_{I\in A/\mcJ}\Ep_I(f|_I)+\sum_{i=1}^m\Ep_i(f),\qquad\text{ for any } f\in l(A),
	\end{equation}
	and for $i=1,\cdots, m$, $R_i^{(p)}(x,y)=\infty$ whenever $x\mcJ y, x\neq y$, where $R_i^{(p)}$ is the $p$-resistance associated with $E_i^{(p)}$.
\end{definition}

Clearly, $\{\Ep\in \mathcal{S}_p(A):\delta_\mcJ(\Ep)<\delta\}\subset \mathcal{M}_{p}(A,\mcJ,\delta)$. Thus, as an immediate consequence of Corollary \ref{coro27}, we have the following lemma.

\begin{lemma}\label{lemma77}
	There exists $C>1$ depending on $\#A$ and $p$, such that for any $\Ep\in \mathcal{M}_p(A)$ and any non-trivial equivalence relation $\mcJ$, we can find $\bar{E}^{(p)}\in \mathcal{M}_{p}\big(A,\mcJ,C^2\delta_\mcJ(\Ep)\big)$ such that $C^{-1}\Ep\leq \bar{E}^{(p)}\leq C\Ep$.
\end{lemma}

We consider two kinds of traces of the $p$-energies in this paper.\vspace{0.2cm}

\begin{definition}\label{def78}
	Let $\Ep\in \widetilde{\mathcal{M}}_p(A)$. Let $B\subset A$, and let $\mcJ|_B$ be the restriction of $\mcJ$ on $B$. We write $I_{I'}$ for the unique equivalence class of $\mcJ$ containing $I'\in B/{\mcJ|_B}$.
	
	For each $I\in A/\mcJ$ we choose $x_I\in I$.
	In addition, for each $I'\in B/{\mcJ|_B}$, we require $x_{I_{I'}}\in I'$.
	
	(a). \textbf{Restriction on $A/\mcJ$}: We define $[[\Ep]]_{A/\mcJ}\in \widetilde{\mathcal{M}}_p(A/\mcJ)$ by
	\[[[\Ep]]_{A/\mcJ}(u)=\Ep\big(\sum_{I\in A/\mcJ}u(I)1_I\big),\quad \forall u\in l(A/\mcJ). \]
	
	(b). \textbf{Modified trace:} Assume $\Ep\in \mathcal{M}_p(A,\mcJ,\delta)$ with the decomposition (\ref{eqn77}). We can define
	\[<\Ep>_B(f)=\sum_{I'\in B/{\mcJ|_B}}[\Ep_{I_{I'}}]_{I'}(f|_{I'})+[[[\Ep]]_{A/\mcJ}]_{B
		/{\mcJ|_B}}(u_f),\quad \text{ for any } f\in l(B),\]
	where $u_f$ is defined as $u_f(I')=f(x_{I_{I'}})$ for any $I'\in B/{\mcJ|_B}$.
\end{definition}

\noindent\textbf{Remark 1.} The modified trace $<\Ep>_B$ actually depends on the decomposition (\ref{eqn77}) and the choice of $x_{I_{I'}}$, which are not reflected in the notation.

\noindent\textbf{Remark 2.} We will see in Proposition \ref{prop79} that for $\delta$ small enough,  $<\Ep>_B\in\mathcal{M}_p(B,\mathcal{J}|_{B},\delta)$ whenever $\Ep\in\mathcal{M}_p(A,\mathcal{J},\delta/2)$.\vspace{0.2cm}

\begin{proposition}\label{prop79}
With the same conditions as in Definition \ref{def78} (b), assume $\delta<\min\{1,C^{-p}\}$ for some constant $C>0$ depending only on $\#A$. Then it holds that
	\[(1+\varepsilon)^{-1}\cdot <\Ep>_B\leq [\Ep]_B\leq (1+\varepsilon)\cdot <\Ep>_B,\]
	where $\varepsilon=\varepsilon(\delta)=\frac{1}{1-C\delta^{1/p}}(\frac{\delta^{-1/p}}{\delta^{-1/p}-1})^{p-1}-1>0$. Clearly, $\varepsilon(\delta)\to 0$ as $\delta\to 0$.
\end{proposition}
\begin{proof}
	Let $f\in l(B)$. We consider the following two different extensions of $f$ to $A$:
	
	\noindent 1). $f_1\in l(A)$ is the minimal energy extension of $f$ to $A$ with respect to $\Ep$.
	
	\noindent 2). $f_2\in l(A)$ achieves the minimal value of
	\[\sum_{I\in A/\mcJ}\Ep_I(f'|_I)+\sum_{i=1}^{m}\Ep_i\left(\sum_{I\in A/\mcJ}f'(x_I)1_I\right), \quad\text{ for } f'\in l(A), f'|_B=f,\]
	which equals to $<\Ep>_B(f)$. This can be realized with two steps: first, extend $f$ to $f_2$ on each $I_{I'},I'\in B/{\mcJ|_B}$ with respect to $\Ep_{I_{I'}}$ separately; then, take $f_2$ to be a constant $c_I$ on each remaining equivalence classes $I$, i.e. $I\in A/\mcJ\setminus\{I_{I'}: I'\in B/\mcJ|_B\}$, such that $\Ep\left(\sum_{I\in A/\mcJ}f_2(x_I)1_I\right)$ is minimized. \vspace{0.2cm}
	
	\noindent \textit{Proof of ``$\ [\Ep]_B\leq (1+\varepsilon)<\Ep>_B$''}. \vspace{0.2cm}
	
	For short, we write $g=\sum_{I\in A/\mcJ}f_2(x_I)1_I$. Observe that
	\begin{equation}\label{equation7.11}
\left\|f_2-g\right\|_{l^\infty(A)}\leq\max_{I\in A/\mcJ}\Osc(f_2|_I)\leq \left(\max_{x\mcJ y}R^{(p)}(x,y)\cdot\Ep(f_2)\right)^{1/p}.
\end{equation}

For each $\Ep_i$, we can apply Corollary \ref{coro27} to get an equivalent energy $\bar{E}^{(p)}_i$ of the form $\sum_{x\mcJ\mkern-10.5mu\backslash y}{c^{(i)}_{x,y}}|u(x)-u(y)|^p$ which is comparable with $\Ep_i$. Since $\sum_{i=1}^m\bar{E}^{(p)}_i\asymp\sum_{i=1}^m\Ep_i\leq\Ep$,
\begin{equation*}
R^{(p)}(x,y)^{-1}\gtrsim \sum_{i=1}^m{c^{(i)}_{x,y}},\quad\forall x\mcJ\mkern-10.5mu\backslash y.
\end{equation*}
We then arrive at
\begin{equation}\label{equation7.12}
\sum_{i=1}^m \Ep_i\left(f_2-g\right)\lesssim\sum_{i=1}^m \bar{E}^{(p)}_i\left(f_2-g\right)\lesssim \frac{\big\|f_2-g\big\|_{l^\infty(A)}^p}{\min_{x\mcJ\mkern-10.5mu\backslash y}R^{(p)}(x,y)}.
\end{equation}
	Hence, by \eqref{equation7.11} and \eqref{equation7.12}, we have for some constant $C>0$,
	\begin{equation}\label{eqn78}
		\sum_{i=1}^m \Ep_i\left(f_2-g\right)\leq C\frac{\max_{x\mcJ y}R^{(p)}(x,y)}{\min_{x\mcJ\mkern-10.5mu\backslash y}R^{(p)}(x,y)}\Ep(f_2)<C\delta\Ep(f_2),
	\end{equation}
	where we have used $\delta_{\mcJ}(\Ep)<\delta$ in the last inequality.
	
	Next, we let $\lambda=\frac{1}{\delta^{-1/p}-1}$. Using the $p$-homogeneity and convexity of each $E_i^{(p)}$, we have
	\[
	\begin{aligned}	E_i^{(p)}(f_2)=E_i^{(p)}(g+f_2-g)&=(1+\lambda)^pE_i^{(p)}\left(\frac{1}{1+\lambda}g+\frac{\lambda}{1+\lambda} \cdot\frac{f_2-g}{\lambda}\right)\\
		&\leq(1+\lambda)^{p-1}E_i^{(p)}(g)+\left(1+\frac1{\lambda}\right)^{p-1}E_i^{(p)}(f_2-g).
	\end{aligned}
	\]
	Summing up over $i=1,\cdots,m$ and using (\ref{eqn78}), we achieve
	\begin{equation}\label{equation7.10}
	\sum_{i=1}^m\Ep_i(f_2)\leq \left(\frac{\delta^{-1/p}}{\delta^{-1/p}-1}\right)^{p-1}\sum_{i=1}^m \Ep_i\left(\sum_{I\in A/\mcJ}f_2(x_I)1_I\right)+C\delta^{1/p}\Ep(f_2).
	\end{equation}
	Therefore
	\[\begin{aligned}
		\Ep(f_2)&\leq\sum_{I\in A/\mcJ}\Ep_I(f_2|_I)+\left(\frac{\delta^{-1/p}}{\delta^{-1/p}-1}\right)^{p-1}\sum_{i=1}^m \Ep_i\left(\sum_{I\in A/\mcJ}f_2(x_I)1_I\right)+C\delta^{1/p}\Ep(f_2)\\
		&\leq (\frac{\delta^{-1/p}}{\delta^{-1/p}-1})^{p-1}\cdot <\Ep>_B(f)+C\delta^{1/p}\Ep(f_2).
	\end{aligned}\]
	Hence, if $\delta<\min\{1,C^{-p}\}$,
	\[[\Ep]_B(f)\leq \Ep(f_2)\leq \big(\frac{1}{1-C\delta^{1/p}}\big)\cdot (\frac{\delta^{-1/p}}{\delta^{-1/p}-1})^{p-1}\cdot <\Ep>_B(f).\]
	\vspace{0.2cm}
	
	\noindent \textit{Proof of ``$\ (1+\varepsilon)^{-1}<\Ep>_B\leq [\Ep]_B$''}. \vspace{0.2cm}

	Similar to \eqref{equation7.10}, we have that
	\[\sum_{i=1}^m \Ep_i\left(\sum_{I\in A/\mcJ}f_1(x_I)1_I\right)\leq \left(\frac{\delta^{-1/p}}{\delta^{-1/p}-1}\right)^{p-1}\sum_{i=1}^m \Ep_i(f_1)+C\delta^{1/p}\Ep(f_1).\]
	As a consequence, we have
	\[
	\begin{aligned}
		<\Ep>_B(f)\leq &\sum_{I\in A/\mcJ}\Ep_I(f_1|_I)+\sum_{i=1}^{m}\Ep_i\left(\sum_{I\in A/\mcJ}f_1(x_I)1_I\right)\\
		\leq & \sum_{I\in A/\mcJ}\Ep_I(f_1|_I)+\left(\frac{\delta^{-1/p}}{\delta^{-1/p}-1}\right)^{p-1}\sum_{i=1}^m \Ep_i(f_1)+C\delta^{1/p}\Ep(f_1)\\
		\leq &\left((\frac{\delta^{-1/p}}{\delta^{-1/p}-1})^{p-1}+C\delta^{1/p}\right)\cdot[\Ep]_B(f).
	\end{aligned}
	\]
	\end{proof}

\section{A strengthened criteria of Sabot}\label{sec8}
In this section, we will prove a  criteria of Sabot (in $p$-energy version) concerning the behavior of $\mcT$. The main theorem, Theorem \ref{thm83}, consists of two parts: under certain conditions, one has (\textbf{A}) hold; under some other conditions, (\textbf{A}) does not hold. Unfortunately, like the $p=2$ case, the above two can not cover all possible situations.

We will improve Sabot's result in that we drop the (\textbf{H}) condition (see \cite[Section 5.1]{Sabot}), which will be explained soon in Remark 2 after the statement of Theorem \ref{thm83}. Throughout this section, we will fix $p\in (1,\infty)$ and omit the subscript $p$ or superscript $(p)$ as we did before.

\begin{definition}\label{def81}
	Let $\mcJ$ be an equivalence relation on $V_0$, and $\mathscr{G}$ be a group action on $V_0$.

	(a). We say that $\mcJ$ is a $\mathscr{G}$-relation if
	\[x\mcJ y\Longrightarrow \sigma(x)\mcJ \sigma(y),\]
	for any $x,y\in V_0$ and $\sigma\in \mathscr{G}$.
	
	(b). We define $\mcJ^{(1)}$ on $V_1$ as the minimal equivalence relation such that
	\[x\mcJ y\Longrightarrow F_ix\mcJ^{(1)}F_iy,\]
	for any $x,y\in V_0$ and $1\leq i\leq N$. With some abuse of the notation ``$\mcT$", we define $\mcT\mcJ=\mcJ^{(1)}|_{V_0}$.
	
	(c). We say $\mcJ$ is preserved if $\mcT\mcJ=\mcJ$; otherwise, $\mcJ$ is non-preserved.
\end{definition}

\begin{definition}\label{def82}
	Let $\mcJ$ be a non-trivial preserved equivalence relation on $V_0$.
	
		(a). We identify each $f\in l(V_0/\mcJ)$ with a function in $l(V_0)$ so that $f(x)=f(I)$ for any $x\in I\in V_0/\mcJ$. Do similarly for $f\in l(V_1/\mcJ^{(1)})$.
	
	(b). The operator $\Lambda:\mathcal{M}(V_0)\to \mathcal{M}(V_1)$ is automatically extended to $\Lambda_{V_0/\mcJ}:\mathcal{M}(V_0/\mcJ)\to \mathcal{M}(V_1/\mcJ^{(1)})$: for each $E\in \mathcal{M}(V_0/\mcJ)$, $\Lambda_{V_0/\mcJ}E\in \mathcal{M}(V_1/\mcJ^{(1)})$ is defined as
	\[\Lambda_{V_0/\mcJ} E(f)=\sum_{i=1}^N r_i^{-1}E(f\circ F_i),\quad\text{ for any } f\in l(V_1/\mcJ^{(1)}),\]
	and we define $\mcT_{V_0/\mcJ}:\mathcal{M}(V_0/\mcJ)\to \mathcal{M}(V_0/\mcJ)$ as $\mcT_{V_0/\mcJ}E=[\Lambda_{V_0/\mcJ} E]_{V_0/\mcJ}$.
	
	We define
	\[\underline{\rho}_{V_0/\mcJ}=\sup_{E\in {\mathcal M}(V_0/\mcJ)}\inf(\mcT_{V_0/\mcJ} E|E).\]
	
	(c). We denote  $\widetilde{\mathcal M}(V_0,\mcJ,0)$ the collection of $E\in \widetilde{\mathcal M}$ that takes the form $E=\sum_{I\in A/\mcJ}E_I$ where $E_I\in \mathcal{M}(I)$. We define
	\[\overline{\rho}_{\mcJ}=\inf_{E\in \widetilde{\mathcal M}(V_0,\mcJ,0)}\sup(\mcT E|E),\quad \underline{\rho}_{\mcJ}=\sup_{E\in \widetilde{\mathcal M}(V_0,\mcJ,0)}\inf(\mcT E|E).\]	
\end{definition}

Inspired by Sabot \cite{Sabot}, our criteria for the existence and non-existence results states as follows.
\begin{theorem}\label{thm83}
	(a). \textbf{Non-existence:} If $\underline{\rho}_\mcJ>\underline{\rho}_{V_0/\mcJ'}$ for some non-trivial preserved $\mathscr{G}$-relations $\mcJ$ and $\mcJ'$ on $V_0$, then $\mcT$ does not have a $\mathscr{G}$-symmetric eigenform in $\mathcal{M}$.
	
	(b). \textbf{Existence:} If $\overline{\rho}_\mcJ<\underline{\rho}_{V_0/\mcJ}$ for any non-trivial preserved $\mathscr{G}$-relation $\mcJ$, then condition (\textbf{A}) holds, hence $\mcT$ has a $\mathscr{G}$-symmetric eigenform in $\mathcal{M}$.
\end{theorem}

\noindent\textbf{Remark 1.} For the $p=2$ case handled in \cite[Theorem 5.1]{Sabot}, the uniqueness of the form is proved, while for general $p$, the uniqueness of the $p$-energy form, even on the Sierpinski gasket, is unsettled since the difference of two forms is not necessarily convex.

\noindent\textbf{Remark 2.} In \cite{Sabot}, the following condition (\textbf{H}) is assumed for the existence result.\vspace{0.15cm}

\noindent (\textbf{H}). {There do not exist two non-trivial preserved $\mathscr{G}$-relations $\mcJ,\mcJ'$ such that $\mcJ\subsetneq \mcJ'$.} \vspace{0.2cm}

It is direct to verify Theorem \ref{thm83} (a); while for (b), we need some more preparations.
\begin{proof}[Proof of Theorem \ref{thm83} (a)]
	We will prove by contradiction. Assume $E\in\mathcal M$ is a $\mathscr{G}$-symmetric eigenform of $\mcT$, i.e. $\mcT E=\lambda E$ for some $\lambda>0$.
	
	On one hand, we see that
	\[\lambda E(f)=\mcT E(f)\leq \mcT_{V_0/\mcJ'}[[E]]_{V_0/\mcJ'}(f),\quad\forall f\in l(V_0/\mcJ'),\]
	so that $\lambda\leq \inf\big(\mcT_{V_0/\mcJ'}[[E]]_{V_0/\mcJ'}\big|[[E]]_{V_0/\mcJ'}\big)\leq \underline{\rho}_{V_0/\mcJ'}$.
	
	On the other hand, one can find $E'=\sum_{I\in V_0/\mcJ}E'_I$ where $E'_I\in \mathcal{M}(I)$ for each $I\subset \mcJ$ such that $\mcT E'\geq \underline{\rho}_{\mcJ} E'$. By Proposition \ref{prop26}, there is $C>0$ such that $E'\leq CE$. Now, fix $f\in l(V_0)$ with $E'(f)>0$. For any $n\geq 1$, we have
	\[\underline{\rho}_{\mcJ}^nE'(f)\leq\mcT^n E'(f)\leq C\mcT^n E(f)=C\lambda^nE(f),\]
	which implies that $\lambda\geq \underline{\rho}_{\mcJ}$.
	
	Combining the above two observations, we have $\underline{\rho}_{\mcJ}\leq\lambda\leq \underline{\rho}_{V_0/\mcJ'}$, which contradicts the assumption $\underline{\rho}_{\mcJ}>\underline{\rho}_{V_0/\mcJ'}$. Hence $\mcT$ does not have a $\mathscr{G}$-symmetric eigenform in $\mathcal{M}$.
\end{proof}

\subsection{Non-symmetric Sierpinski gasket}\label{subsec81}
In this subsection, we consider an example, the Sierpinski gasket.

Let $q_1=(0,0)$, $q_2=(1,0)$ and $q_3=(\frac12,\frac{\sqrt{3}}{2})$ be the three vertices of a unit triangle in $\mathbb R^2$. For $i=1,2,3$, let $F_i$ be the contraction on $\mathbb R^2$ of the form $F_i(x)=\frac12 (x-q_i)+q_i$. The self-similar set  w.r.t. the i.f.s. $\{F_1,F_2,F_3\}$ is the Sierpinski gasket (SG) in $\mathbb R^2$.

In \cite{HPS}, Herman, Peirone and Strichartz constructed a fully symmetric $p$-energy on SG. Here, following \cite{Sabot}, we consider a non-symmetric case. In the following we fix a vector of renormalization factors $\bm{r}=(r_1,r_2,r_3)$. In other words, the associated renormalization map $\mcT$ is defined by
	\begin{equation*}
		\mcT E(f)=\min\{ \Lambda E(g):\ g\in l(V_1),g|_{V_0}=f\},\quad E\in\mathcal M(V_0), f\in l(V_0),
	\end{equation*}
	where
	\begin{equation*}
		\Lambda E(g)=r_1^{-1}E(g\circ F_1)+r_2^{-1}E(g\circ F_2)+r_3^{-1}E(g\circ F_3), \quad g\in l(V_1).
\end{equation*}

Clearly, there are three non-trivial preserved relations $\mcJ_1,\mcJ_2$ and $\mcJ_3$, which can be described as
\[\begin{cases}
	q_1\mcJ\mkern-10.5mu\backslash_1q_2,\  q_2\mcJ_1 q_3;\\
	q_2\mcJ\mkern-10.5mu\backslash_2q_3,\  q_1\mcJ_2 q_3;\\
	q_3\mcJ\mkern-10.5mu\backslash_3q_1,\  q_1\mcJ_3 q_2.
\end{cases}
\]
Noticing that for $i=1,2,3$, there is only one $p$-energy form (up to a constant multiplier) in $\mathcal{M}(V_0/\mcJ_i)$ and in $\widetilde{\mathcal{M}}(V_0,\mcJ,0)$. For convenience, we use $\{i,j,k\}$ to denote a permutation of $\{1,2,3\}$. By direct computation, we can see
\[\overline{\rho}_{\mcJ_i}=\underline{\rho}_{\mcJ_i}=\left(r_j^{\frac{1}{p-1}}+r_k^{\frac{1}{p-1}}\right)^{1-p},\quad \underline{\rho}_{V_0/\mcJ_i}=\left(\left(\frac{1}{r_j}+\frac{1}{r_k}\right)^{\frac{1}{1-p}}+r_i^{\frac{1}{p-1}}\right)^{1-p}.\]
For convenience of readers, in the following, we briefly explain the computation.

We first consider the unique $p$-energy (up to a constant multiplier) $E_i\in \widetilde{\mathcal{M}}(V_0,\mcJ_i,0)$ defined as $E_i(f)=|f(q_j)-f(q_k)|^p$ for $f\in l(V_0)$. Then
\begin{align*}
	\mathcal TE_i(f)&=\min_{x,y,z\in\mathbb R} \left\{r_j^{-1}|f(q_j)-x|^p+r_k^{-1}|x-f(q_k)|^p+r_i^{-1}|y-z|^p\right\}\notag\\
	&=\left(r_j^{\frac{1}{p-1}}+r_k^{\frac{1}{p-1}}\right)^{1-p}|f(q_j)-f(q_k)|^p.
\end{align*}
From this, we find that
\begin{equation*}
	\overline{\rho}_{\mcJ_i}=\underline{\rho}_{\mcJ_i}=\mathcal TE_i(f)/E_i(f)=\left(r_j^{\frac{1}{p-1}}+r_k^{\frac{1}{p-1}}\right)^{1-p}.
\end{equation*}
Next, we consider the unique $p$-energy (up to a constant multiplier) $E_i'\in \mathcal{M}(V_0/\mcJ_i)$ defined as $E'_{i}(f)=\big|f(\{q_i\})-f(\{q_j,q_k\})\big|^p$ for $f\in l(V_0/\mcJ_i)$, and hence
\begin{align*}
	\mathcal TE_i'(f)&=\min_{x\in\mathbb R} \left\{r_i^{-1}\big|f(\{q_i\})-x\big|^p+\left(r_j^{-1}+r_k^{-1}\right)\big|x-f(\{q_j,q_k\})\big|^p\right\}\notag\\
	&=\left(\left(\frac{1}{r_j}+\frac{1}{r_k}\right)^{\frac{1}{1-p}}+r_i^{\frac{1}{p-1}}\right)^{1-p}\big|f(\{q_i\})-f(\{q_j,q_k\})\big|^p.
\end{align*}
Thus we have
\begin{equation*}
	\underline{\rho}_{V_0/\mcJ}=\mathcal TE_i'(f)/E_i'(f)=\left(\left(\frac{1}{r_j}+\frac{1}{r_k}\right)^{\frac{1}{1-p}}+r_i^{\frac{1}{p-1}}\right)^{1-p}.
\end{equation*}

\begin{corollary}\label{thmSG}
Assume $r_i\geq r_j\geq r_k$ for $\{i,j,k\}=\{1,2,3\}$.
	
(a). If $r_j^{\frac{1}{p-1}}+r_k^{\frac{1}{p-1}}<\left(\frac{1}{r_j}+\frac{1}{r_k}\right)^{\frac{1}{1-p}}+r_i^{\frac{1}{p-1}}$, then there does not exist a non-degenerate $p$-energy eigenform of $\mathcal{T}$.

(b). If $r_j^{\frac{1}{p-1}}+r_k^{\frac{1}{p-1}}>\left(\frac{1}{r_j}+\frac{1}{r_k}\right)^{\frac{1}{1-p}}+r_i^{\frac{1}{p-1}}$, then there does exist a non-degenerate $p$-energy eigenform of $\mathcal{T}$.
\end{corollary}
\begin{proof}
One can easily check that $\overline{\rho}_{\mcJ_i}=\underline{\rho}_{\mcJ_i}\geq\overline{\rho}_{\mcJ_j}=\underline{\rho}_{\mcJ_j}\geq \overline{\rho}_{\mcJ_k}=\underline{\rho}_{\mcJ_k}$ and $\underline{\rho}_{V_0/\mcJ_i}\leq \underline{\rho}_{V_0/\mcJ_j}\leq \underline{\rho}_{V_0/\mcJ_k}$. So to apply Theorem \ref{thm83}, it suffices to compare $\overline{\rho}_{\mcJ_i}$ and $\underline{\rho}_{V_0/\mcJ_i}$.

(a). If $\underline{\rho}_{\mcJ_i}=(r_j^{\frac{1}{p-1}}+r_k^{\frac{1}{p-1}})^{1-p}>\left(\left(\frac{1}{r_j}+\frac{1}{r_k}\right)^{\frac{1}{1-p}}+r_i^{\frac{1}{p-1}}\right)^{1-p}=\underline{\rho}_{V_0/\mcJ_i}$.  Then by Theorem \ref{thm83} (a), there does not exist a non-degenerate $p$-energy eigenform of $\mathcal{T}$.

(b). If $\overline{\rho}_{\mcJ_i}=(r_j^{\frac{1}{p-1}}+r_k^{\frac{1}{p-1}})^{1-p}<\left(\left(\frac{1}{r_j}+\frac{1}{r_k}\right)^{\frac{1}{1-p}}+r_i^{\frac{1}{p-1}}\right)^{1-p}=\underline{\rho}_{V_0/\mcJ_i}$. Then it also holds that $\overline{\rho}_{\mcJ_j}<\underline{\rho}_{V_0/\mcJ_j}$ and $\overline{\rho}_{\mcJ_k}<\underline{\rho}_{V_0/\mcJ_k}$. Hence by Theorem \ref{thm83} (b), there exists a non-degenerate $p$-energy eigenform of $\mathcal{T}$.
\end{proof}

\subsection{Non-preserved $\mathscr{G}$-relations}\label{subsec82}
Let $\bm r$ be $\mathscr{G}$-symmetric, so that for any $\mathscr{G}$-symmetric $p$-energy $E\in \mathcal{M}$,  $\mcT^nE$ is also $\mathscr{G}$-symmetric, $n\geq 1$. Then if $\mcJ$ is not a $\mathscr{G}$-relation, it is straightforward to see that $\delta_\mcJ(\mcT^nE)\geq 1$. Hence, to verify (\textbf{A}), we only need to care about $\mathscr{G}$-relations. In this subsection, we further show that we do not need to worry about non-preserved $\mathscr{G}$-relations (Proposition \ref{prop85} (b)).

\begin{lemma}\label{lemma84}
	Let $E,E'\in \mathcal{M}$ and $R,R'$ be the $p$-resistances associated with $E,E'$ respectively. Then
	\[
	\inf(E|E') \leq \min_{x\neq y}\frac{R'(x,y)}{R(x,y)}\leq \max_{x\neq y}\frac{R'(x,y)}{R(x,y)}\leq \sup(E|E').
	\]
\end{lemma}
\begin{proof}
	We pick $x_0\neq y_0$ so that $\frac{R'(x_0,y_0)}{R(x_0,y_0)}=\max_{x\neq y}\frac{R'(x,y)}{R(x,y)}$, and let $f$ be a function such that $f(x_0)=1$, $f(y_0)=0$ and $E'(f)=R'(x_0,y_0)^{-1}$. Noticing that $E(f)\geq R(x_0,y_0)^{-1}$, hence
	\[\sup(E|E')\geq \frac{E(f)}{E'(f)}\geq \frac{R'(x_0,y_0)}{R(x_0,y_0)}=\max_{x\neq y}\frac{R'(x,y)}{R(x,y)}.\]
	Similarly, we have $\inf(E|E') \leq \min_{x\neq y}\frac{R'(x,y)}{R(x,y)}$.
\end{proof}

Recall that for each $E\in \mathcal{M}$, $\theta(E)=\frac{\sup(\mcT E|E)}{\inf (\mcT E|E)}$, see Definition \ref{def33} (b).

\begin{proposition}\label{prop85}
	Let $\mcJ$ be a non-trivial equivalence relation on $V_0$.
	
	(a). There is a constant $C>0$ such that $\delta_{\mcT\mcJ}(\mcT E)\leq C\delta_{\mcJ}(E)$ for any $E\in \mathcal{M}$.
	
	(b). Assume that $\mcT\mcJ\neq\mcJ$. Then there is a constant $C'>0$ such that $\theta(E)\geq C'\cdot\delta_\mcJ(E)^{-1}$ for any $E\in \mathcal{M}$.
\end{proposition}
\begin{proof}
	Let $E\in \mathcal{M}$ and $R$ be the $p$-resistance associated with $E$. For convenience, we denote $R_1$ the $p$-resistance associated with $\Lambda E$. Also, let $x\neq y\in V_0$.
	
	First, assume $x\mcT\mcJ y$, which is equivalent to say $x\mcJ^{(1)}y$ since $x,y\in V_0$. We can find a path $x=z_0,z_1,z_2,\cdots,z_n=y$ such that $z_i\mcJ^{(1)}z_{i+1}$ for each $0\leq i<n$. Then $n\leq \#V_1$ and $R_1(z_i,z_{i+1})\leq (\max_{1\leq i\leq N} r_i)\cdot (\max_{x'\mcJ y'}R(x',y'))$. Hence, there exists $C_1>0$ independent of $E,x,y$ such that
	\begin{equation}\label{eqn81}
		R_1(x,y)\leq C_1\cdot\max_{x'\mcJ y'}R(x',y').
	\end{equation}
	
	Next, assume $x\mcJ\mkern-10.5mu\backslash^{(1)}\ y$. Then, as in Corollary \ref{coro27}, we define $\bar E$ as
	\[\bar E(f)=\sum_{x,y\in V_0}R(x,y)^{-1}|f(x)-f(y)|^p,\quad \forall f\in l(V_0).\]
	By Corollary \ref{coro27}, we know $E\leq C_2 \bar E$ for some $C_2>0$ independent of $E$. Let $I^{(1)}_x\in V_1/\mcJ^{(1)}$ such that $x\in I^{(1)}_x$, and let $1_{I^{(1)}_x}\in l(V_1)$ be the indicator function of $I^{(1)}_x$. Then,
	\begin{equation}\label{eqn82}
		R_1(x,y)^{-1}\leq \Lambda E(1_{I^{(1)}_x})\leq C_2\Lambda\bar{E}(1_{I^{(1)}_x})\leq C_3\big(\min_{x'\mcJ\mkern-10.5mu\backslash y'}R(x',y')\big)^{-1},
	\end{equation}
	for some $C_3>0$ independent of $E,x,y$.\vspace{0.2cm}
	
Now (a) follows immediately by combining (\ref{eqn82}) with (\ref{eqn81}).
	
\medskip

Then we prove (b). We will consider two different cases. \vspace{0.2cm}
	
	\noindent\textit{Case 1: $\mathcal{T}\mcJ\not\subset\mcJ$}. In this case, we can find $x,y$ such that $x\mathcal{T}\mcJ y$ and $x\mcJ\mkern-10.5mu\backslash y$.
	
	On one hand, by (\ref{eqn81}), $R_1(x,y)\leq C_1\cdot\max_{x'\mcJ y'}R(x',y')\leq C_1\delta_\mcJ(E) R(x,y)$. Hence, by Lemma \ref{lemma84}, we have $\sup(\mcT E|E)\geq C_1^{-1}\delta_\mcJ(E)^{-1}$.
	
	On the other hand, by letting $E'\in\mathcal M$ defined by $E'(f)=\sum_{x\neq y}\big|f(x)-f(y)\big|^p$ for $f\in l(V_0)$ and $C_4=\sup(\mathcal{T}E'|E')>0$, using Lemma \ref{lemma34} (c), we also have $\inf(\mathcal{T}E|E)\leq C_4$.
	
	Thus, $\theta(E)\geq C_1^{-1}C_4^{-1}\cdot\delta_\mcJ(E)^{-1}$, where $C_1,C_4$ are independent of $E$. \vspace{0.2cm}

	\noindent\textit{Case 2: $\mcJ\not\subset\mathcal{T}\mcJ$.} In this case, there exists a pair $x\neq y\in V_0$ such that $x\mcJ y$ and $x\mcJ\mkern-10.5mu\backslash^{(1)}\ y$.
	
	Then, by (\ref{eqn82}), $R_1(x,y)\geq C_3^{-1}\min_{x'\mcJ\mkern-10.5mu\backslash y'}R(x',y')\geq C_3^{-1}\delta_\mcJ(E)^{-1}R(x,y)$. Hence, by Lemma \ref{lemma84}, we know $\inf(\mcT E|E)\leq C_3\delta_\mcJ(E)$.
	
	On the other hand, by letting $E'$ be the same as in Case 1, and $C_5=\inf(\mathcal{T}E'|E')>0$, using Lemma \ref{lemma34} (c), we also have $\sup(\mathcal{T}E|E)\geq C_5$.
	
	Thus, $\theta(E)\geq C_5C_3^{-1}\cdot\delta_\mcJ(E)^{-1}$, where $C_3,C_5$ are independent of $E$.\vspace{0.2cm}
	
	Combining above two cases, we finish the proof of (b) by choosing $C'=\min\{C_1^{-1}C_4^{-1},C_5C_3^{-1}\}$.
\end{proof}

\noindent\textbf{Remark}. Noting that  by Lemma \ref{lemma34} (b), we have $\theta(\mcT^nE)\leq \theta(E),\forall\ n\geq 0$. Together with Proposition \ref{prop85} (b), we see that if $\mcJ$ is non-preserved, then  $\inf_{n\geq 0}\delta_\mcJ(\mcT^nE)\geq C\theta(E)^{-1}$ for some $C>0$.
In particular, by applying Corollary \ref{coro74} for $\mcJ=0,\mcJ'=1$, we can see that Theorem \ref{thm83} (b) holds if there does not exist any non-trivial preserved $\mathscr{G}$-relations.

\subsection{Preserved $\mathscr{G}$-relations}\label{subsec83}
In this subsection, we consider preserved non-trivial $\mathscr{G}$-relations. We will use Proposition \ref{prop79} in this part.

\begin{definition}\label{def87}
	Let $\mcJ$ be a non-trivial $\mathscr{G}$-relation and $E\in \mathcal{M}(V_0,\mcJ,\delta)$ for some $\delta>0$ which admits the decomposition (\ref{eqn77}) that
	\[E(f)=\sum_{I\in V_0/\mcJ}E_I(f|_I)+\sum_{i=1}^mE_i(f), \quad\text{ for any } f\in l(V_0),\]
	where $E_I\in\mathcal M(I)$ and $E_i\in \mathcal{M}$ such that $R_i(x,y)=\infty$ for any $x\mcJ y$, $x\neq y$. We  define
	\[\tilde{\mcT}_\mcJ  E(f)=\mcT\big(\sum_{I\in V_0/\mcJ} E_I\big)(f)+\mcT_{V_0/\mcJ}[[E]]_{V_0/\mcJ}(\sum_{I\in V_0/\mcJ}f(x_I)1_I),\quad\text{ for any } f\in l(V_0).\]
\end{definition}

The following Lemma is an immediate consequence of Proposition \ref{prop79}.
\begin{lemma}\label{lemma88}
	For any $\varepsilon>0$, there is $\delta>0$ so that
	\[(1+\varepsilon)^{-1}\cdot \tilde{\mcT}_{\mcJ}E\leq \mcT E\leq (1+\varepsilon)\cdot\tilde{\mcT}_{\mcJ} E\]
	for any preserved $\mathscr{G}$-relation $\mcJ$ and $E\in \mathcal{M}(V_0,\mcJ,\delta)$.
\end{lemma}

\noindent\textbf{Remark}. Let $E\in \mathcal{M}(V_0,\mcJ,\delta)$ for some $\delta>0$. Then $E$ admits the standard decomposition $E(f)=\sum_{I\in V_0/\mcJ}E_I(f|_I)+\sum_{i=1}^mE_i(f),\forall f\in l(V_0)$. For convenience, we denote $E_\mcJ(f)=\sum_{I\in V_0/\mcJ} E_I(f|_I),\forall f\in l(V_0)$. Then by Definition \ref{def87}, we immediately have the following two claims.

1) If $f\in l(V_0)$ satisfies $f(x_I)=0,\forall I\in V_0/\mcJ$, then $\tilde{\mcT}^m_\mcJ E(f)=\mcT^mE_\mcJ(f),\forall m\geq 1$.

2) If  $f\in l(V_0)$ satisfies $f(x)=\sum_{I\in V_0/\mcJ}f(x_I)\cdot 1_I$, then $\tilde{\mcT}^m_\mcJ E(f)=\mcT^m_{V_0/\mcJ}E(f), \forall m\geq 1$.\vspace{0.2cm}

The main result of this subsection is the following Proposition \ref{prop89}. For the special $p=2$ case, the proposition can be easily proved by using Lemma \ref{lemma88}. For general $p$, additional difficulty comes from the fact that the trace of $\mcT E$ may not be a good form even if $E\in \mathcal{Q}$.

\begin{proposition}\label{prop89}
	For any $\varepsilon>0$, there is $\delta>0$ and a constant $C\geq 1$ so that for any preserved non-trivial $\mathscr{G}$-relation $\mcJ$ and $E\in \mathcal{M}$, if $n\geq 1$ and $\delta_\mcJ(\mcT^mE)<\delta,\ \forall\ 0\leq m\leq n$, then the following (a) and (b) hold.

	(a). There exists $f\in l(V_0)$ such that $\Osc(f)=1$, $f(x_I)=0,\forall\ I\in V_0/\mcJ$ and \[\mcT^nE(f)\leq C^2(1+\varepsilon)^n\overline{\rho}_\mcJ^n E(f).\]
	
	(b). There exists $f\in l(V_0)$ such that $\Osc(f)=1$, $f(x)=\sum_{I\in V_0/\mcJ}f(x_I)\cdot 1_I$ and
	\[\mcT^nE(f)\geq C^{-2}(1+\varepsilon)^{-n}\underline{\rho}_{V_0/\mcJ}^n E(f).\]
	
\end{proposition}
\begin{proof}
	Let $n\in \mathbb{N}$ and $\varepsilon>0$ be the same as in the statement of the proposition. $E$ satisfies the condition of Proposition \ref{prop89}, with the constant $\delta$ to be explained soon.
	
	We apply Lemma \ref{lemma77} to introduce $\bar{E}\in \mathcal{M}\big(V_0,\mcJ,C^2\delta_\mcJ(E)\big)$ satisfying $C^{-1}E\leq \bar{E}\leq C E$  for some constant $C\geq 1$ independent of $\mcJ,E$. 
For convenience, we will first focus on $\bar{E}$ in the proof, noticing that by $C^{-1}E\leq \bar{E}\leq C E$, we have $C^{-1}\mcT^mE\leq \mcT^m\bar{E}\leq C \mcT^mE$, and consequently,
\begin{equation*}\delta_\mcJ(\mcT^m\bar{E})\leq C^2\delta_\mcJ(\mcT^mE),\qquad\forall m\geq 1.
\end{equation*}

	\medskip

	Next, let us introduce the choice of $\delta$.\vspace{0.2cm}
	
	\noindent\textbf{Choice of $\delta$}: we choose a large enough $n_0\in \mathbb{N}$ so that
	\[C^2\cdot (1+\varepsilon/2)^{n_0}\leq (1+\varepsilon)^{n_0}.\]
	By Lemma \ref{lemma88}, we can choose $\delta_0>0$ such that whenever $E\in \mathcal{M}(V_0,\mcJ,\delta_0)$, it holds
\begin{equation}\label{equation8.3}
(1+\varepsilon/2)^{-1}\tilde{\mcT}_\mcJ E\leq \mcT E\leq (1+\varepsilon/2)\tilde\mcT_{\mcJ}E.
 \end{equation}
 We choose $\delta=C^{-4}\cdot(1+\varepsilon/2)^{-2n_0}\cdot\delta_0$. \vspace{0.2cm}
	
	The key to prove (a), (b) is the following claims together with the remark below Lemma \ref{lemma88}.\vspace{0.2cm}
	
	\noindent\textit{Claim 1. If $\delta_{\mcJ}(\mcT^m\bar{E})<(1+\varepsilon/2)^{-2m}\cdot\delta_0$ for any $0\leq m\leq (n_0\wedge n)-1$, then for any $0\leq m\leq (n_0\wedge n)$, we have
	\begin{equation}\label{equation8.4}
(1+\varepsilon/2)^{-m}\tilde\mcT^m_{\mcJ}\bar{E}\leq \mcT^m\bar{E}\leq (1+\varepsilon/2)^m\tilde\mcT^m_{\mcJ}\bar{E}.
\end{equation}}
	
	\begin{proof}[Proof of Claim 1] We will use induction. The case $m=0$ is trivial. If \eqref{equation8.4} holds for some $0\leq m\leq (n_0\wedge n)-1$, i.e.
		\begin{equation}\label{equation8.5}
(1+\varepsilon/2)^{-m}\tilde\mcT^{m}_{\mcJ}\bar{E}\leq \mcT^{m}\bar{E}\leq (1+\varepsilon/2)^{m}\tilde\mcT^{m}_{\mcJ}\bar{E},
\end{equation}
		then consequently,
\begin{equation}\label{equation8.6}
\delta_\mcJ(\tilde\mcT^{m}_{\mcJ}\bar{E})\leq (1+\varepsilon/2)^{2m}\delta_\mcJ(\mcT^{m}\bar{E})<\delta_0.
\end{equation}

Hence by \eqref{equation8.6} and applying \eqref{equation8.3} to $\tilde\mcT^{m}_{\mcJ}\bar{E}$, we have
\begin{equation}\label{equation8.7}
(1+\varepsilon/2)^{-1}\tilde\mcT^{m+1}_{\mcJ}\bar{E}\leq \mcT\tilde\mcT^{m}_{\mcJ}\bar{E}\leq (1+\varepsilon/2)\tilde\mcT^{m+1}_{\mcJ}\bar{E}.
\end{equation}
Now by acting $\mcT$ to \eqref{equation8.5} and comparing with \eqref{equation8.7}, we obtain
\begin{equation*}
(1+\varepsilon/2)^{-m-1}\tilde\mcT^{m+1}_{\mcJ}\bar{E}\leq \mcT^{m+1}\bar{E}\leq (1+\varepsilon/2)^{m+1}\tilde\mcT^{m+1}_{\mcJ}\bar{E}.
\end{equation*}
The induction is completed and hence Claim 1 holds.
	\end{proof}
	
	\noindent\textit{Claim 2. $(1+\varepsilon)^{-n}\tilde\mcT^n_{\mcJ}\bar{E}\leq \mcT^n\bar{E}\leq (1+\varepsilon)^n\tilde\mcT^n_{\mcJ}\bar{E}$.}
	
	\begin{proof}[Proof of Claim 2]
		If $n\leq n_0$, then Claim 2 follows by Claim 1. Now assume $n>n_0$, write $n=kn_0+s$ for some $k\geq 1$ and $0\leq s<n_0$.
		
		For the ``$\leq$'' side, we still use induction. Assume for $2\leq l\leq k$, we have (noticing that this is  true by Claim 1 when $l=2$)
		\[
		\mcT^{n_0(l-1)}\bar{E}\leq C^{-2}(1+\varepsilon)^{n_0(l-1)}\tilde\mcT^{n_0(l-1)}_{\mcJ}\bar{E}.
		\]
		Applying Lemma \ref{lemma77}, we can find $\bar E_{l-1}\in \mathcal M_p\big(V_0,\mcJ, C^2\delta_\mcJ(\mcT^{n_0(l-1)}\bar E)\big)$ independent of $\mcJ$, so that
		\begin{equation*}
			C^{-1}\mcT^{n_0(l-1)}\bar{E}\leq \bar{E}_{l-1}\leq C\mcT^{n_0(l-1)}\bar{E}.
		\end{equation*}
		Hence
		\begin{equation}\label{eqn84}
			\bar{E}_{l-1}\leq C\mcT^{n_0(l-1)}\bar{E}\leq C^{-1}(1+\varepsilon)^{n_0(l-1)}\tilde\mcT^{n_0(l-1)}_{\mcJ}\bar{E}.
		\end{equation}
		Just as in the remark below Lemma \ref{lemma88}, one can decompose $\bar{E}_{l-1}$ into two parts $\bar{E}_{l-1,\mcJ}$ and $[[\bar{E}_{l-1}]]_{V_0/\mcJ}$. One will have the following two estimates
		\begin{eqnarray}
			\label{eqn85}&\bar{E}_{l-1,\mcJ}\leq C^{-1}(1+\varepsilon)^{n_0(l-1)}\mcT^{n_0(l-1)}\bar{E}_{\mcJ},\\
			\label{eqn86}&{[}{[}\bar{E}_{l-1}{]}{]}_{V_0/\mcJ}\leq C^{-1}(1+\varepsilon)^{n_0(l-1)}\mcT^{n_0(l-1)}_{V_0/\mcJ}[[\bar{E}]]_{V_0/\mcJ}.
		\end{eqnarray}
		In fact, (\ref{eqn86}) is immediate from the remark and (\ref{eqn84}); to check (\ref{eqn85}), one only need to test $f\in l(V_0)$ such that $f(x_I)=0,\forall I\in V_0/\mcJ$, then $\bar{E}_{l-1}(f)\leq C^{-1}(1+\varepsilon)^{n_0(l-1)}\tilde\mcT^{n_0(l-1)}_{\mcJ}\bar{E}(f)=C^{-1}(1+\varepsilon)^{n_0(l-1)}\mcT^{n_0(l-1)}\bar{E}_{\mcJ}(f)$.
		
		As a consequence of (\ref{eqn85}) and (\ref{eqn86}), one can see
		\[
		\tilde{\mcT}^{n_0}_{\mcJ}\bar{E}_{l-1}\leq C^{-1}(1+\varepsilon)^{n_0(l-1)}\tilde{\mcT}^{n_0l}_{\mcJ}\bar{E}.
		\]
		Noticing that one can apply Claim 1 to $\bar{E}_{l-1}$ as $\delta_{\mcJ}(\mcT^m\bar{E}_{l-1})\leq C^2\delta_{\mcJ}(\mcT^{m+n_0(l-1)}\bar E)\leq C^4\delta_{\mcJ}(\mcT^{m+n_0(l-1)} E)<(1+\varepsilon/2)^{-2n_0}\cdot \delta_0$, $\forall 0\leq m<n_0$, we have
		\[\begin{aligned}
			\mcT^{n_0l}\bar{E}\leq C\mcT^{n_0}\bar{E}_{l-1}&\leq C(1+\varepsilon/2)^{n_0}\tilde{\mcT}^{n_0}_{\mcJ}\bar{E}_{l-1}\\
			&\leq (1+\varepsilon/2)^{n_0}(1+\varepsilon)^{n_0(l-1)}\tilde\mcT^{n_0l}_{\mcJ}\bar{E}\leq C^{-2}(1+\varepsilon)^{n_0l}\tilde\mcT^{n_0l}_{\mcJ}\bar{E}.
		\end{aligned}\]
		Thus, by the above induction, we get $\mcT^{n_0k}\bar{E}\leq C^{-2}(1+\varepsilon)^{n_0k}\tilde\mcT^{n_0k}_{\mcJ}\bar{E}$.
		
		Finally,  still by Lemma \ref{lemma77}, we choose $\bar E_k\in \mathcal M_p\big(V_0,\mcJ, C^2\delta_\mcJ(\mcT^{n_0k}\bar E)\big)$ independent of $\mcJ$, so that
		\begin{equation*}
			C^{-1}\mcT^{n_0k}\bar{E}\leq \bar{E}_{k}\leq C\mcT^{n_0k}\bar{E}.
		\end{equation*}
		 By the same argument as above, one can see
		\[\mcT^n\bar{E}\leq C\mcT^s\bar{E}_k\leq C(1+\varepsilon/2)^s\tilde{\mcT}_\mcJ^s\bar{E}_k\leq (1+\varepsilon/2)^s(1+\varepsilon)^{n_0k}\tilde\mcT^{n}_{\mcJ}\bar{E}\leq (1+\varepsilon)^{n}\tilde\mcT^{n}_{\mcJ}\bar{E},\]
		so ``$\leq$'' is proved.

The  ``$\geq$'' side can be obtained in a same way and hence Claim 2 holds.
	\end{proof}
	
	The rest of the proof is routine by using Claim 2 and the remark below Lemma \ref{lemma88}.
	
(a). We first show that $\inf(\mcT^n\bar{E}_\mcJ|\bar{E}_\mcJ)\leq \overline{\rho}^n_\mcJ$.

 Indeed, if it is not true, there is $\rho>\overline{\rho}_\mcJ$ such that
  		\begin{equation}\label{equation8.11}
		\inf(\mcT^n\bar{E}_\mcJ|\bar{E}_\mcJ)>\rho^n,
		\end{equation}
and by $\rho>\overline{\rho}_\mcJ$, there exists $E'\in\widetilde{\mathcal M}(V_0,\mcJ,0)$ such that
\begin{equation}\label{equation8.12}
	\sup(\mcT E'|E')<\rho.
		\end{equation}
Now by Lemma \ref{lemma34} (a) and using \eqref{equation8.12}, we see that
\begin{equation}\label{equation8.13}
	\sup(\mcT^n E'|E')=\sup_{f}\prod_{i=0}^{n-1}\frac{\mcT^{i+1} E'(f)}{\mcT^{i}E'(f)}\leq (\sup(\mcT E'| E'))^n<\rho^n.
		\end{equation}
Combining \eqref{equation8.11} and \eqref{equation8.13}, we obtain
$
	\inf(\mcT^n E'|\mcT^n \bar E_\mcJ)<\inf( E'| \bar E_\mcJ),
$
a contradiction with Lemma \ref{lemma34} (a). So we have $\inf(\mcT^n\bar{E}_\mcJ|\bar{E}_\mcJ)\leq \overline{\rho}^n_\mcJ$.

Hence there exists $f$ satisfying $\Osc(f)=1$, $f(x_I)=0,\forall I\in V_0/\mcJ$ such that $\mcT^n\bar{E}_\mcJ(f)\leq \overline{\rho}^n_\mcJ\bar{E}_\mcJ(f)$. Then by the second inequality in Claim 2 and the remark after Lemma \ref{lemma88}, we have
	\[\begin{aligned}
		\mcT^nE(f)\leq C\mcT^n\bar{E}(f)&\leq C(1+\varepsilon)^n\tilde\mcT_\mcJ^n\bar{E}(f)\\
		&=C(1+\varepsilon)^n\mcT^n\bar{E}_\mcJ(f)\leq C(1+\varepsilon)^n\overline{\rho}^n_\mcJ\bar{E}_\mcJ(f)\leq C^2(1+\varepsilon)^n\overline{\rho}^n_\mcJ E(f).
	\end{aligned}\]

 (b). Similar to (a), we can show that $\sup(\mcT_{V_0/\mcJ}^n[[\bar{E}]]_{V_0/\mcJ}|[[\bar{E}]]_{V_0/\mcJ})\geq \underline{\rho}^n_{V_0/\mcJ}$. Hence there exists $f\in l(V_0)$  of the form $f(x)=\sum_{I\in V_0/\mcJ}f(x_I)1_I$ with $\Osc(f)=1$, such that $\mcT^n\bar{E}(f)\geq \underline{\rho}^n_{V_0/\mcJ}\bar{E}(f)$. Then by the first inequality in Claim 2 and the remark after Lemma \ref{lemma88}, we have
 \[\begin{aligned}
		\mcT^nE(f)\geq C^{-1}\mcT^n\bar{E}(f)&\geq C^{-1}(1+\varepsilon)^{-n}\tilde\mcT_\mcJ^n\bar{E}(f)
		=C^{-1}(1+\varepsilon)^{-n}\mcT^n\bar{E}(f)\\&\geq C^{-1}(1+\varepsilon)^{-n}\underline{\rho}^n_{V_0/\mcJ}\bar{E}(f)\geq C^{-2}(1+\varepsilon)^{-n}\underline{\rho}^n_{V_0/\mcJ} E(f).
	\end{aligned}\]
\end{proof}

\subsection{Proof of Theorem \ref{thm83} (b)}\label{subsec84}
First, we consider a pair of preserved $\mathscr{G}$-relations $\mcJ,\mcJ'$ such that $\mcJ\subset \mcJ'$ (which can be equal).

\begin{proposition}\label{prop810}There exists $\delta\in(0,1)$,  $C\in (0,1)$ and $\gamma>0$, so that for any pair of  preserved non-trivial $\mathscr{G}$-relations  $\mcJ,\mcJ'$ satisfying $\mcJ\subset \mcJ'$ and $\overline{\rho}_{\mcJ}<\underline{\rho}_{V_0/\mcJ}$, the following holds.
	
	(a). $\overline{\rho}_\mcJ<\underline{\rho}_{V_0/\mcJ'}$.
	
	(b). If $E\in \mathcal{M}$ and $\delta_{\mcJ}(E)\vee \delta_{\mcJ'}(E)<\delta$, then we can find $n\geq 1$ (depending on $E$) such that
	\begin{equation}\label{eqn87}
		\delta_{\mcJ}(\mcT^nE)\vee \delta_{\mcJ'}(\mcT^nE)\geq  \sqrt{\delta_{\mcJ}(E)\vee \delta_{\mcJ'}(E)},
	\end{equation}
	and
	\begin{equation}\label{eqn88}
		\delta_{\mcJ,\mcJ'}(\mcT^nE)\geq C\cdot\big(\delta_{\mcJ}(E)\vee \delta_{\mcJ'}(E)\big)^{-\gamma}\cdot \delta(E).
	\end{equation}
	 In addition,
	\begin{equation}\label{eqn89}
		\delta_{\mcJ,\mcJ'}(\mcT^mE)\geq C\cdot\delta(E),\quad \forall 0\leq m\leq n.
	\end{equation}
\end{proposition}
\begin{proof}
	(a). If suffices to prove that $\underline{\rho}_{V_0/\mcJ}\leq \underline{\rho}_{V_0/\mcJ'}$. Notice that for any $E_{V_0/\mcJ'}\in \mathcal{M}(V_0/\mcJ')$, one can find $E\in \mathcal{M}$ such that $E_{V_0/\mcJ'}=[[E]]_{V_0/\mcJ'}$; also notice that $f'\in l(V_0/\mcJ')$ naturally induces $f\in l(V_0/\mcJ)$ by defining $f(I)=f'(I')$ if $I\subset I'$. Then, it is easy to check
	\[\mcT_{V_0/\mcJ}[[E]]_{V_0/\mcJ}(f)\leq \mcT_{V_0/\mcJ'}[[E]]_{V_0/\mcJ'}(f'),\quad\forall f'\in l(V_0/\mcJ').\]
	Noticing that $[[E]]_{V_0/\mcJ}(f)=[[E]]_{V_0/\mcJ'}(f')$, we have
	\[\inf(\mcT_{V_0/\mcJ}[[E]]_{V_0/\mcJ}\big|[[E]]_{V_0/\mcJ})\leq \inf(\mcT_{V_0/\mcJ'}[[E]]_{V_0/\mcJ'}\big|[[E]]_{V_0/\mcJ'}).\]
By taking the supremum on each side of the above inequality, we obtain $\underline{\rho}_{V_0/\mcJ}\leq \underline{\rho}_{V_0/\mcJ'}$, and hence (a) follows.
	
\medskip

	(b).
By (a), we can choose $\varepsilon>0$ so that $(1+\varepsilon)^2\cdot \frac{\overline{\rho}_\mcJ}{\underline{\rho}_{V_0/\mcJ'}}<1$. By Proposition \ref{prop89}, one can find $\delta_0>0$ and $C_1>0$ independent of $E$, $\mcJ$ and $\mcJ'$, so that if $\delta_{\mcJ}(\mcT^mE)\vee \delta_{\mcJ'}(\mcT^mE)<\delta_0,\forall 0\leq m\leq n$, then we can find $f\in l(V_0)$ such that
	\[f(x_I)=0,\forall I\in V_0/\mcJ,\ \Osc(f)=1\text{, and }\mcT^nE(f)\leq C_1(1+\varepsilon)^n\overline{\rho}^n_\mcJ E(f);\]
	and $g\in l(V_0)$ such that
	\[g(x)=\sum_{I\in V_0/\mcJ'}g(x_I)1_I(x),\forall x\in V_0,\ \Osc(g)=1,\text{ and }\mcT^nE(g)\geq C^{-1}_1(1+\varepsilon)^{-n}\underline{\rho}^n_{V_0/\mcJ'} E(g).\]

We first show that for any $C'\geq1$, there is some $n\geq 1$ such that
\begin{equation}\label{equation8.17}
\delta_\mcJ(\mcT^nE)\vee \delta_{\mcJ'}(\mcT^n E)\geq C'^{-1}\delta_0.
\end{equation}
 We argue by contradiction. Assume $\delta_{\mcJ}(\mcT^nE)\vee\delta_{\mcJ'}(\mcT^nE)<C'^{-1}\delta_0\leq\delta_0$ for all $n\geq1$.
By a direct application of Corollary \ref{coro27} to $\mcT^nE$, we obtain
	\begin{equation*}
\max_{x\mcJ y} R_n(x,y)\gtrsim \frac1{\mcT^nE(f)},\qquad \min_{x\mcJ\mkern-10.5mu\backslash' y} R_n(x,y)\lesssim \frac1{\mcT^nE(g)},
\end{equation*}
where $R_n$ is the $p$-resistance associated with $\mcT^nE$.
Hence one can find $C_2>0$ depending only on $\#V_0$ and $p$ so that
	\begin{equation}\label{eqn810}
	\begin{aligned}
		\delta_{\mcJ,\mcJ'}(\mcT^n E)\geq C_2\frac{\mcT^n E(g)}{\mcT^n E(f)}&\geq C_2C_1^{-2}(1+\varepsilon)^{-2n}\left(\frac{\overline{\rho}_\mcJ}{\underline{\rho}_{V_0/\mcJ'}}\right)^{-n}\frac{E(g)}{E(f)}\\
		&\geq C^2_2C_1^{-2}(1+\varepsilon)^{-2n}\left(\frac{\overline{\rho}_\mcJ}{\underline{\rho}_{V_0/\mcJ'}}\right)^{-n}\delta(E),
	\end{aligned}
	\end{equation}
where we used Corollary \ref{coro27} again to $E$ in the last inequality.
By that $(1+\varepsilon)^2\cdot \frac{\overline{\rho}_\mcJ}{\underline{\rho}_{V_0/\mcJ'}}<1$, we see from \eqref{eqn810} that $\delta_{\mcJ,\mcJ'}(\mcT^n E)\rightarrow\infty$ as $n\rightarrow\infty$.
But clearly, we have \[\delta_{\mcJ,\mcJ'}(\mcT^n E)\leq\delta_{\mcJ}(\mcT^n E)\vee \delta_{\mcJ'}(\mcT^n E),\]
a contradiction.

By Proposition \ref{prop85} (a), we see that there is $C_3>1$ (independent of $E$, $n$, $\mcJ$, $\mcJ'$) such that
\begin{equation}\label{equation8.19}
\delta_\mcJ(\mcT^nE)\vee \delta_{\mcJ'}(\mcT^n E)\leq C_3\delta_\mcJ(\mcT^{n-1}E)\vee \delta_{\mcJ'}(\mcT^{n-1} E)\leq C_3^n\delta_\mcJ(E)\vee \delta_{\mcJ'}(E).
\end{equation}
	
We choose $n$ to be the smallest positive integer such that (\ref{equation8.17}) holds with $C'=C_3$, then by letting $\delta=C_3^{-2}\delta_0^2$ in the proposition, we can find $n$ so that (\ref{eqn87}) holds;
also by \eqref{equation8.19}, we must have
\begin{equation*}
\delta_\mcJ(\mcT^nE)\vee \delta_{\mcJ'}(\mcT^n E)\leq C_3\delta_\mcJ(\mcT^{n-1}E)\vee \delta_{\mcJ'}(\mcT^{n-1} E)< C_3C_3^{-1}\delta_0=\delta_0,
\end{equation*}
then (\ref{eqn89}) is a consequence of (\ref{eqn810}). Finally, we notice that by \eqref{eqn87} and \eqref{equation8.19}, we see that
\begin{equation*}
C_3^n\geq \left(\delta_\mcJ(E)\vee \delta_{\mcJ'}(E)\right)^{-\frac12},
\end{equation*}
which implies for some $C_4>0$,
\begin{equation*}
n\geq-C_4\log(\delta_\mcJ(E)\vee \delta_{\mcJ'}(E)).
\end{equation*}
Substituting this back into \eqref{eqn810}, we obtain (\ref{eqn88}).
\end{proof}

The proof of Theorem \ref{thm83} (b) will be completed by an inductive argument as well.
\begin{definition}\label{def811}
	Let $\mcJ$ be a non-trivial preserved $\mathscr G$-relation.
	
	(a). Define $S(\mcJ)=\max\{m\geq 0:\text{there exist $m$ different non-trivial preserved $\mathscr{G}$-relations}\\ \mcJ_1,\mcJ_2,\cdots,
	\mcJ_m\text{ such that } \mcJ_1\subsetneq \mcJ_2\subsetneq\cdots \subsetneq \mcJ_m\subsetneq \mcJ\}$.
	
	(b). Define $L(\mcJ)=\max\{m\geq 0:\text{there exist $m$ different non-trivial preserved $\mathscr{G}$-relations}\\\mcJ_1,\mcJ_2,\cdots,\mcJ_m\text{ such that }\mcJ_1\supsetneq \mcJ_2\supsetneq\cdots \supsetneq \mcJ_m\supsetneq \mcJ\}$.

	In particular, $\mcJ$ is a minimal non-trivial preserved $\mathscr G$-relation if $S(\mcJ)=0$; $\mcJ$ is a maximal non-trivial preserved $\mathscr G$-relation if $L(\mcJ)=0$.
	
\end{definition}

\begin{proof}[Proof of Theorem \ref{thm83} (b).]
	Let $E_0\in \mathcal{M}$ with $\mathscr{G}$-symmetry, our goal is to prove that \[\inf_{n\geq 0}\delta(\mcT^nE_0)>0.\] It suffices to show that $\min_{\mcJ}\inf_{n\geq 0}\delta_{\mcJ}(\mcT^nE_0)>0$ by Corollary \ref{coro74}.
	
	We do not worry about $\mcJ$ without $\mathscr{G}$-symmetry since $\delta_\mcJ(\mcT^nE_0)\geq 1,\forall n\geq 0$. In addition, by Lemma \ref{lemma34} (b) and Proposition \ref{prop85} (b), we know that there is $0<C_1<1$ depending on $E_0$ such that $\delta_\mcJ(\mcT^nE_0)\geq C_1,\forall n\geq 1$ for each non-preserved $\mathscr{G}$-relation $\mcJ$. Hence, we only need to deal with preserved $\mathscr{G}$-relations.
	
	Let us recall some important facts from Proposition \ref{prop75}. \emph{For convenience, throughout the proof, $E$ always denotes a $p$-energy of the form $\mcT^nE_0,n\geq 1$.} We let $\mathcal{C}=\mathcal{C}(E)$ be the collection of non-trivial preserved $\mathscr{G}$-relations $\mcJ$ such that $\delta_{\mcJ}(E)<C_1$, then we can order them by `$\subsetneq$' by Proposition \ref{prop75}. Moreover, if $\mathcal{C}$ is not empty, then for any pair $\mcJ,\mcJ'\in \mathcal{C}(E)$ such that $\mcJ\subset \mcJ'$ (allow $\mcJ=\mcJ'$), we have by Proposition \ref{prop75},
	\begin{equation}\label{eqn811}
		\delta(E)\geq C_2\cdot \delta_{\mcJ,\mcJ'}(E)\cdot \prod_{\tilde{\mcJ}\in \mathcal{C},\tilde\mcJ\subsetneq \mcJ}\delta_{\tilde\mcJ}(E)\cdot \prod_{\tilde{\mcJ}\in \mathcal{C},\mcJ'\subsetneq \tilde\mcJ}\delta_{\tilde\mcJ}(E),
	\end{equation}
	where $C_2\in(0,1)$ is a constant independent of $E=\mcT^n E_0$ and $\mcJ,\mcJ'$.
	
	The proof will be done by induction, and we will introduce a sequence of claims named Claim $i$, $i\geq 0$, where the proof of Claim $i$ is based on Claim $0,1,\cdots,i-1$ for $i>0$. In the proof, we will introduce a sequence of positive small numbers $1\gg\delta_0\gg\delta_1\gg\delta_2\gg\cdots$, where ``$\gg$" means ``much larger than". We also use ``$\ll$" to mean ``much smaller than".

Let $\gamma>0$ be the same constant as in Proposition \ref{prop810} (b). We choose $\delta_0$ small enough so that Proposition \ref{prop810} holds, and we require $\delta_0^{-\gamma/2}\cdot C_2C_3\geq1$, where $C_3\in(0,1)$ stands for the constant $C$ in Proposition \ref{prop810}. \vspace{0.2cm}

	\noindent\textit{Case $0$. There exists $\mcJ,\mcJ'$ such that $\delta_\mcJ(E)\vee\delta_{\mcJ'}(E)<\delta_0$ and }
	\[\mcJ\subset \mcJ',\quad S(\mcJ)+L(\mcJ')=0.\]
	\noindent\textit{Claim $0$: If Case $0$ happens, then there exist $n_0\geq 1$, $\eta_{0}>0$ and $\gamma_0>0$ such that\\
		1). Case $0$ does not happen for $\mcT^{n_0}E$;\\
		2). $\delta(\mcT^m E)\geq \eta_{0}\delta(E),\forall\ 1\leq m\leq n_0$;\\
		3). $\delta(\mcT^{n_0} E)\geq \delta(E)\cdot(\delta_{\mcJ}(E)\vee \delta_{\mcJ'}(E))^{-\gamma_0}$,
		where $\mcJ$, $\mcJ'$ can be any pair satisfying Case $0$ (unique for $i=0$ case here). }
	
	\begin{proof}[Proof of Claim $0$]

		If Case $0$ happens, we let $\mcJ,\mcJ'$ be a pair as in the statement, and let $n_{0,0}$ be the $n$ as in Proposition \ref{prop810} (b). Then by Proposition \ref{prop810} and using \eqref{eqn811} , we can see the following 0.1) and 0.2) hold:
		
		\noindent0.1). $\delta(\mcT^mE)\geq C_2\delta_{\mcJ,\mcJ'}(\mcT^mE)\geq  C_2C_3\delta(E),\forall\ 0\leq m\leq n_{0,0}$;
		
		\noindent0.2). $\delta(\mcT^{n_{0,0}}E)\geq C_2\delta_{\mcJ,\mcJ'}(\mcT^{n_{0,0}}E)\geq C_2C_3\big(\delta_{\mcJ}(E)\vee \delta_{\mcJ'}(E)\big)^{-\gamma}\delta(E)\geq \big(\delta_{\mcJ}(E)\vee \delta_{\mcJ'}(E)\big)^{{-\gamma/2}}\delta(E)$.
		
		Now if 1) also holds for $\mcT^{n_{0,0}}E$, then we set $n_0=n_{0,0}$, $\eta_0=C_2C_3$ and $\gamma_0=\gamma/2$.
		
		Otherwise, we apply the above process to $E'=\mcT^{n_{0,0}}E$, and define $n'_{0,1}$ in a same manner, and we let $n_{0,1}=n_{0,0}+n'_{0,1}$, with the following 0.3) and 0.4) hold:

\noindent0.3). $\delta(\mcT^mE)\geq (C_2C_3)^2\cdot\delta(E),\forall\ 0\leq m\leq n_{0,1}$;

\noindent0.4). $\delta(\mcT^{n_{0,1}}E)\geq \big(\delta_{\mcJ}(E')\vee \delta_{\mcJ'}(E')\big)^{-\gamma/2}\delta(E')$.

Then if 1) holds for $\mcT^{n_{0,1}}E$, then we set $n_0=n_{0,1}$. If $n_{0,1}$ is still not satisfying, repeat the above process to find $n_{0,2}$ and so on, noticing that by 0.2) and 0.4), we have
\begin{equation}\label{equation8.21}
\delta(E)\leq \delta_0^{\gamma/2}\delta(\mcT^{n_{0,0}}E)\leq \delta_0^{\gamma}\delta(\mcT^{n_{0,1}}E)\leq \delta_0^{3\gamma/2}\delta(\mcT^{n_{0,2}}E)\leq\cdots.
 \end{equation}
Since $\delta(E)>0$ and $\delta_0<1$ and each term $\delta(\mcT^{n_{0,k}}E)\leq 1$, wee see that ``$\cdots$" in \eqref{equation8.21} will finally stop when we find a satisfying $n_{0,k}$, which we denote by $n_0$. Then Claim 0 holds with this $n_0$ and $\eta_0=(C_2C_3)^{k+1}$, $\gamma_0=\gamma/2$.
	\end{proof}
	
	For $i\geq1$, we inductively choose positive $\eta_i\ll\eta_{i-1}$ and $\delta_i\ll\delta_{i-1}\wedge \eta_i$, where the requirements will be made clear soon.\vspace{0.2cm}
	
	\noindent\textit{Case $i$. There exists $\mcJ,\mcJ'$ such that $\delta_\mcJ(E)\vee\delta_{\mcJ'}(E)<\delta_i$ and }
	\[\mcJ\subset \mcJ',\quad S(\mcJ)+L(\mcJ')=i.\]
	\noindent\textit{Claim $i$.  If Case $i$ happens and none of Cases $0,1,\cdots,i-1$ happens, then there exist $n_i\geq 1$ and $\gamma_i>0$ such that\\
		1). None of Cases $0,1,\cdots,i$ happens for $\mcT^{n_i}E$;\\
		2). $\delta(\mcT^m E)\geq \eta_{i}\delta(E)$ for all $1\leq m\leq n_i$;\\
		3). $\delta(\mcT^{n_i} E)\geq \delta(E)\cdot (\delta_{\mcJ}(E)\vee \delta_{\mcJ'}(E))^{-\gamma_i}$.}
	
	\begin{proof}[Proof of Claim i]

		If Case $i$ happens, we let $\mcJ,\mcJ'$ be a pair as in the statement, and let $n'_{i,0}$ be the $n$ as in Proposition \ref{prop810} (b). Hence
\begin{equation}\label{eq8.26}
		\delta_{\mcJ}(\mcT^mE)\vee \delta_{\mcJ'}(\mcT^mE)\lesssim\sqrt{\delta_{i}}\ll \delta_{i-1}\quad\text{ for $0\leq m\leq n'_{i,0}-1$.}
\end{equation}		
Now, if one of Cases $0,1,\cdots,i-1$ happens at $n'_{i,0}$, we let $n_{i,0}>n'_{i,0}$ be the smallest number such that non of the cases $0,1\cdots,i-1$ happens (the existence of $n_{i,0}$ is provided by Claims $0,1\cdots,i-1$ and that $\sup_{n\geq1}\delta(\mcT^{n}E)\leq1$); otherwise, simply let $n_{i,0}=n'_{i,0}$.
		
		In addition, we can see $i.1$), $i.2$) hold:\vspace{0.2cm}
		
		\noindent$i.1$). $\delta(\mcT^mE)\geq \eta_i\delta(E),\forall 0\leq m\leq n_{i,0}$.
		
		We first show $i.1$) with $0\leq m\leq n'_{i,0}$. Indeed, if none of Cases $0,1,\cdots,i-1$ happens for $\mcT^mE$, then by \eqref{eqn811}, and observe that  by \eqref{eq8.26}, $\delta_{\tilde\mcJ}(E)\gtrsim \delta_{i-1}$ for either $\tilde\mcJ\subsetneq \mcJ$ or $\tilde\mcJ\supsetneq \mcJ'$, we see that $\delta(\mcT^mE)\gtrsim\delta_{\mcJ,\mcJ'}(\mcT^mE)$, which implies $i.1$) by using Proposition \ref{prop810} (b).
Otherwise, if some of Cases $0,1,\cdots,i-1$ happen at $m$, let $m'<m$ be the nearest time such that none of Cases $0,1,\cdots,i-1$ happens, whose existence is guaranteed by the assumption of the claim. Then at each $k$ between $m'+1$ and $m$, there must holds some Case $j$ with $0\leq j\leq i-1$. Note that each associated claim gives a ``loop" in which $\delta(\mcT^{k}E)$ has a significant growth at the end of the loop comparing to the beginning. Since $m$ must be located in some loop, we have $\delta(\mcT^mE)\gtrsim \delta(\mcT^{m'+1}E)\gtrsim \delta(\mcT^{m'}E)$. Thus $\delta(\mcT^mE)\gtrsim \delta(\mcT^{m'}E)\geq\eta_i\delta(E)$ as previous.

By a similar argument of comparing $\delta(\mcT^{n_{i,0}}E)$ with $\delta(\mcT^{n'_{i,0}}E)$ as above, we see that $i.1$) also holds for $0\leq m\leq n_{i,0}$ with a slightly adjustment of the constant $\eta_i$.

\vspace{0.2cm}

		\noindent$i.2$). $\delta(\mcT^{n_{i,0}}E)\geq \big(\delta_{\mcJ}(E)\vee \delta_{\mcJ'}(E)\big)^{-\gamma_i/2}\delta(E)$.
		
To see $i.2$), we only need to check the case $n_{i,0}> n'_{i,0}$ since otherwise $i.2)$ is immediate due to Proposition \ref{prop810} (b). We need to check the following two possibilities.
		
		(Possibility 1). For any $\tilde\mcJ$ such that $\tilde{\mcJ}\subsetneq \mcJ$ or $\tilde{\mcJ}\supsetneq \mcJ'$,
		\[\delta_{\tilde{\mcJ}}(\mcT^{n_{i,0}}E)\geq \big(\delta_{\mcJ}(E)\vee \delta_{\mcJ'}(E)\big)^{\frac{\gamma}{4i}}.\]
		In this possibility, we can see by (\ref{eqn811}) and Proposition \ref{prop810} (b),
		\[\begin{aligned}
			\delta(\mcT^{n_{i,0}}E) &\geq C_2\big(\delta_{\mcJ}(E)\vee \delta_{\mcJ'}(E)\big)^{\frac{\gamma}{4i}\cdot i}\delta_{\mcJ,\mcJ'}(\mcT^{n_{i,0}}E)\\
			&\geq \big(\delta_{\mcJ}(E)\vee \delta_{\mcJ'}(E)\big)^{-\gamma/4}\delta(E).
		\end{aligned}\]

		(Possibility 2). There exists some $\tilde\mcJ$ such that $\tilde{\mcJ}\subsetneq \mcJ$ or $\tilde{\mcJ}\supsetneq \mcJ'$,
		\[\delta_{\tilde{\mcJ}}(\mcT^{n_{i,0}}E)< \big(\delta_{\mcJ}(E)\vee \delta_{\mcJ'}(E)\big)^{\frac{\gamma}{4i}}.\]
In this possibility, we apply 2) and Case $j$ for some $0\leq j<i$,
		\[
		\begin{aligned}
			\delta(\mcT^{n_{i,0}}E)\gtrsim \delta(\mcT^{n'_{i,0}}E)\gtrsim \big(\delta_{\mcJ}(E)\vee \delta_{\mcJ'}(E)\big)^{-\gamma}\delta(E).
		\end{aligned}
		\]
		Hence, by choosing some $\delta_i,\gamma_i$ small enough, we can see $i.2$) happens for both possibilities.\vspace{0.1cm}
		
Finally, similar to Claim 0, we let $n_i=n_{i,0}$ if none of Cases $0,1,\cdots, i$ holds at $n_{i,0}$; otherwise repeat the procedure, noticing that none of Cases $0,1,\cdots,i-1$ holds for $\mcT^{n_{i,0}}E$, the procedure will finally stop since $\delta(E)\ll\delta(\mcT^{n_{i,0}}E)\ll\delta(\mcT^{n_{i,1}}E)\ll\cdots$. Therefore, we can find $n_i$ such that 1) holds, while 2) is from $i.1$) and 3) is from $i.2$).
	\end{proof}

	We can finalize the proof with Claims $0$-$M$, where $M+1$ is the maximum length of chains of non-trivial preserved $\mathscr{G}$-relations (hence there is no Case $M+1$, and escaping Case $0$-$M$ means escaping all possible cases): if none of Cases $0$-$M$ holds for $\mcT^nE_0$, then $\delta(\mcT^nE_0)$ is uniformly bounded from below; otherwise, we let $n'<n$ be the last time that none of the cases holds (by choosing $\delta_i$ small enough, we can assume none of the cases holds for $E_0$, and hence $n'$ exists), and apply the corresponding Claim $j$ 2) for each loop between $n'$ and $n$, we see that $\delta(\mcT^nE_0)\gtrsim\delta(\mcT^{n'}E_0)$, and hence $\delta(\mcT^nE_0)$ has a uniform positive lower bounded as well.
\end{proof}

\appendix
\renewcommand{\appendixname}{Appendix~\Alph{section}}
\section{Properties of $\mathcal{Q}_p(A)$}\label{AppendixA}
In this appendix, we consider some useful properties concerning $\Ep\in \mathcal{Q}_p(A)$, where $A$ is a finite set. We begin with a lemma concerning the norm $\|\cdot\|_{\widetilde{\mathcal{M}}_p(A)}$.

\begin{lemma}\label{lemmaa1}
	Let $\Ep_n\in \widetilde{\mathcal{M}}_p(A),n\geq 1$ and $\Ep\in \widetilde{\mathcal{M}}_p(A)$, and assume $\|\Ep_n-\Ep\|_{\widetilde{\mathcal{M}}_p(A)}\to 0$ as $n\to\infty$. Then, for any $f_n\in l(A),n\geq 1$ and $f\in l(A)$ such that $f_n\to f$ pointwisely, we have
	\[\lim\limits_{n\to\infty}\Ep_n(f_n)=\Ep(f).\]
\end{lemma}
\begin{proof}
	First, from the definition of $\|\cdot\|_{\widetilde{M}_p(A)}$, we have
	\begin{equation}\label{equation2.7}
		|E^{(p)}_n(f)-E^{(p)}(f)|\leq \text{Osc}(f,A)^p\|E^{(p)}_n-E^{(p)}\|_{\widetilde{M}_p(A)},
	\end{equation}
	where $\text{Osc}(f,A)=\max\{|f(x)-f(y)|:\ x,y\in A\}$. Then by the condition that $\lim_{n\rightarrow\infty}\|\Ep_n-\Ep\|_{\widetilde{\mathcal{M}}_p(A)}=0$, we have
	\begin{equation}\label{equation2.8}
		\lim_{n\rightarrow\infty}\big|E^{(p)}_n(f)-E^{(p)}(f)\big|=0.
	\end{equation}
	
	We then estimate $\big|E^{(p)}_n(f_n)-E^{(p)}_n(f)\big|$. Note that by \eqref{equation2.7}, we see that
	$\big|E^{(p)}_n(f)-E^{(p)}(f)\big|$ is uniformly bounded for all $n\geq1$ and for all $f$ with $\text{Osc}(f,A)=1$, and from this, there is $C_0>0$ such that
	\begin{equation}\label{equation2.9}
		E^{(p)}_n(f)\leq C_0,\qquad\forall n\geq1, \forall f \text{ with } \text{Osc}(f,A)=1.
	\end{equation}
	Since $f_n\rightarrow f$ as $n\rightarrow\infty$, we can write both $f_n=(1-\varepsilon_n) f+\varepsilon_n g_n$ and $f=(1-\varepsilon_n) f_n+\varepsilon_n g'_n$, where $\varepsilon_n\in(0,1)$ satisfying $\lim_{n\rightarrow\infty}\varepsilon_n=0$ and $g_n,g_n'$ are uniformly bounded in $n$. Now using the convexity of $E^{(p)}_n$, we see that
	\begin{align*}
		E^{(p)}_n(f_n)&\leq (1-\varepsilon_n)E^{(p)}_n(f)+\varepsilon_nE^{(p)}_n(g_n),\\
		E^{(p)}_n(f)&\leq (1-\varepsilon_n)E^{(p)}_n(f_n)+\varepsilon_nE^{(p)}_n(g'_n),
	\end{align*}
	which gives that
	\begin{equation}\label{equation2.10}
		\varepsilon_nE_n^{(p)}(f_n)-\varepsilon_nE_n^{(p)}(g'_n)\leq E^{(p)}_n(f_n)-E^{(p)}_n(f)\leq -\varepsilon_nE_n^{(p)}(f)+\varepsilon_nE_n^{(p)}(g_n).
	\end{equation}
	Since $f_n,g_n,g_n'$ are all uniformly bounded, by \eqref{equation2.9}, we see that there is $C'_0>0$ such that $E_n^{(p)}(f)$, $E_n^{(p)}(f_n)$, $E_n^{(p)}(g_n)$, $E_n^{(p)}(g_n')$ are all bounded by $C'_0$. Thus using \eqref{equation2.10}, we obtain
	\begin{equation}\label{equation2.11}
		\big|E^{(p)}_n(f_n)-E^{(p)}_n(f)\big|\leq 2\varepsilon_nC_0',
	\end{equation}
	which goes to $0$ as $n\rightarrow\infty$.
	
	Above all, \eqref{equation2.8} together with \eqref{equation2.11} implies that
	$
	\lim\limits_{n\to\infty}\Ep_n(f_n)=\Ep(f)
	$
	as desired.
\end{proof}

We will need to compare the $p$-harmonic (minimal energy) extensions of different functions in Section \ref{sec5}. Noticing that by Remark 2 after Definition \ref{def28}, $\Ep\in \mathcal{Q}_p(B)$ is always strictly convex, so the $p$-harmonic extension of a function is always unique.
\begin{lemma}\label{lemmaa2}
	Let $\Ep\in \mathcal{Q}_p(B)$ and assume $A\subset B$. Let $f,g\in l(B)$ and assume $\Ep(f)={[\Ep]_A}(f|_A)$, $\Ep(g)={[\Ep]_A}(g|_A)$. Then
	\[f(x)\leq g(x),\forall x\in A\Longrightarrow f(x)\leq g(x),\forall x\in B.\]
\end{lemma}
\begin{proof}
	First, we assume $\Ep\in \mathcal{S}_p(B)$, so there are $c_{x,y}\geq 0$ depending on $x,y$ such that $\Ep(u)=\sum_{x,y\in B}c_{x,y}|u(x)-u(y)|^p,\forall u\in l(B)$. In this situation, the conditions $\Ep(f)={[\Ep]_A}(f|_A)$, $\Ep(g)={[\Ep]_A}(g|_A)$ imply that
	\begin{equation}\label{eqn27}
		\begin{cases}
			\sum_{y\neq x}c_{x,y}\cdot p\big|f(x)-f(y)\big|^{p-1}\cdot \text{sgn}\big(f(x)-f(y)\big)=0,\quad\forall x\in B\setminus A,\\
			\sum_{y\neq x}c_{x,y}\cdot p\big|g(x)-g(y)\big|^{p-1}\cdot \text{sgn}\big(g(x)-g(y)\big)=0,\quad\forall x\in B\setminus A.
		\end{cases}
	\end{equation}
	Now we show $\max_{x\in B}\big(f(x)-g(x)\big)=\max_{x\in A}\big(f(x)-g(x)\big)$, then the lemma (for this special case) follows. In fact, assume the equality does not hold, then we can find $x_0\in B\setminus A$, so that $f(x_0)-g(x_0)=\max_{x\in B}\big(f(x)-g(x)\big)>\max_{x\in A}\big(f(x)-g(x)\big)$. Then, we have $f(x_0)-f(x)\geq g(x_0)-g(x)$ for any $x\in B$, so that
	\[\big|f(x_0)-f(x)\big|^{p-1}\cdot \text{sgn}\big(f(x_0)-f(x)\big)\geq \big|g(x_0)-g(x)\big|^{p-1}\cdot \text{sgn}\big(g(x_0)-g(x)\big).\]
	Then, by (\ref{eqn27}) (taking $x=x_0$ there), we know that the above inequality must be equality, so we get $f(x_0)-f(x)=g(x_0)-g(x)$, or equivalently $f(x)-g(x)=f(x_0)-g(x_0)$, for any $x\in B$ such that $c_{x,x_0}>0$. We can repeat the above argument to see there is some $x\in A$ such that $f(x)-g(x)=f(x_0)-g(x_0)$, since $B$ is finite. A contradiction.
	
	Next, assume $\Ep\in \mathcal{Q}_p'(B)$, so we can find $V\supset B$ and $\Ep_V\in \mathcal{S}_p(V)$ such that $[\Ep_V]_B=\Ep$. Noticing that $\Ep_V$ is strictly convex, we can extend $f,g$ uniquely to $f_V,g_V\in l(V)$ so that $f_V|_B=f,g_V|_B=g$ and
	\[\begin{cases}
		\Ep_V(f_V)=[\Ep_V]_B(f)=\Ep(f)=[\Ep]_A(f|_A),\\ \Ep_V(g_V)=[\Ep_V]_B(g)=\Ep(g)=[\Ep]_A(g|_A).
	\end{cases}\]
	By the previous paragraph, we then get $f(x)\leq g(x),\forall x\in A$ implies $f(x)\leq g(x),\forall x\in V$ where $V\supset B$.
	
	Finally, let $\Ep\in \mathcal{Q}_p(B)$, so we can find $\Ep_n\in \mathcal{Q}_p'(B)$ such that $\|\Ep_n-\Ep\|_{\widetilde{\mathcal{M}}_p(B)}\to 0$ as $n\to \infty$. Define $f_n,g_n\in l(B)$ as the unique functions such that $f_n|_A=f|_A$, $\Ep_n(f_n)={[\Ep_n]_A}(f|_A)$ and $g_n|_A=g|_A$, $\Ep_n(g_n)={[\Ep_n]_A}(g|_A)$, noticing that $\Ep_n$ are strictly convex. In addition, $f_n\leq g_n$ by the previous paragraph. Then, by passing to a subsequence, we have $f_{n_l}\to f'$ and $g_{n_l}\to g'$ for some $f',g'\in l(B)$. We claim that $f'=f$ and $g'=g$ so the proof is completed. In fact, we have $f'|_A=f|_A$ and by Lemma \ref{lemmaa1},
	\[\begin{aligned}
		\Ep(f')=\lim_{l\to\infty}\Ep_{n_l}(f_{n_l})\leq \lim_{l\to\infty}\Ep_{n_l}(f)=\Ep(f),\\
	\end{aligned}\]
	and the same argument works for $g$.
\end{proof}

 The following Lemma \ref{lemmaa3} will play a key role in the proof of $\mathscr{H}\subset C(K)$ in Section \ref{sec5}.

\begin{lemma}\label{lemmaa3}
Let $\Ep\in \mathcal{Q}_p(A)$ and $f\in l(A)$.  Define
\[\partial_{X,-}\Ep(f)=\lim\limits_{t\nearrow 0}\frac{\Ep(f+t\cdot 1_X)-\Ep(f)}{t},\]
where $X\subset A$ and $1_X\in l(A)$ is the indicator function of $X$.

Let $M_f=\big\{x\in A:f(x)=\max_{y\in A}f(y)\big\}$. Then  $\partial_{X,-}\Ep(f)\geq 0$ if $X\subset M_f$. In addition, if $M_f\neq A$, then $\partial_{M_f,-}\Ep(f)>0$.
\end{lemma}
\begin{proof}
	Define $u_X(t)=\Ep(f+t\cdot 1_X)$, then $u_X$ is a convex function with non-negative real values, so $u_X$ has left derivatives $\frac{d}{dt-}u_X(t)$ everywhere. Hence $\partial_{X,-}\Ep(f)$ is well defined.
	
	We are only interested in the case that $X\neq A$, since otherwise the lemma is trivial. In this case, since $\Ep\in \mathcal{Q}_p(A)$ is strictly convex, $u_X$ is also strictly convex. Hence there is a unique $s$ such that
	\[u_X(s)=\min\{u_X(t):t\in \mathbb{R}\}.\]
	Clearly, if $X\subset M_f$, then $s\leq 0$ by the Markov property. In addition, by the strict convexity, we know that $\frac{d}{dt-}u_X(t)>0$ if $t>s$. In particular, this holds when $X=M_f\neq A$. To deal with the case $t=s$, we introduce the following construction. \vspace{0.2cm}
	
	\noindent\textit{Construction.} Let $A_*=(A/X)\cup \{o\}$, where $A/X:\ ={(A\setminus X)}\cup\{X\}$ (we view $X$ as a single point in $A_*$) and $o$ is a new point outside $A$. For each $g_*\in l(A_*)$, we can define $g\in l(A)$ by
	\[g(x)=\begin{cases}
		g_*(x),\text{ if }x\in A\setminus X,\\
		g_*(X), \text{ if }x\in X.
	\end{cases}\]
    Then, we define
	\[\Ep_*(g_*)=\Ep(g)+\big|g_*(X)-g_*(o)\big|^p,\quad\forall g_*\in l(A_*).\]
	It is easy to see that $\Ep_*\in \mathcal{Q}_p(A_*)$.  \vspace{0.2cm}
	
	Returning to our basic settings, one can see that $f|_X\equiv c$, where $c=\max_{y\in A}f(y)$. So, for each $\delta>0$, one can define $f_{*,\delta}\in l(A_*)$ by
	\[f_{*,\delta}|_{A\setminus X}=f|_{A\setminus X},\quad f_{*,\delta}(X)=c,\quad\text{and}\quad f_{*,\delta}(o)=c-\delta.\]

	Define $u_{X,*,\delta}(t)=\Ep_*(f_{*,\delta}+t\cdot 1_{X})$. Let $s_\delta$ be the unique value such that $u_{X,*,\delta}(s_\delta)=\min\{u_{X,*,\delta}(t):t\in \mathbb{R}\}$. Now we claim that $s_\delta< 0$. Indeed, let $g\in l(A_*)$ be such that $g|_{A\setminus X}=c$, $g(o)=c-\delta$, and $g$ be $p$-harmonic at $X$ w.r.t. $\Ep_*$. It is clear that $g(X)<c$ (by a same argument as the proof of Lemma \ref{lemma54}). Now we compare $f_{*,\delta}+s_\delta\cdot 1_{X}$ with $g$. By Lemma \ref{lemmaa2}, we see that \[c+s_\delta=f_{*,\delta}(X)+s_\delta\cdot 1_{X}(X)\leq g(X)<c.\]
	
	Finally, by the same argument as before, we see
	\[0<\frac{d}{dt-}u_{X,*,\delta}(0)=\partial_{X,-}\Ep_*(f_{*,\delta})=\partial_{X,-}\Ep(f)+p\delta^{p-1}.\]
	Since $\delta$ can be arbitrarily small, we have $\partial_{X,-}\Ep(f)\geq 0$ as desired.
\end{proof}

\section*{Acknowledgments}
We received the sad news that Prof. Ka-Sing Lau and Prof. Robert S. Strichartz passed away, when we were finishing the paper. We are grateful for their long time support, and would like to dedicate this paper to the memory of them.

\bibliographystyle{amsplain}
	
\end{document}